\documentclass{article}
\usepackage[utf8]{inputenc}
\usepackage{lmodern}
\usepackage{textgreek}

\usepackage[backend=biber,style=alphabetic]{biblatex}
\renewbibmacro{in:}{}
\usepackage{amsmath}
\usepackage{amsthm}
\usepackage{amssymb}
\usepackage{mathtools}
\usepackage{amsfonts}
\usepackage{wasysym}
\usepackage{bbm}
\usepackage{thmtools} 
\usepackage{faktor}
\usepackage[hidelinks]{hyperref}
\usepackage{float}

\usepackage{xargs}

\usepackage{enumitem}

\usepackage[capitalize, nameinlink]{cleveref}

\usepackage{graphicx}
\usepackage{xcolor}

\usepackage{pdfpages}

\definecolor{Bnavy}{RGB}{0, 66, 128}
\definecolor{Bdust}{RGB}{140,179,217}
\definecolor{Bsugarpaper}{RGB}{198, 217, 236}
\definecolor{Bgreen}{RGB}{142, 183, 114}
\definecolor{Blimegreen}{RGB}{202, 222, 189}
\definecolor{Bgreentheme}{RGB}{36, 87, 1}

\usepackage{makecell}

\usepackage{microtype}

\usepackage{tikz} 
\usepackage{tikz-cd}
\usetikzlibrary{ 
	calc,%
    math,
	arrows,%
	shapes,
	shapes.geometric,
    positioning,
    fit,
} 

\tikzcdset{arrow style=tikz, diagrams={>=stealth}}
\usepackage[labelsep=quad,indention=10pt]{subfig}
\theoremstyle{plain}
\newtheorem{theorem}{Theorem}[section]

\newtheorem{innercustomthm}[theorem]{Main Result}
\newenvironment{ctheorem}[1]
  {\renewcommand\theinnercustomthm{#1}\innercustomthm}
  {\endinnercustomthm}

\newtheorem{lemma}[theorem]{Lemma}
\newtheorem{corollary}[theorem]{Corollary}
\newtheorem{proposition}[theorem]{Proposition}
\newtheorem{construction}[theorem]{Construction}
\theoremstyle{definition}
\newtheorem{definition}[theorem]{Definition}
\newtheorem{example}[theorem]{Example}
\newtheorem{recollection}[theorem]{Recollection}
\newtheorem{remark}[theorem]{Remark}
\newtheorem{notation}[theorem]{Notation}

\newcounter{diagram}  
\crefname{diagram}{Diagram}{Diagrams}
\creflabelformat{diagram}{#2\textup{(#1)}#3}
\newenvironment{diagram}[1][]{
    \crefalias{equation}{diagram}
    \begin{equation}
    \begin{tikzcd}[#1]
}{
    \end{tikzcd}
    \end{equation}
}

\newenvironment{acknowledgements}{%
  \begin{abstract}
}{%
  \end{abstract}
}

\usepackage{notation}

\usepackage{authblk}

\title{From Samples \\ to \\ Persistent Stratified Homotopy Types}

\author[1]{Tim Mäder\thanks{tmaeder@mathi.uni-heidelberg.de}}
\author[1]{Lukas Waas \thanks{lwaas@mathi.uni-heidelberg.de}}
\affil[1]{Institut für Mathematik, Universität Heidelberg}

\begin{document}

\maketitle

\begin{abstract}
    The natural occurrence of singular spaces in applications has led to recent investigations on performing topological data analysis (TDA) in a stratified framework. In many applications, there is no a priori information on what points should be regarded as singular or regular. For this purpose we describe a fully implementable process that provably approximates the stratification for a large class of two-strata Whitney stratified spaces from sufficiently close non-stratified samples. \\
    Additionally, in this work, we establish a notion of persistent stratified homotopy type obtained from a sample with two strata. In analogy to the non-stratified applications in TDA which rely on a series of convenient properties of (persistent) homotopy types of sufficiently regular spaces, we show that our persistent stratified homotopy type behaves much like its non-stratified counterpart and exhibits many properties (such as stability, and inference results) necessary for an application in TDA.\\
    In total, our results combine to a sampling theorem guaranteeing the (approximate) inference of (persistent) stratified homotopy types of sufficiently regular two-strata Whitney stratified spaces.  \\
\end{abstract}

\begin{acknowledgements}
The first author's work is supported by Deutsche Forschungsgemeinschaft (DFG, German Research Foundation) under Germany’s Excellence Strategy EXC-2181/1 - 390900948 (the Heidelberg STRUCTURES Excellence Cluster).
The second author is supported by a PhD-stipend of the Landesgraduiertenförderung Baden-Württemberg. The authors would like to thank an anonymous referee for their detailed remarks and suggestions.
\end{acknowledgements}

\section{Introduction}\label{sec1}

Topological data analysis has proven itself to be a source of qualitative and quantitative data features that were not readily accessible by other means. Arguably, the most important concept for the development of this field is persistent homology (\cite{edelsbrunner2002toppers,zomorodian2005computing,cohen2007stability,ghrist2008barcode,niyogi2008finding,carlsson2009topanddata,oudot2015quiver}). Both in practice, as well as abstractly speaking, persistent homology usually is divided up into a two-step process. 
First, one assigns to a data set $\sampX$ a filtration of topological or combinatorial objects $(\sampX_{\alpha})_{\alpha \geq 0}$. Most prominently, this is done for $\sampX \subset \supSp$, by taking $\sampX_{\alpha}$ to be an $\alpha$-thickening of $\sampX$, which is the case we will consider in the following.
Then, from this filtered object, a persistence module is computed, essentially given by computing homology in each filtration degree while keeping track of the functoriality of homology on the inclusions. As homology is a homotopy invariant what is relevant to this computation is only what one may call the \define{persistent homotopy type of $\sampX$}. More precisely, if we think of $(\sampX_{\alpha})_{\alpha \geq 0}$ as a functor from the non-negative reals $\mathbb R_+$ into the category of topological spaces $\Top$ then the persistent homotopy type is the isomorphism class of $(\sampX_{\alpha})_{\alpha \geq 0}$ in the homotopy category  $\ho \Top^{\mathbb R_+}$ obtained by inverting pointwise (weak) homotopy equivalences. From this perspective, persistent homology is the composition
\begin{equation}\label{equ:factorization_of_ph}
    \mathrm{PH}_{i}\pp \Sam \xrightarrow{\mathcal{P}}  \ho \Top^{\mathbb R_+} \xrightarrow{\mathrm{H}_i^{\mathbb R_+}} \Vect^{\mathbb R_+}.
\end{equation}
Here $\Sam$ is the category of subspaces of some fixed $\supSp$, $\mathcal{P}$ assigns an object in the persistent homotopy category (for example, through thickening spaces or possibly using combinatorial models thereof) and $\mathrm{H}_i^{\mathbb R_+}$ computes homology degree-wise. The composition produces an object in the category of persistence modules over some field $k$, denoted $\Vect^{\mathbb R_+}$. Many of the advantages of persistent homology turn out to not be properties of the right-hand side of this composition but of the left-hand side $\mathcal{P}$. That is, they are properties of the persistent homotopy type. 
Such properties include, for example:
\begin{enumerate}[label= {P(\arabic*):}, ref={P(\arabic*)}]
    \item \label{enum:properties_of_PH1} The fact that persistent homology defined through thickenings is computable at all; (This is a consequence of the nerve theorem (see e.g. \cite[Prop. 4G.3]{hatcher2002algtop} or \cite{borsuk1948nerve}), which states that for $\sampX \subset \supSp$ the persistent stratified homotopy type $\mathcal{P}(\sampX)$ may equivalently be represented by a filtered \v{C}ech complex.)
    \item \label{enum:properties_of_PH2} The stability of persistent homology with respect to Hausdorff and interleaving type distances (see \cite{cohen2007stability,chazal2008interleave,bauer2015isometry});
    \item \label{enum:properties_of_PH3} The possibility to infer information from the sampling source by using persistent homology. (This is usually justified by stability together with the result that $\mathcal{P}(\topSp)_{\alpha} \xleftarrow{\simeq} \mathcal{P}(\topSp)_{0} = \topSp$ for $\alpha$ sufficiently small and $\topSp$ a sufficiently regular space such as a compact smooth submanifold of Euclidean space (compare to \cite{niyogi2008finding}).)
\end{enumerate}
At the same time, many of the limitations of persistent homology also stem from the factorization in \labelcref{equ:factorization_of_ph}. Consider, for example, the two subspaces of $\mathbb R^2$ depicted in \cref{fig:lemniscate,fig:ball}. It is not hard to see that (up to a rescaling) they have the same persistent homotopy type and thus have the same persistent homology.
\begin{figure}[H]
\centering
    \begin{minipage}{0.49\textwidth}
        \centering
        \includegraphics[width=1.0\textwidth]{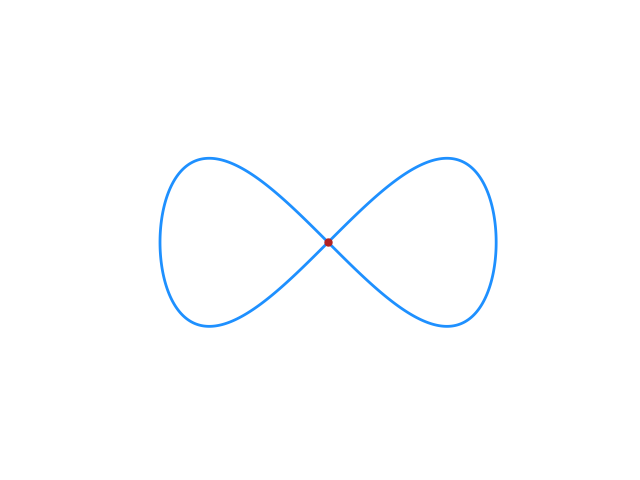}
        \caption{Lemniscate $V=\{x \in \mathbb R^2 \mid x_1^4 - x_1^2 + x_2^2 = 0\}$}
        \label{fig:lemniscate}
    \end{minipage}\hfill
    \begin{minipage}{0.49\textwidth}
        \centering
        \includegraphics[width=1\textwidth]{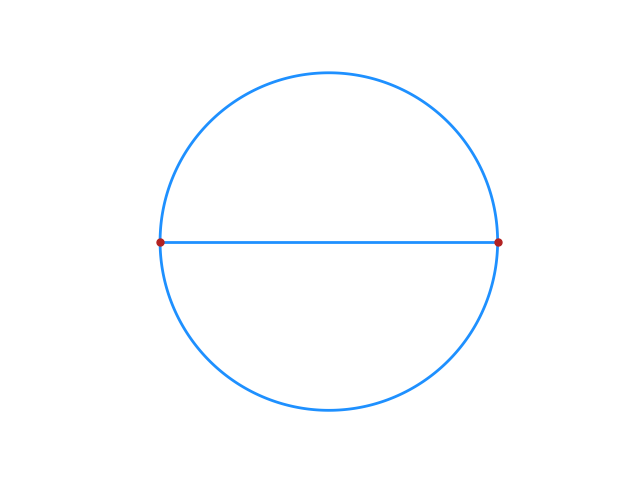}
        \caption{Circle with a diameter filament}
        \label{fig:ball}
    \end{minipage}
\end{figure}
Of course, the spaces themselves are topologically quite different, the space shown in \cref{fig:lemniscate} having one and the space shown in \cref{fig:ball} having two singularities. Depending on the application, one may be interested in an invariant capable of distinguishing the two.
For example, one may consider the two spaces in \cref{fig:lemniscate,fig:ball} as so-called stratified spaces, taking care to mark their singularities.\\
The topological data analysis of stratified objects has recently received increased interest (see, for example, \cite{mileyko2021another,stolz2020geometric,nanda2020local,skraba2014approximating,bendich2012local,fasy2016exploring}). However, as suggested by the properties of the non-stratified scenario described in \ref{enum:properties_of_PH1} to \ref{enum:properties_of_PH3}, to successfully establish persistent methods in a stratified framework a notion of persistent stratified homotopy type is needed. No such thing was available so far, at least not to our knowledge and not in a way that satisfies analogs to the properties \ref{enum:properties_of_PH1} to \ref{enum:properties_of_PH3}. This is because stratified homotopy theory has only recently received a wave of renewed attention from the theoretical perspective \cite{woolf2009fundamental,lurie2012higher, miller2013strongly, haine2018homotopy, douteau2021homotopy, douteau2021stratified,douteauwaas2021}. A series of new results in this field now lay the foundation for stratified investigations in topological data analysis.\\
Establishing such a notion of persistent stratified homotopy type and showing that it fulfills properties much like the non-stratified persistent homotopy type is precisely what this work is concerned with. Thus, the focus lies entirely on the left-hand side of the factorization in ~\eqref{equ:factorization_of_ph}, leaving investigating algebraic invariants of the latter (for example, intersection homology, as in \cite{bendich2011persistent}) for future work. Note, however, that whatever invariants they may be, they automatically inherit many of the convenient properties of persistent homology. 
\subsection{Persistent stratified homotopy types}
Let us illustrate our methods and results by following the example of the lemniscate $V$ shown in \cref{fig:lemniscate}. We may treat the lemniscate as a so-called Whitney stratified space (see \cref{recol:Whitney_strat}) $W$, with two strata given by $W_p = \{0\}$, the singularity, and $W_q = V \setminus W_p$ given by the regular part. It follows from results in \cite{douteau2019stratified,douteauwaas2021} (see \cref{thrm:fully_faithful_embedding,recol:diag_are_equ}) that, for a Whitney stratified space with two strata, the so-called \define{stratified homotopy type} (the analogue of the classical homotopy type, obtained by considering a stratum preserving notion of map and homotopy) may equivalently be thought of as (the homotopy type of) a diagram of spaces of the form
\begin{diagram}\label{diag:intro_abstract}
    D(W)_p & \arrow[l] D(W)_{\{p,q\}} \arrow[r] & D(W)_q.
\end{diagram}
Here, the spaces $D(W)_p$ and $D(W)_q$ correspond respectively to the strata of the stratified space $W$, while $ D(W)_{\{p,q\}}$ corresponds to the homotopy type of the space connecting the two strata, the so-called \define{(homotopy)link} (see \cref{def:holink}).
More explicitly, it follows from \cref{prop:ALT_holinks_are_links} that for sufficiently small positive real numbers $\diagParamLow < \diagParamUp$ the stratified homotopy type of $W$ is encoded in the diagram
\begin{diagram}\label{diag:intro_sublevel}
      \textnormal{d}_{W_p}^{-1}[0, \diagParamUp] & \arrow[l, hook'] \textnormal{d}_{W_p}^{-1}[\diagParamLow, \diagParamUp] \arrow[r,hook] &  \textnormal{d}_{W_p}^{-1}[\diagParamLow, \infty],
\end{diagram}
where $\textnormal{d}_{W_p}$ is the function assigning to a point its distance to the singular stratum $W_p$. 
\begin{example}\label{ex:intro_sublvl}
For the lemniscate, as shown in \cref{fig:lemniscate}, the \cref{diag:intro_sublevel} with parameter $\diagParam = (0.2,0.3)$ can be visualized as follows:
\begin{center}
     \begin{tikzpicture}[scale = 0.815]
        \node(sing) at (-5,0){\includegraphics[width=3.5cm]{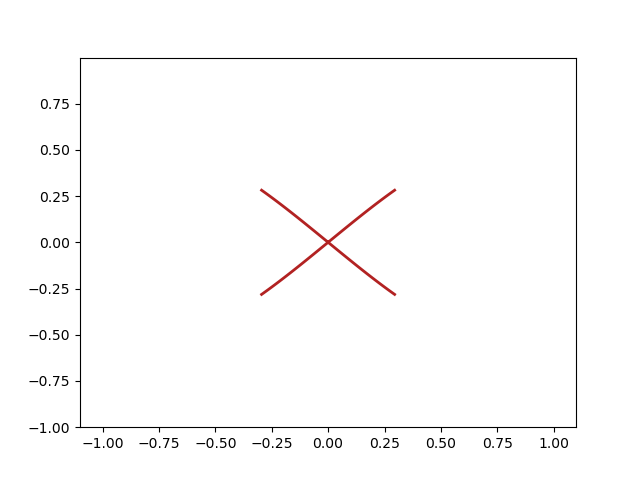}}; 
          \node(link) at (0,0){\includegraphics[width=3.5cm]{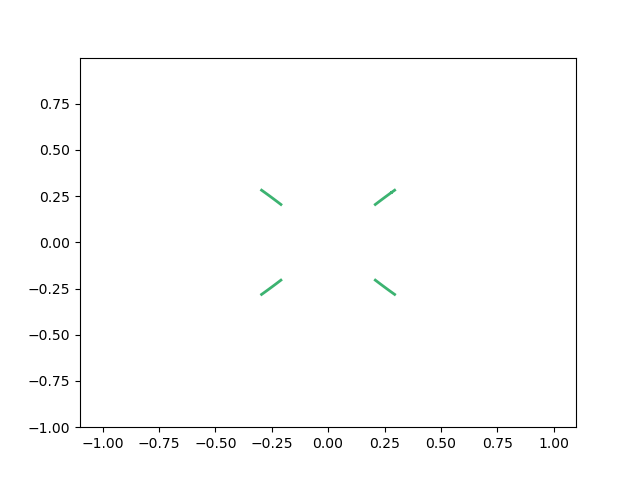}};
                  \draw[ left hook->] (link) -- (sing);
         \node(reg) at (5,0){\includegraphics[width=3.5cm]{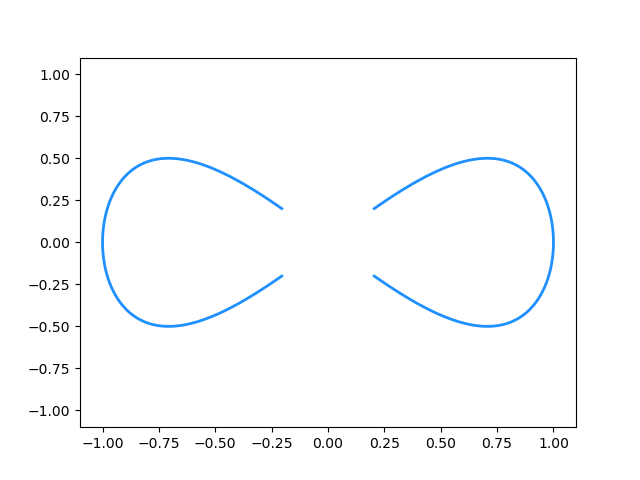}};
          \draw[ right hook->] (link) -- (reg);
    \end{tikzpicture}
\end{center}
    This diagram is (weakly) equivalent to the simpler diagram of discrete spaces
\begin{diagram}\label{intro:diag_disc_lem}
    \{c\} & \arrow[l] \{a,b\} \times \{x,y\}  \arrow[r, "\pi_2"] & \{x,y\}.
\end{diagram}
In the case of the space shown in \cref{fig:ball}, with the singular stratum given by the singularities, \cref{diag:intro_sublevel} is given by:
\begin{center}
     \begin{tikzpicture}[scale = 0.815]
        \node(sing) at (-5,0){\includegraphics[width=3.5cm]{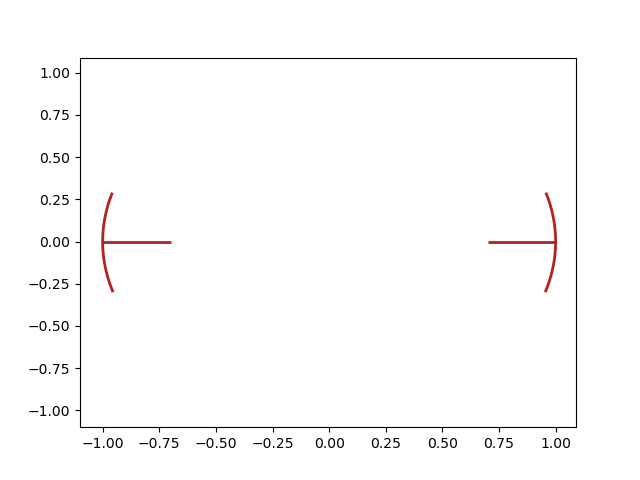}}; 
          \node(link) at (0,0){\includegraphics[width=3.5cm]{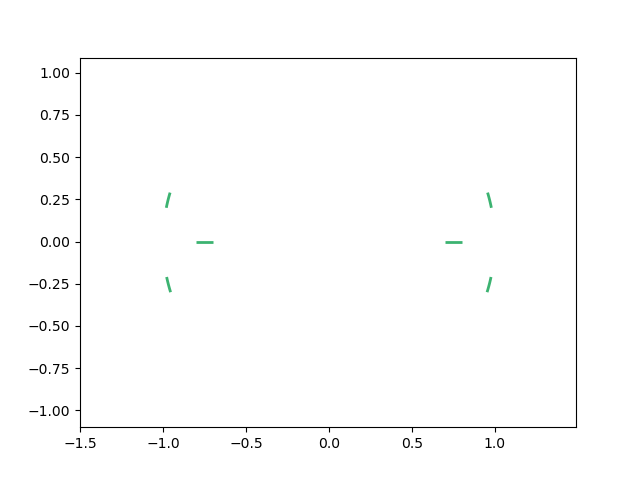}};
                  \draw[ left hook->] (link) -- (sing);
         \node(reg) at (5,0){\includegraphics[width=3.5cm]{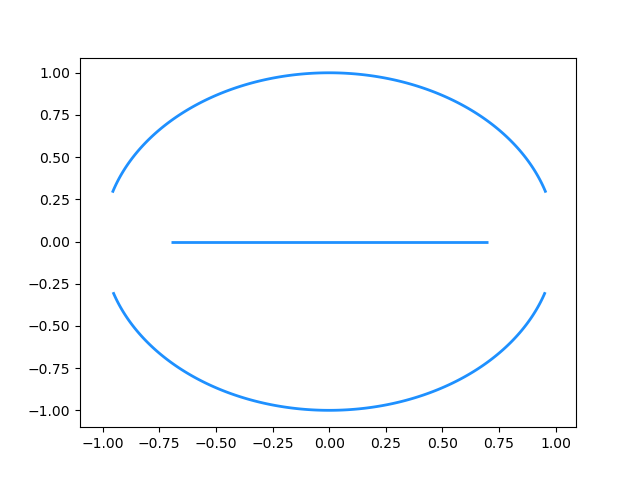}};
          \draw[ right hook->] (link) -- (reg);
    \end{tikzpicture}
\end{center}
From the perspective of homotopy theory, this diagram simplifies to the diagram of discrete spaces
\begin{diagram}\label{intro:diag_disc_fil}
    \{a,b\} & \arrow[l, "\pi_1"'] \{a,b\} \times  \{x,y,z\} \arrow[r, "\pi_2"]  & \{x,y,z\}.
\end{diagram}
A simple point count already shows that \cref{intro:diag_disc_lem,intro:diag_disc_fil} distinguish the lemniscate from the space shown in \cref{fig:ball}.
\end{example}
This example illustrates how the stratified homotopy type provides a significantly finer invariant than the classical homotopy type. \\
In order to perform topological data analysis with such diagram representations of stratified homotopy types, we need a persistent analogue, i.e. a \define{persistent stratified homotopy type}. (In this paper, we focus on the two strata case, for reasons elaborated in more detail in the end of the introduction.)\\
In \cref{sec:pers_strat}, we construct such an object by separately thickening the pieces of \cref{diag:intro_sublevel} in a surrounding Euclidean space.
In this fashion, after having chosen parameters $\diagParam = (\diagParamLow, \diagParamUp)$, we may associate to a stratified subset $S \subset \supSp$ a persistent stratified homotopy type, $\epersParam (S)$.\\ By construction, for each $S$ the persistent stratified homotopy type $\epersParam (S)$ is represented by a space valued functor 
\[
 \{ p \leftarrow \{p,q\} \rightarrow q \} \times \mathbb R_+ \to \Top,
\]
 where $\{ p \leftarrow \{p,q\} \rightarrow q \}$ is the indexing category for diagrams such as \cref{diag:intro_abstract}. For each fixed $\varepsilon \in \mathbb R_+$ we recover a diagram indexed over $\{ p \leftarrow \{p,q\} \rightarrow q \}$ (see \cref{ex:thicken_intro} for a visualization), which under \cref{recol:diag_are_equ} corresponds to a (weak) stratified homotopy type. 
Our main results pertaining to this construction may be summarized as follows:
\begin{ctheorem}{A}\label{mainResA}
      The assignment 
    \[
    S \mapsto \epersParam (S),
    \]
    sending a stratified subset $S$ of $\supSp$ with two strata to its persistent stratified homotopy type (depending on a choice of parameters $\diagParam$) fulfills stratified analogues of \ref{enum:properties_of_PH1} to \ref{enum:properties_of_PH3}.
\end{ctheorem}
More specifically, the analogue of \ref{enum:properties_of_PH1} is guaranteed by the fact, that for finite stratified point clouds, we may encode diagrams of the form \cref{diag:intro_sublevel} in terms of diagrams of \v{C}ech complexes (see \cref{rem:computability_of_strat_pers}). The invariance under thickenings part of \ref{enum:properties_of_PH3} is guaranteed by \cref{prop:thicken_stab_lok,prop:thicken_stab_glob}. 
 \cref{prop:thicken_stab_lok} roughly states that for sufficiently constructible two strata subspaces $S \subset \supSp$, and sufficiently small choices of $\diagParam$, the (weak) stratified homotopy type given by $\epersParam (S)_{\varepsilon}$ agrees with the one of  $S$ for sufficiently small $\varepsilon>0$. In other words, the stratified homotopy type does not change under small thickenings. 
 \begin{example}\label{ex:thicken_intro}
     Again, consider the example of the lemniscate, using parameter values $\diagParam = (0.2,0.3)$. For $\varepsilon = 0.12$, the homotopy type of the thickened diagram
     \begin{center}
     \begin{tikzpicture}[scale = 0.815]
        \node(sing) at (-5,0){\includegraphics[width=3.5cm]{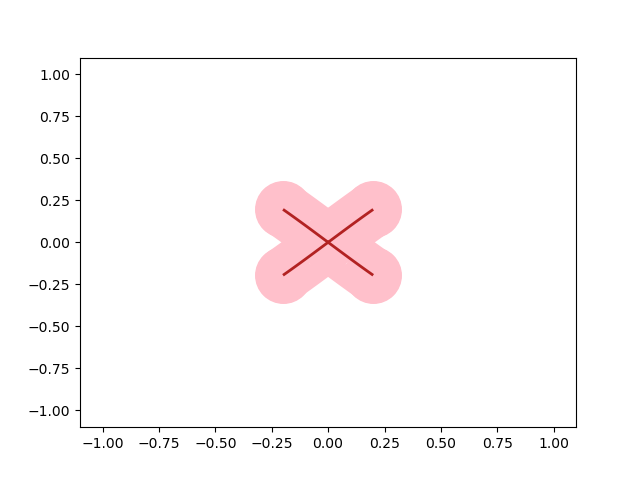}}; 
          \node(link) at (0,0){\includegraphics[width=3.5cm]{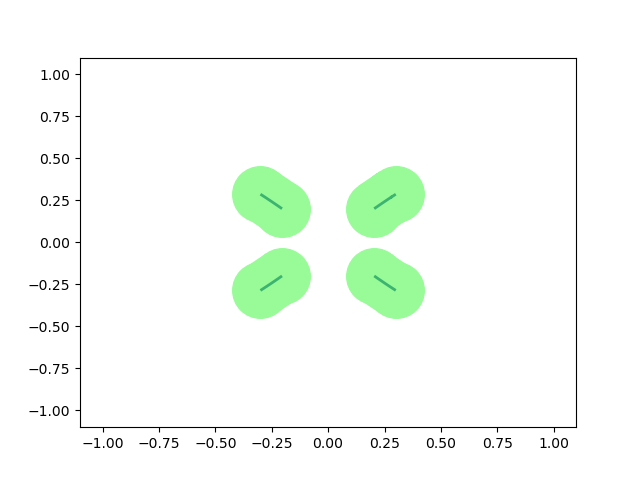}};
                  \draw[ left hook->] (link) -- (sing);
         \node(reg) at (5,0){\includegraphics[width=3.5cm]{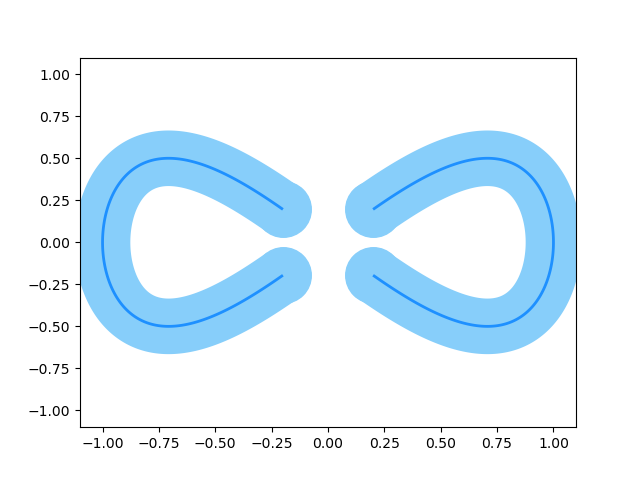}};
          \draw[ right hook->] (link) -- (reg);
    \end{tikzpicture}
\end{center}
     agrees with the unthickened diagram in \cref{ex:intro_sublvl}.
     If instead we take $\varepsilon = 0.24$, then we obtain: \\
     \begin{center}
         \begin{tikzpicture}[scale = 0.815]
        \node(sing) at (-5,0){\includegraphics[width=3.5cm]{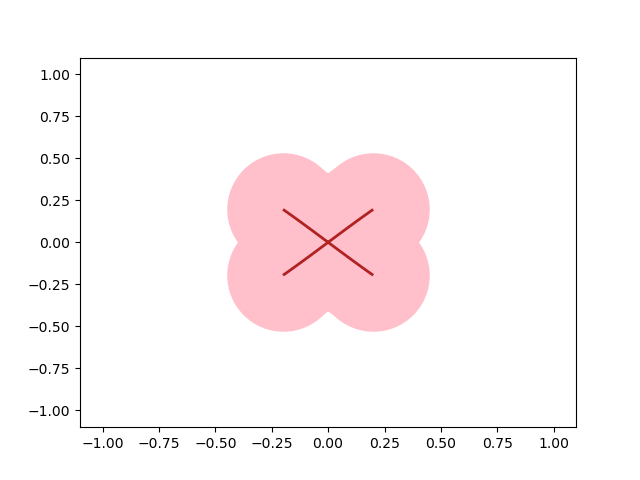}}; 
          \node(link) at (0,0){\includegraphics[width=3.5cm]{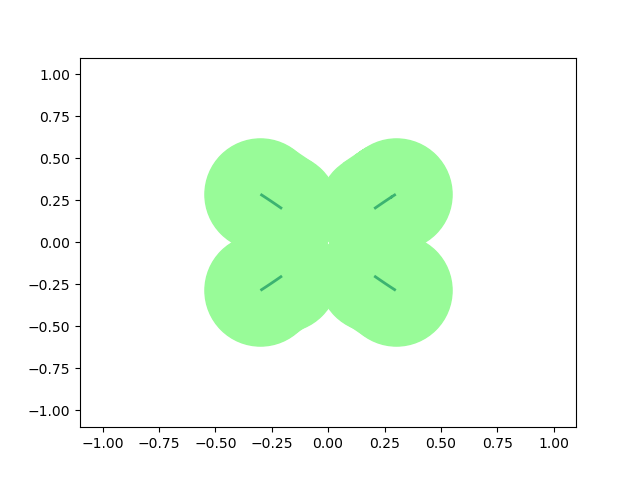}};
                  \draw[ left hook->] (link) -- (sing);
         \node(reg) at (5,0){\includegraphics[width=3.5cm]{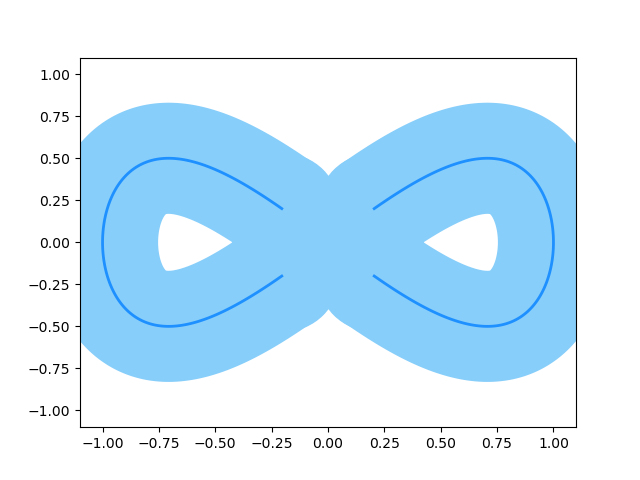}};
          \draw[ right hook->] (link) -- (reg);
    \end{tikzpicture}
    \end{center}
     The latter diagram is (weakly) equivalent to the diagram
     \begin{diagram}
         \star & \arrow[l] \star \arrow[r] & S^1 \cup_{\star} S^1 ,
     \end{diagram}
     which is not (weakly) equivalent to \cref{intro:diag_disc_lem}.
 \end{example}
 Finally, a version of stability - and hence an analogue of \ref{enum:properties_of_PH2} - is guaranteed by \cref{prop:pers_strat_htpy_type_C-lipschitz,cor:continuity_of_pers_tame}. It follows from these results that for a two strata Whitney stratified $W \subset \supSp$ and a sequence of stratified space $\stratSamp^i \subset \supSp$ converging to $W$ (in a stratified version of the Hausdorff distance, see \cref{def:samP}) the associated persistent stratified homotopy types $\epersParam (\stratSamp^i)$ converge to $\epersParam (W)$, and for small $\diagParam$ this convergence is even of Lipschitz type.
\begin{example}
    Convergence in the stratified version of the Hausdorff distance that we used is equivalent to convergence in both the underlying spaces, as well as in the singular strata.
     For example, consider the family of real algebraic varieties 
     \[
    V^s:=\{x \in \mathbb R^2 \mid f_s(x) = x_1^4 -x_1^2 +x_2^2 = s \} \textnormal{, for }s \in \mathbb R,
    \] 
     equipped with the stratification given by singular loci. In other words, the singular stratum of $V^s$ is given by the intersection of $V^s$ with the vanishing set of the Jacobian of $f_s$.
     \begin{figure}[H]
         \centering
         \includegraphics[width=0.3\textwidth]{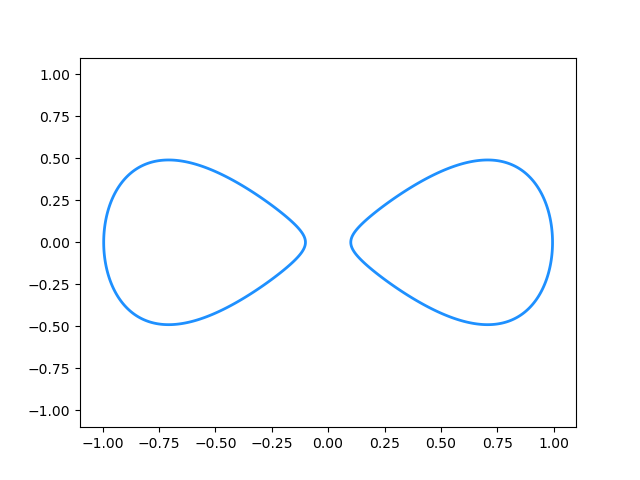}
         \includegraphics[width=0.3\textwidth]{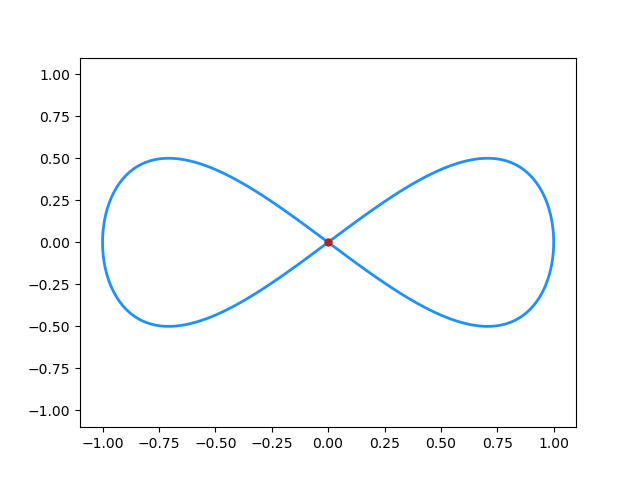}
         \includegraphics[width=0.3\textwidth]{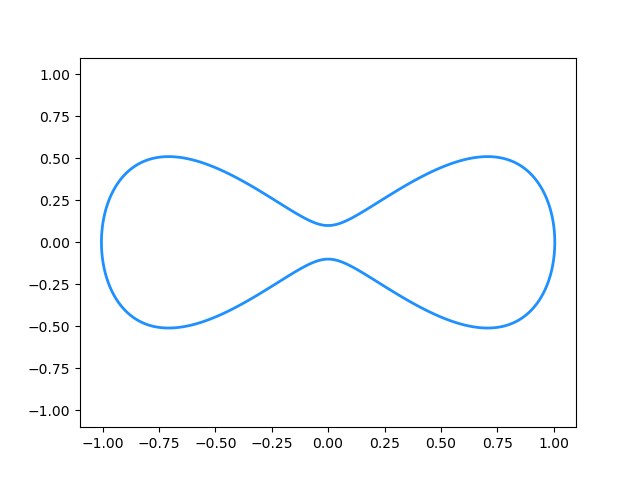}
         \caption{The real algebraic variety $V^s$, for $s\in \{ -0.1,0,0.1$ \}, with singularities marked in red}
         \label{fig:enter-label}
     \end{figure}
     Note that $V^0$ is the lemniscate as in \cref{fig:lemniscate}.
     We may not expect convergence \[
     \epersParam(V^s) \xrightarrow{s \to 0} \epersParam(V^0)\] in stratified Hausdorff distance, since for $s \neq 0$ the variety has no singular points, and hence the singular stratum is empty. In particular, the stratified Hausdorff distance between $V^s$ and $V^0$ is in fact infinite for $s \neq 0$.
     \\
     However, instead of classifying points as singular by the vanishing of the Jacobian of $f_s$, we could, for example, use a more quantitative measure of singularity and let $V^s_0 = \{ x \in V^s \mid \norm{Jf_x} \leq 3\sqrt{s} \}$ (see \cref{fig:strat_conv} for an illustration). 
     \begin{figure}[htp]
     \center
     
          \begin{tikzpicture}[scale = 0.9]
        \node(a) at (-3,0){\includegraphics[width=2.5cm, clip, trim = {1.5cm 0 1.5cm 0}]{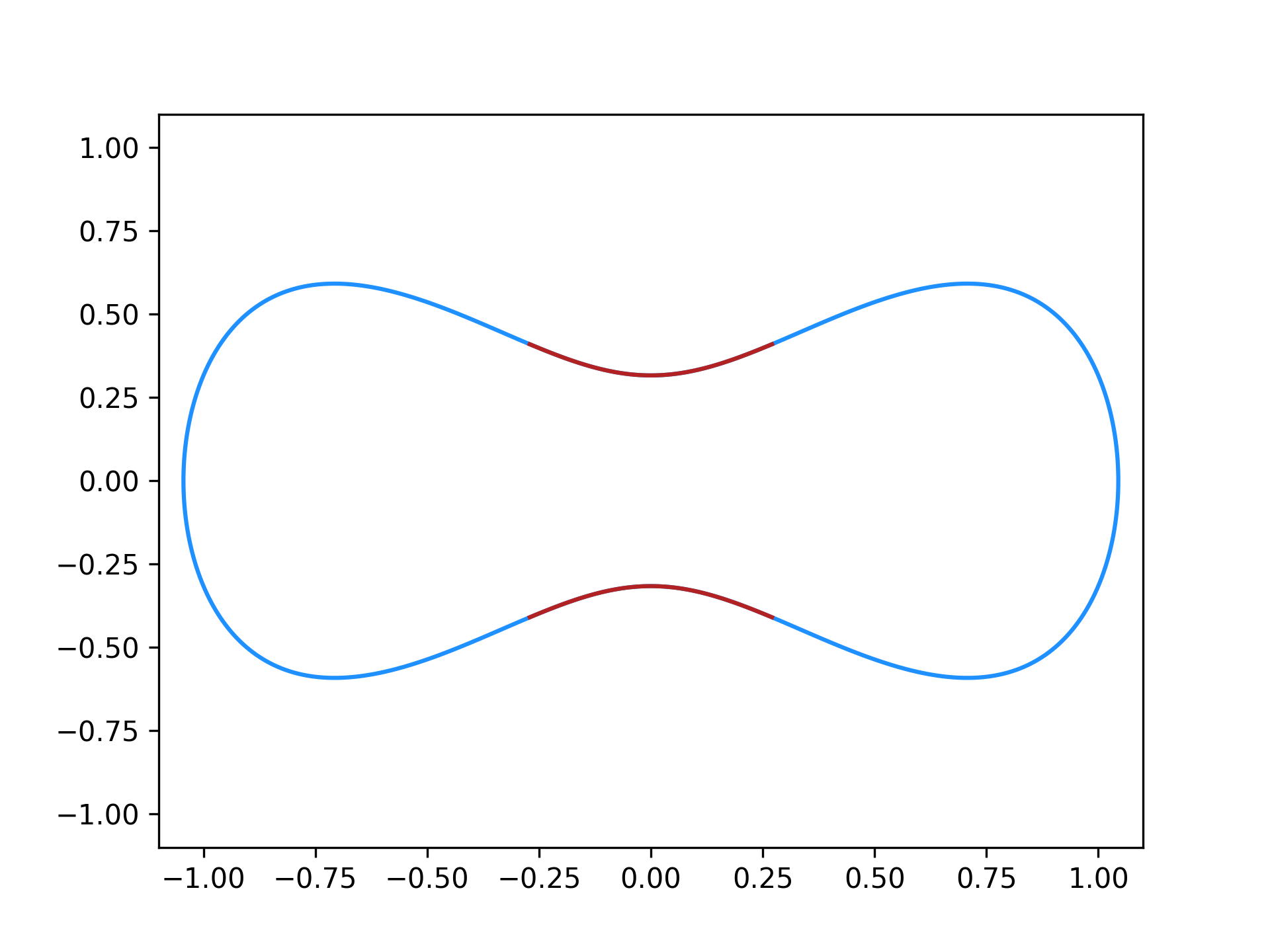}}; 
          \node(b) at (0,0){\includegraphics[width=2.5cm,  clip, trim = {1.5cm 0 1.5cm 0}]{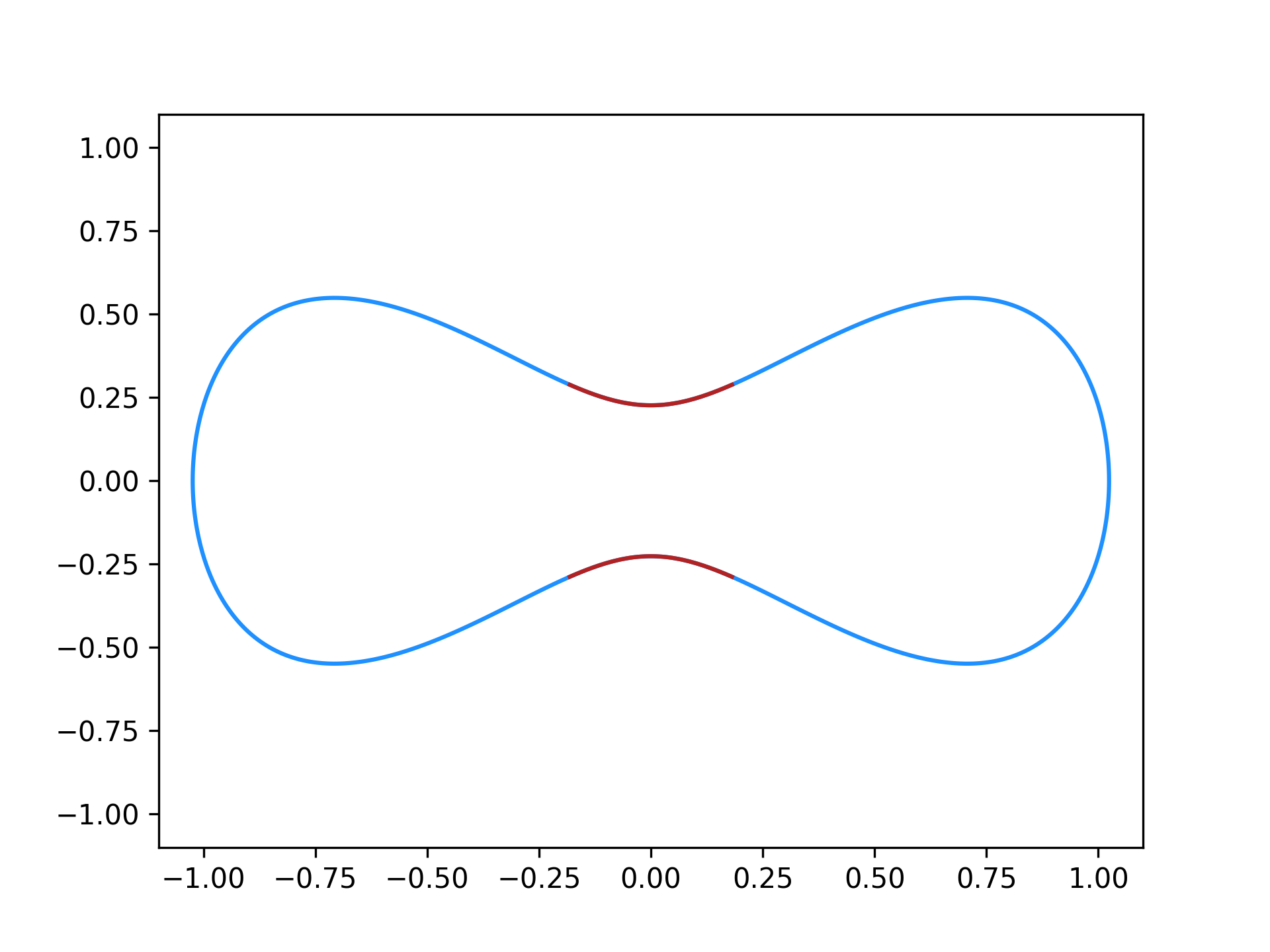}};
         \node(c) at (3,0){\includegraphics[width=2.5cm,  clip, trim = {1.5cm 0 1.5cm 0}]{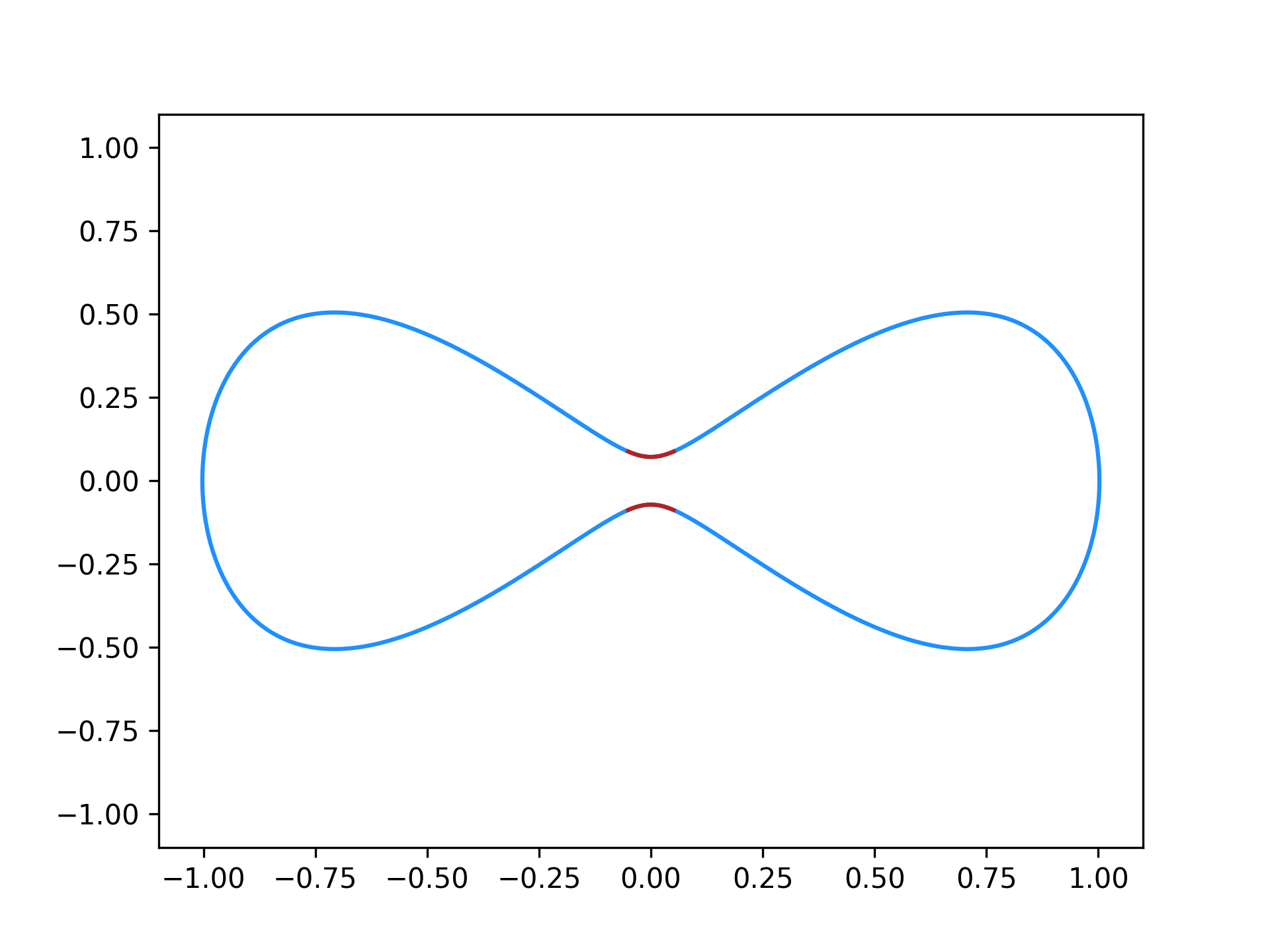}};
         \node(d) at (7,0){\includegraphics[width=2.5cm,  clip, trim = {1.5cm 0 1.5cm 0}]{Figures//lem_s0_sing.png}};
         \draw[->]  (c) -- (d) node[midway, above]{$s \to 0$}; 
    \end{tikzpicture}
    \caption{$V^s$ for $s = 0.1, 0.5, 0.05$ with bottom stratum in red given by $\{ x \in V^s \mid \norm{Jf_x} \leq 3\sqrt{s} \}$.}
    \label{fig:strat_conv}
    \end{figure}
    With these alternative stratifications $V^s$ converges to $V^0$ in stratified Hausdorff distance, as indicated in \cref{fig:strat_conv}. Hence, convergence of persistent stratified homotopy types also holds.
     The general approach of using a more quantitative measure of singularity is also central in \cref{sec:recover_strat}, which deals with stratifying point clouds.
\end{example}
Together, the statements of \cref{mainResA} allow for the computational inference of stratified  homotopy theoretic information about a two strata Whitney stratified space $W$ from a sufficiently close (potentially noisy), finite, stratified sample.

\subsection{Stratification learning through tangent cones}

In an applied scenario we would generally not expect a point cloud to already be equipped with an appropriate stratification. Much attention has been paid to how such stratifications can be obtained \cite{mileyko2021another,stolz2020geometric,nanda2020local,skraba2014approximating,bendich2012local,fasy2016exploring}. In \cref{sec:recover_strat}, we provide a method as well as theoretical guarantees (\cref{thrm:recovery_thrm}) for the scenario of approximating a stratified space from a non-stratified sample, in the stratified Hausdorff distance. 
Our approach is inspired by the approaches to persistent local homology in \cite{bendich2012local,skraba2014approximating,nanda2020local,mileyko2021another}, which detects singularities using local data of the form $\ballEps[\frac{1}{\zeta}]{x} \cap W$, for large $\zeta>0$ and $x \in W \subset \supSp$. From a scale independent point of view, one may equivalently consider the so-called 
\define{magnifications}
\[
\magnification{W} := \zoomParam (W - x) \cap \ballEps[1]{0}.
\]

If $W$ is a sufficiently constructible Whitney stratified space, then it is a classical result (see \cite{hironaka1969normal,bernig2007tangent}) that as $\zeta$ converges to $\infty$ these magnifications converge to the (unit ball in the) so-called \define{extrinsic tangent cone} at $x\in W$ (see \cref{def:tangent_cone}) - a generalization of tangent spaces in the singular setting.
Thus, the class of stratified spaces we investigate are the \define{tangentially stratified} (see \cref{def:tangentially_stratified}) spaces, for which singularities may be detected by their extrinsic tangent cones. 
\begin{figure}
    \centering
    \includegraphics[width = 0.4\textwidth]{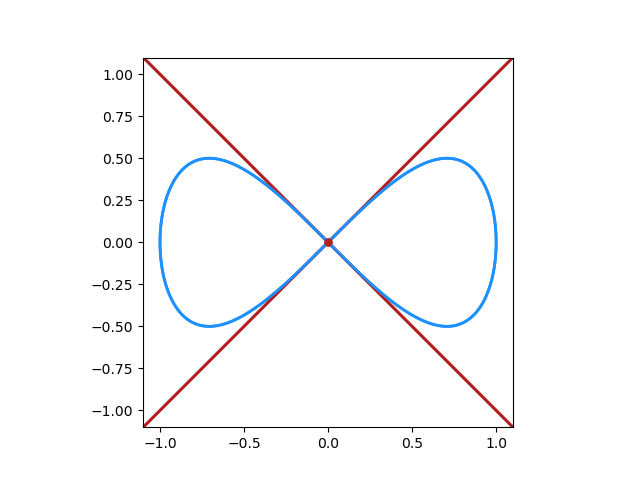}
        \includegraphics[width = 0.4\textwidth]{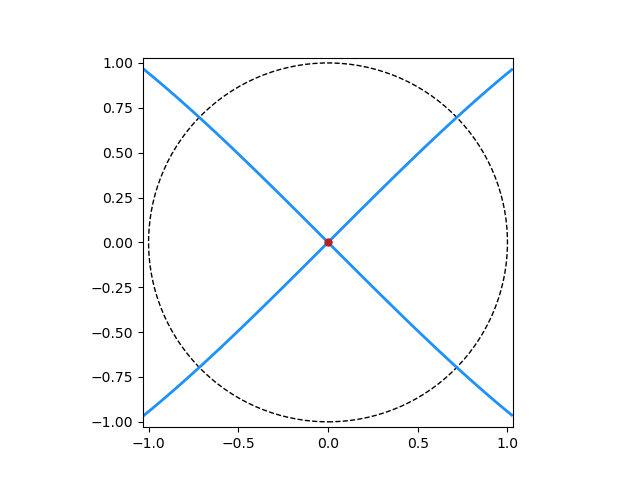}
    \caption{Tangent cone (red) at origin of the lemniscate and magnification at the origin for $\zeta = 3$}
    \label{fig:tangent_cone_lem}
\end{figure}
\begin{example}
     In the case of the lemniscate, for example, the tangent cone at the origin is given by the algebraic variety \[
\{ x \in \mathbb R^2 \mid (x_1 - x_2) (x_1 + x_2) = 0 \}
\]
which identifies the origin as singular (see \cref{fig:tangent_cone_lem}). Hence, with respect to the standard stratification, the lemniscate is tangentially stratified.
\end{example}
For the purpose of topological data analysis, we needed a global sampling version of the convergence results of \cite{hironaka1969normal}. More specifically, we needed a version which at the same time allows for the replacement of $W$ by a sufficiently close sample, and furthermore holds with respect a global notion of convergence, where a single tangent cone is replaced by a bundle of tangent cones.
We prove such convergence results in \cref{prop:loc_sampl_convergence_of_magnifications,prop:tangent_co_conv_compacta}. \\
\\
In \cref{subsec:restrat}, we leverage the convergence guaranteed by \cref{prop:tangent_co_conv_compacta} to provide theoretical guarantees for learning stratifications with two strata from a non-stratified sample $\sampX \subset \supSp$ close to a tangentially Whitney stratified space $W$. To decide which points of $\sampX$ should be considered singular, we consider specific functions  (see \cref{def:Phi_strat})
\[
 \Phi \colon \{ \textnormal{Local Data} \} \to [0,1],
\] 
which quantifiy the degree of degeneracy of $\magnification{\sampX}$ (with $1$ being regular and $0$ highly singular). 
Examples of such functions $\Phi$ come from minimizing truncated Hausdorff distances to $q$-dimensional linear subspaces (where $q$ is the dimension of $W$, see \cref{ex:phi-strat_hausdorff}) and local persistent homology (see \cref{ex:phi-strat_pers_cohom}). \\
A Whitney stratified space $W \subset \supSp$, with two strata, for which the singular stratum $W_p$ is precisely given by such points $x \in W$, for which $\Phi(\tangentCo W) < 1$, is called \define{(tangentially) $\Phi$-stratified}. In other words, $\Phi$-stratifications are precisely the stratifications which we may hope to learn through the lens of $\Phi$.
Tangentially stratified spaces, for example, are precisely the $\Phi$-stratified spaces, with respect to the function of \cref{ex:phi-strat_hausdorff}. \\
After a choice of cutoff parameter $u \in (0,1)$, one may then turn a non-stratified point cloud $\sampX \subset \supSp$ into a stratified point cloud with two strata, denoted $\ePhiStr(\sampX)$, by taking the singular stratum to be
\[
\ePhiStr(\sampX)_p = \{x \in \sampX \mid \Phi( \magnification \sampX ) \leq u\}.
\]
This construction depends, of course, on the magnification parameter $\zeta >0$. 
Our main result on stratification learning through tangent cones then describes the convergence behavior of $\ePhiStr(\sampX)$ in $\zeta$ and $\sampX$. In the following $\dmet[\HD]{-}{-}$ denotes non-stratified Hausdorff distance:
\begin{ctheorem}{B}[\cref{thrm:recovery_thrm}]
        Let $W \subset \supSp$ be a compact (sufficiently constructible, see \cref{def:Loj_Whitney}) Whitney stratified space, which is $\Phi$-stratified with respect to a function $\Phi$ as in \cref{def:Phi_strat}. Denote by $X$ the underlying space of $W$.
    Then there exists $u_0 \in (0,1)$ such that for all $u\in [u_0,1)$, the convergence in stratified Hausdorff distance
    \[
    \ePhiStr[\zeta](\sampX) \to W,
    \]
    holds, for $\zeta \to \infty$ and $\sampX \to X$ such that $\zeta \dmet[\HD]{\sampX}{X} \to 0$. 
 \end{ctheorem}
In other words, we may approximate the singular stratum $W_p$ of $W$ by $\ePhiStr[\zeta](\sampX)_p$ if $\zeta$ is sufficiently large and given that $\sampX$ is a sufficiently fine approximation of $X$, where the sufficient degree of fineness depends on $\zeta$.
This dependence on $\zeta$ is not surprising at all, in fact, it is simply a rigorous restatement of the following principle:
To recover local information from a sample, the quality of the sample needs to be finer by some magnitude than the locality scale we work at.
\begin{example}\label{ex:3x3_strat_sam}
Below, we depict a family of point clouds converging to the lemniscate $V^0$ in Hausdorff distance. These samples have been stratified using $\Phi$ as in \cref{ex:phi-strat_hausdorff} with $u = 0.6$, and we marked points in $\ePhiStr[\zeta](\sampX)_p$ with thickened red dots. 
In horizontal direction from left to right the non-stratified Hausdorff distance, denoted $d$, to $V^0$ decreases. In vertical direction going downwards we increase the magnification parameter $\zeta$.\\
\\
\begin{centering}
       \begin{tikzpicture}
        \newdimen\x
        \x = 4.49cm     
        
        \node (lu) at (-0.8\x,0.8\x){\includegraphics[trim={7cm 1cm 4cm 2cm},clip,width=\x]{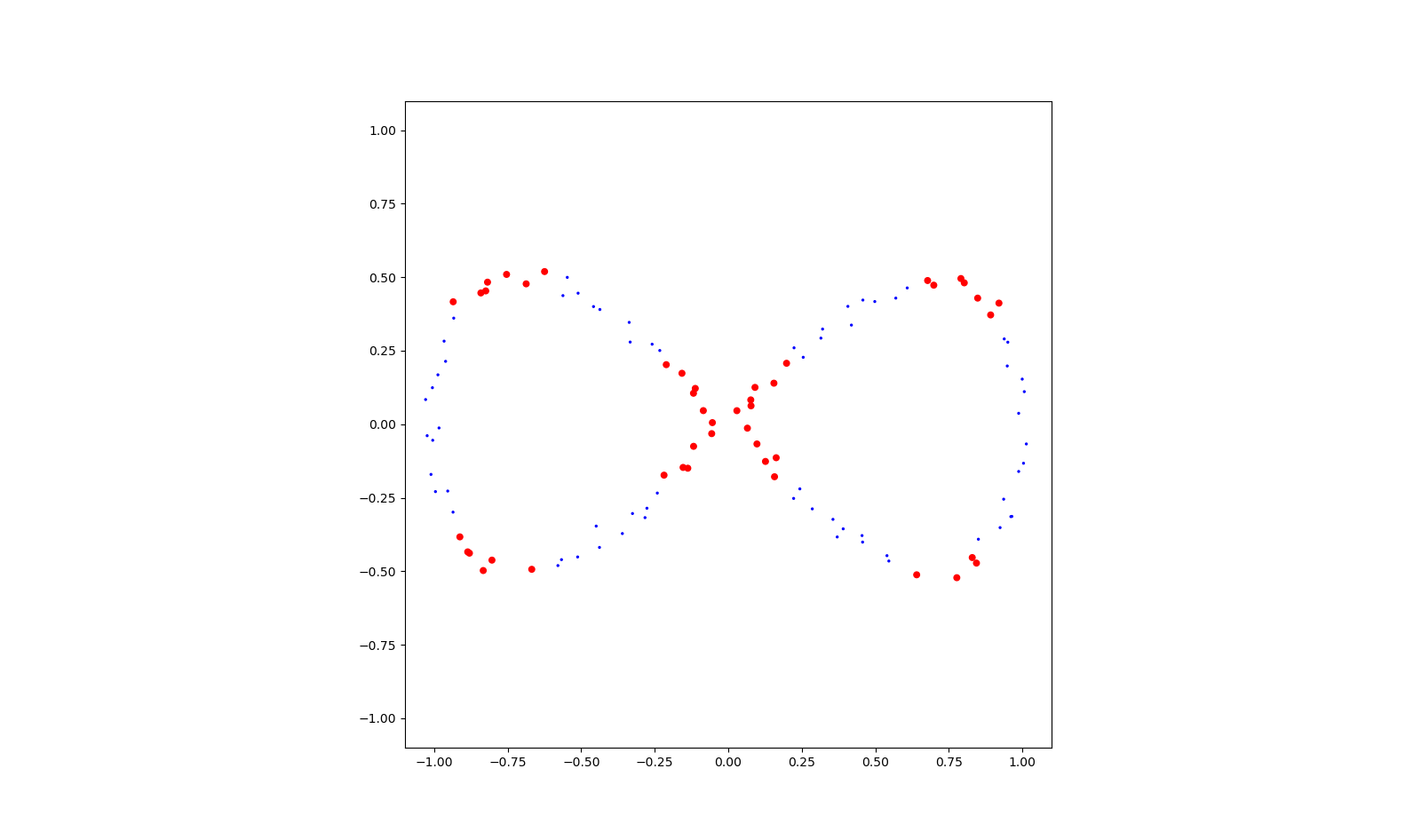}}; 
         \node (lm) at (-0.8\x,0){\includegraphics[trim={7cm 1cm 4cm 2cm},clip,width=\x]{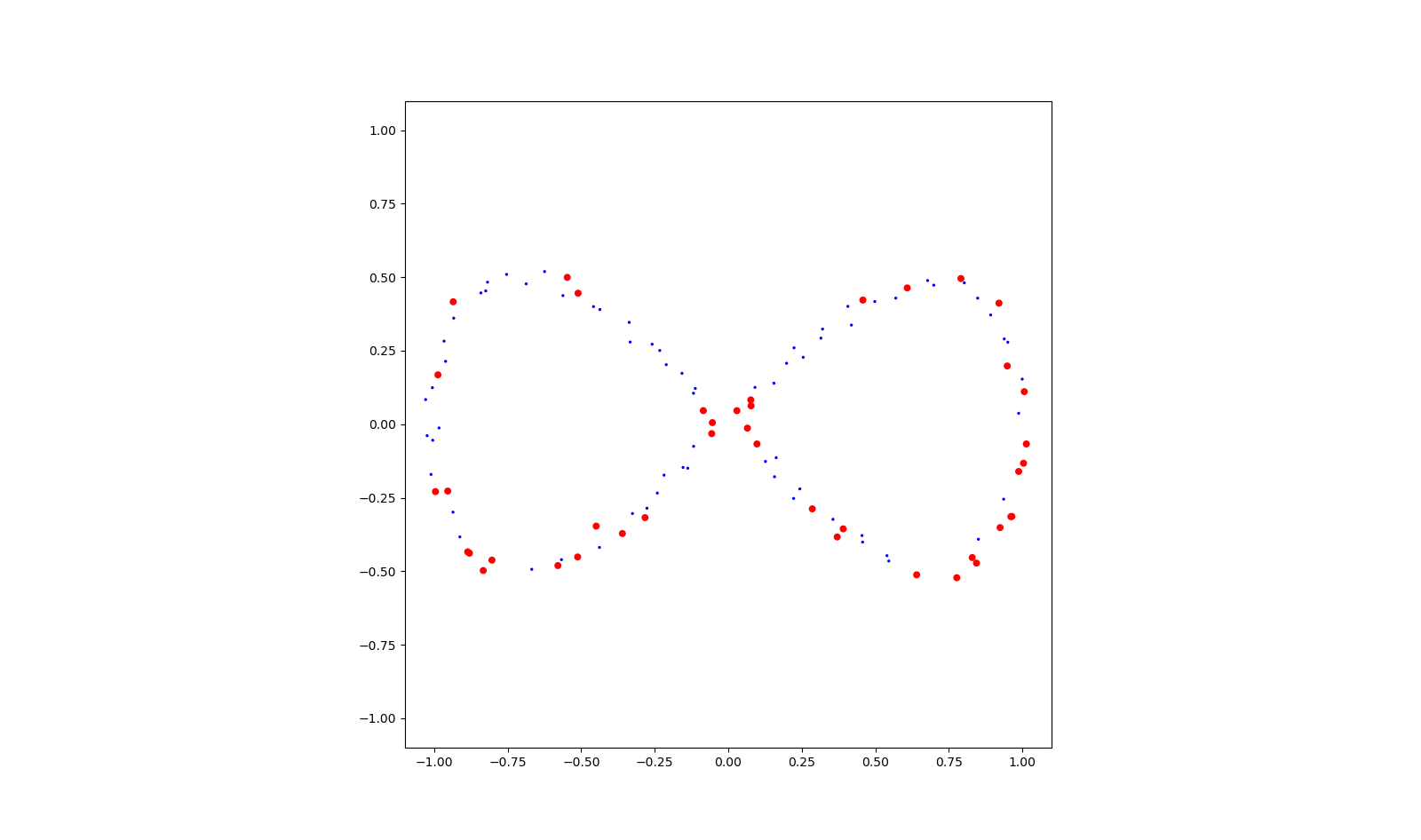}}; 
         \node (ld) at (-0.8\x,-0.8\x){\includegraphics[trim={7cm 1cm 4cm 2cm},clip,width=\x]{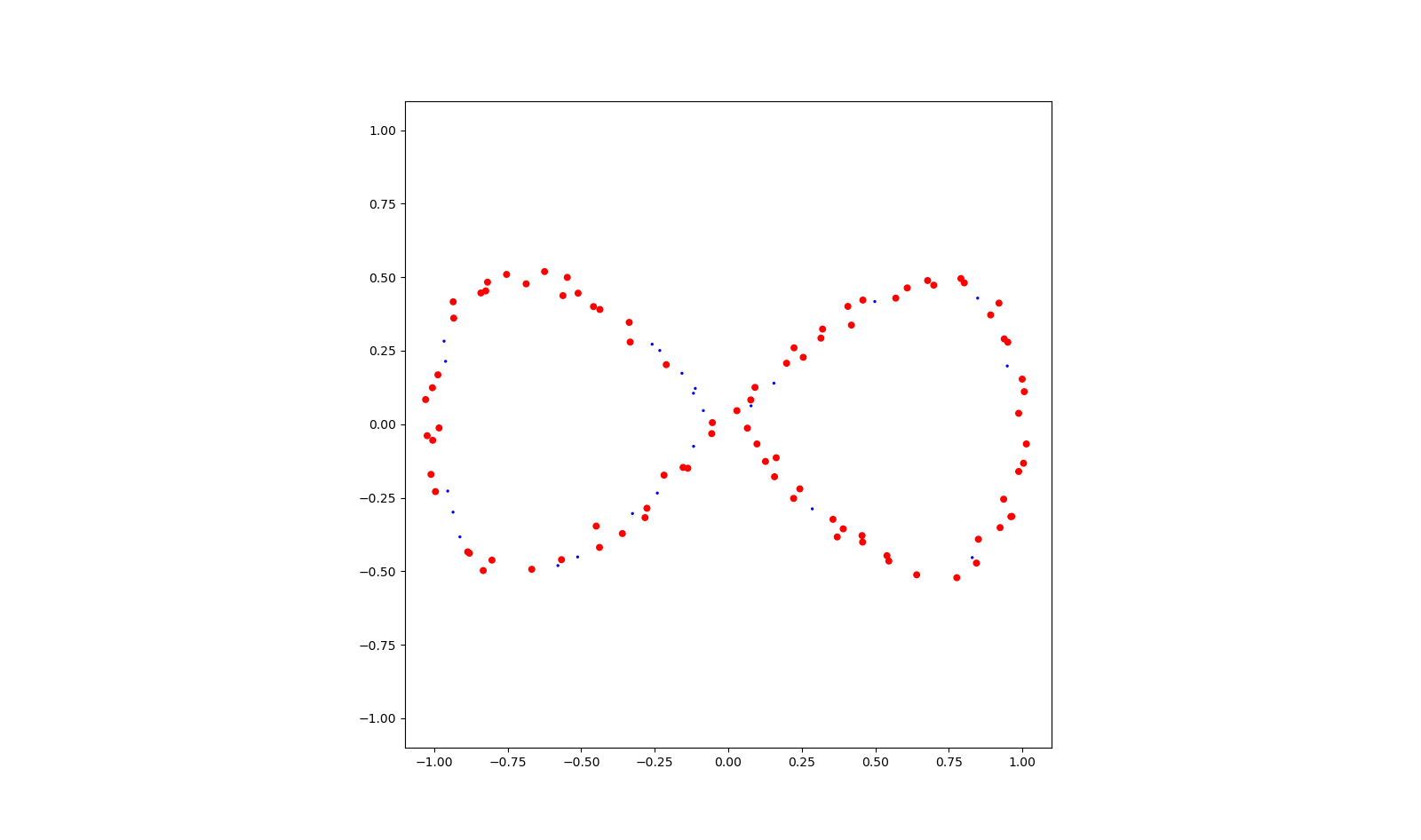}}; 
         \node (mu) at (0,0.8\x){\includegraphics[trim={7cm 1cm 4cm 2cm},clip,width=\x]{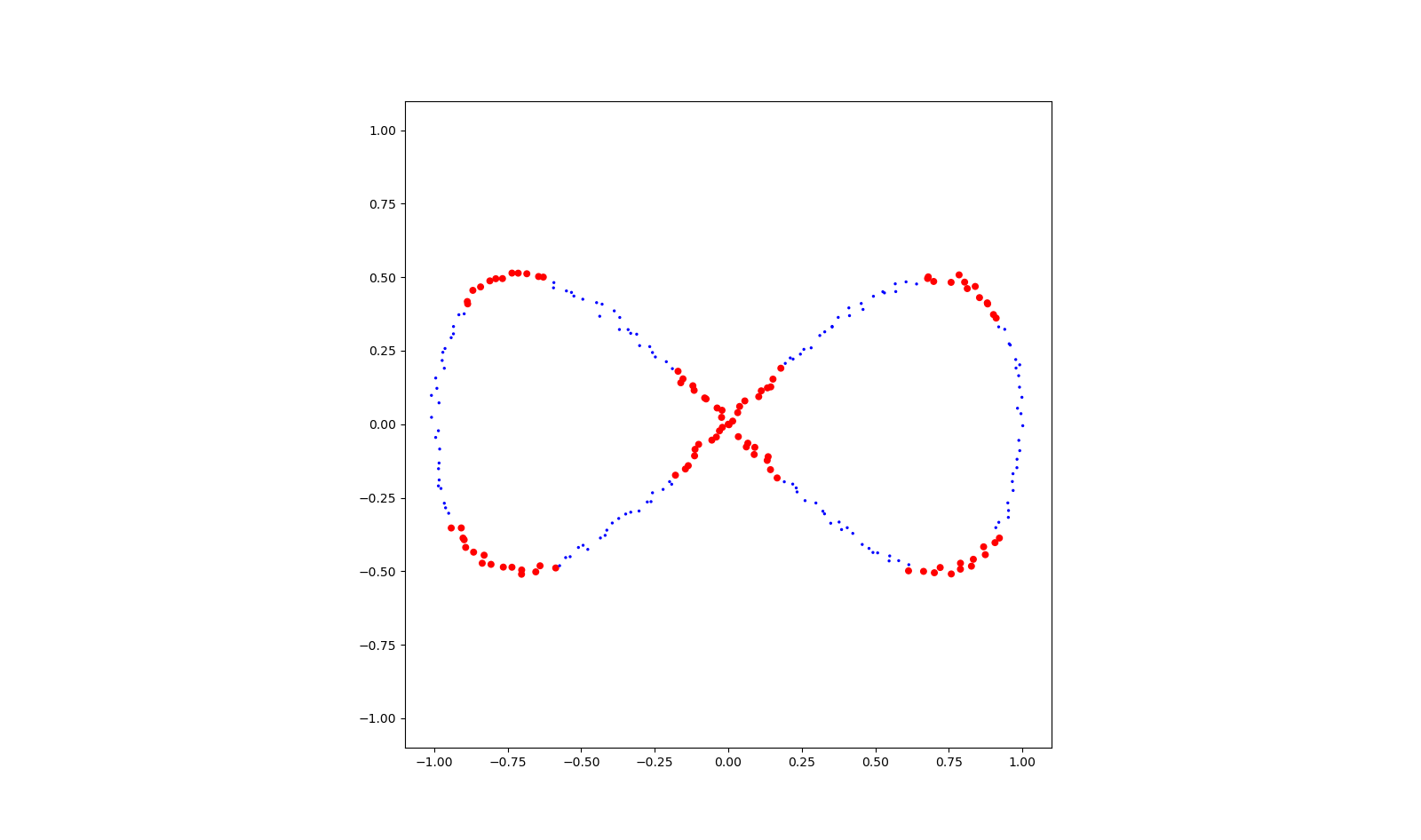}}; 
         \node (mm) at (0,0){\includegraphics[trim={7cm 1cm 4cm 2cm},clip,width=\x]{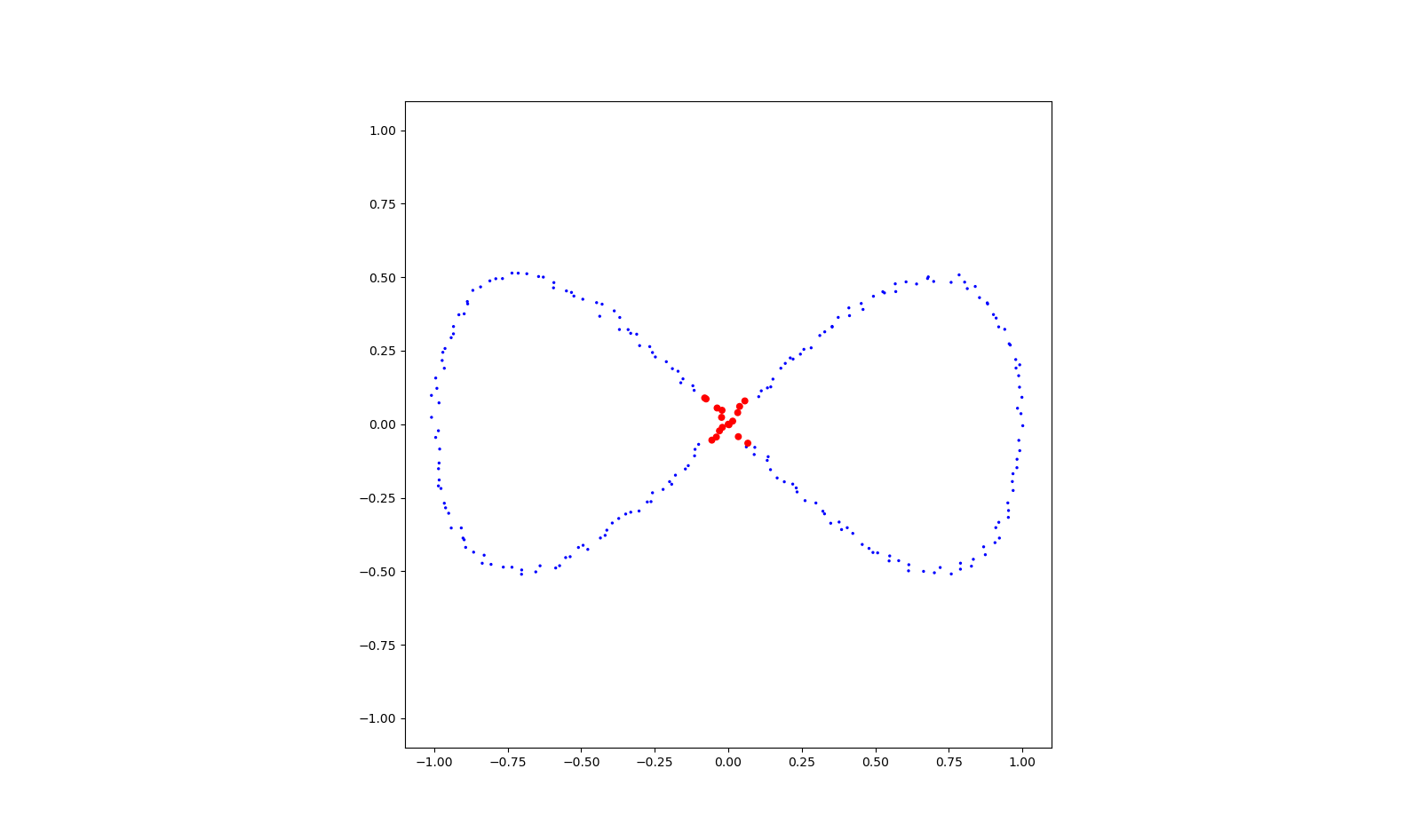}}; 
         \node (md) at (0,-0.8\x){\includegraphics[trim={7cm 1cm 4cm 2cm},clip,width=\x]{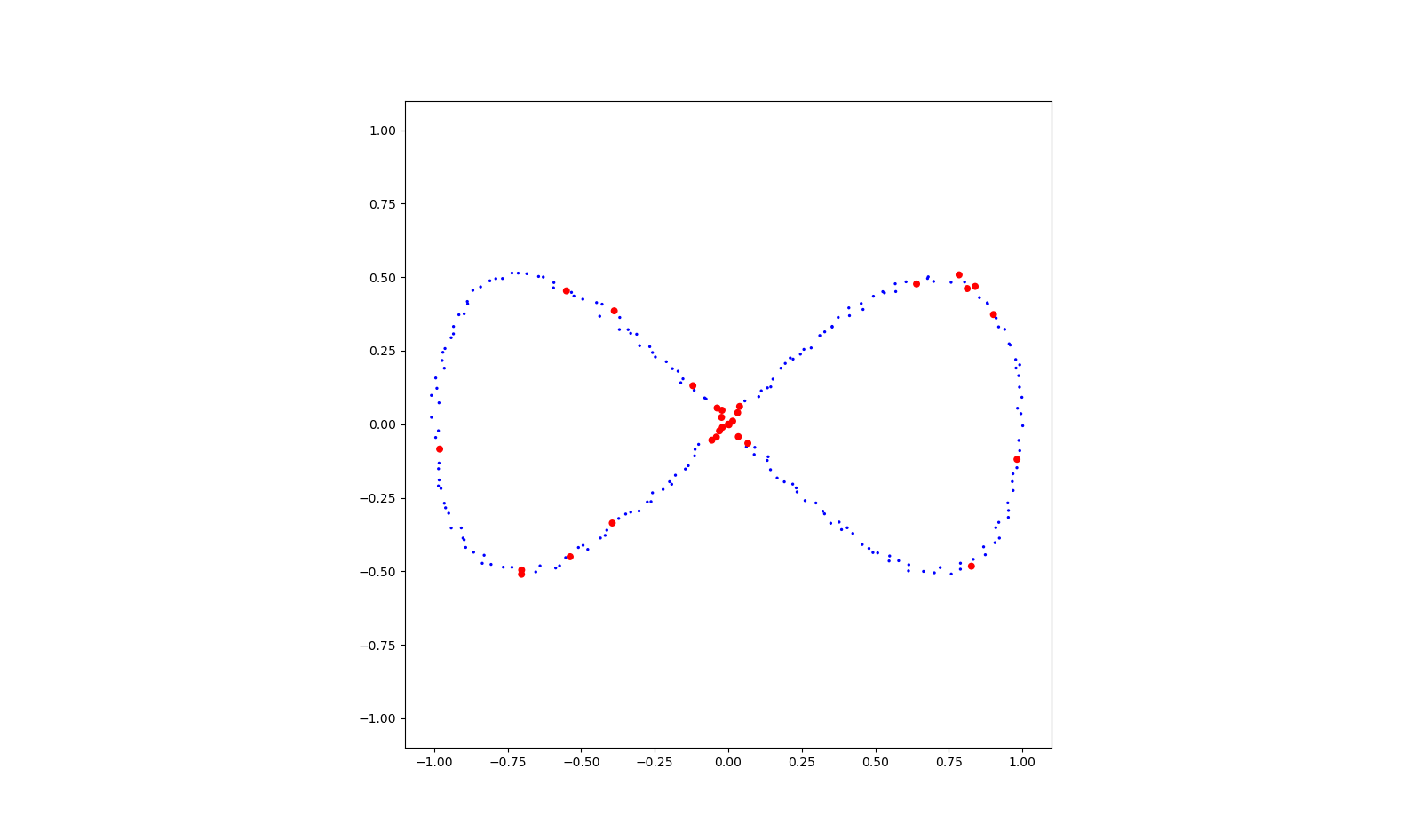}}; 
         \node (ru) at (0.8\x,0.8\x){\includegraphics[trim={7cm 1cm 4cm 2cm},clip,width=\x]{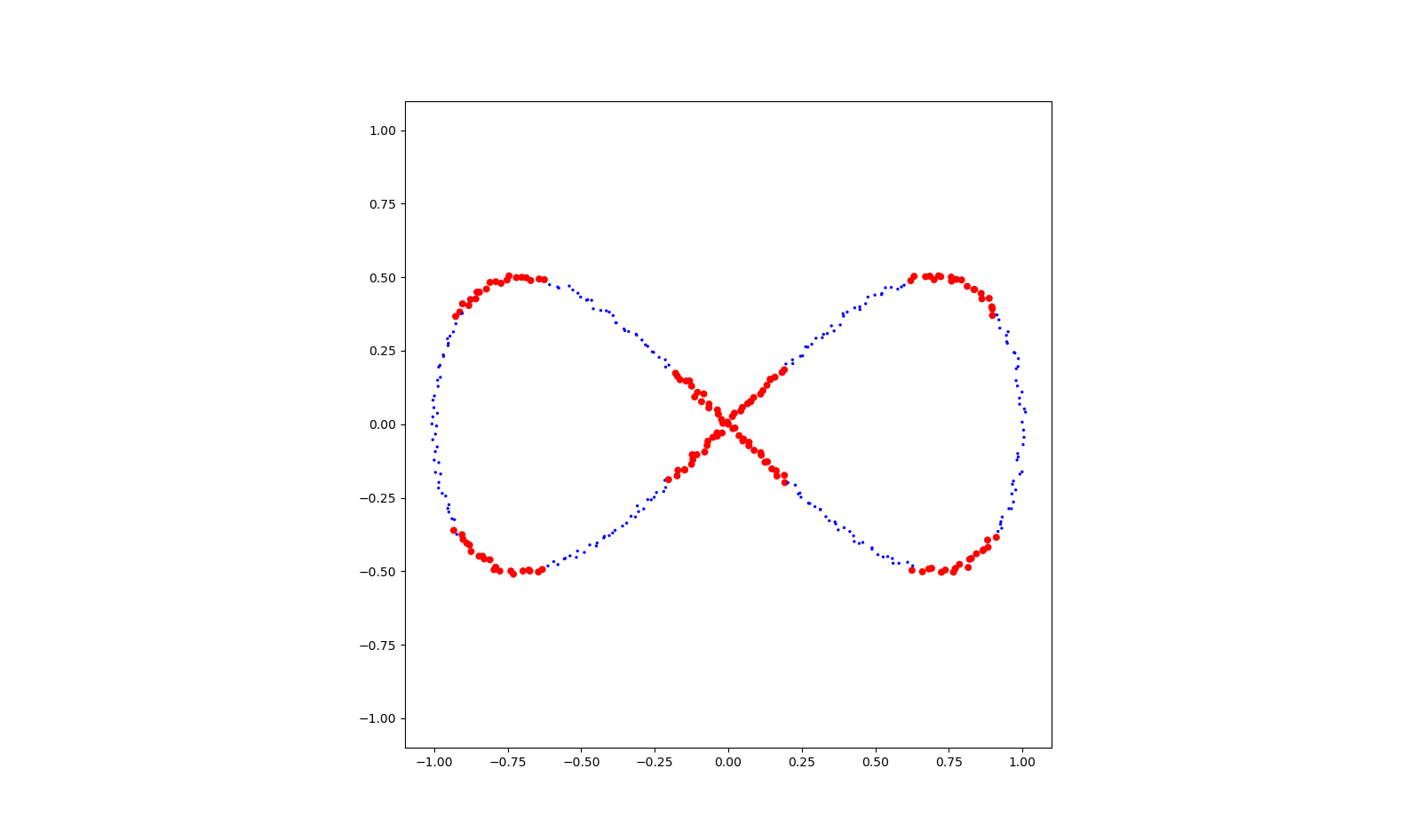}}; 
         \node (rm) at (0.8\x,0){\includegraphics[trim={7cm 1cm 4cm 2cm},clip,width=\x]{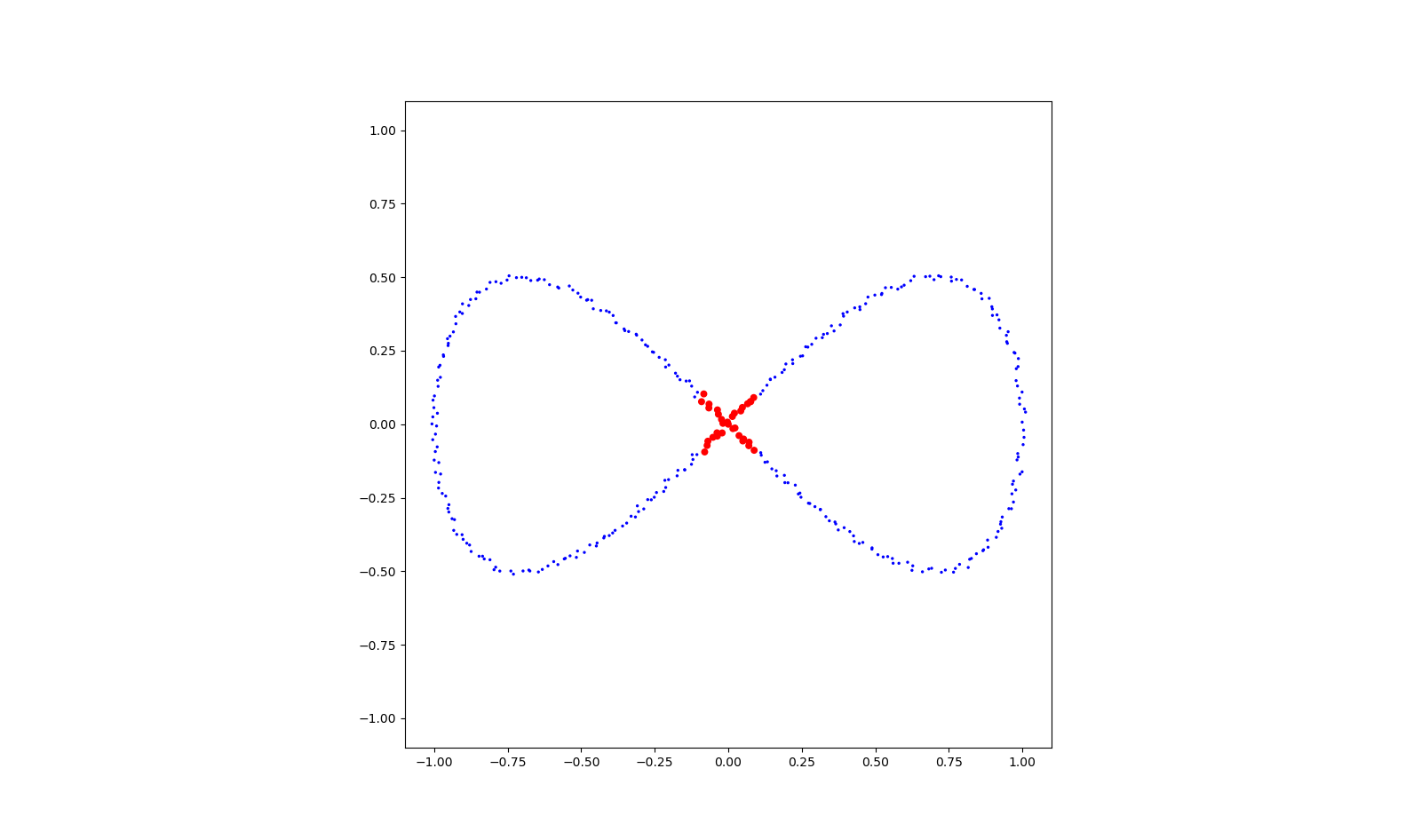}}; 
         \node (rd) at (0.8\x,-0.8\x){\includegraphics[trim={7cm 1cm 4cm 2cm},clip,width=\x]{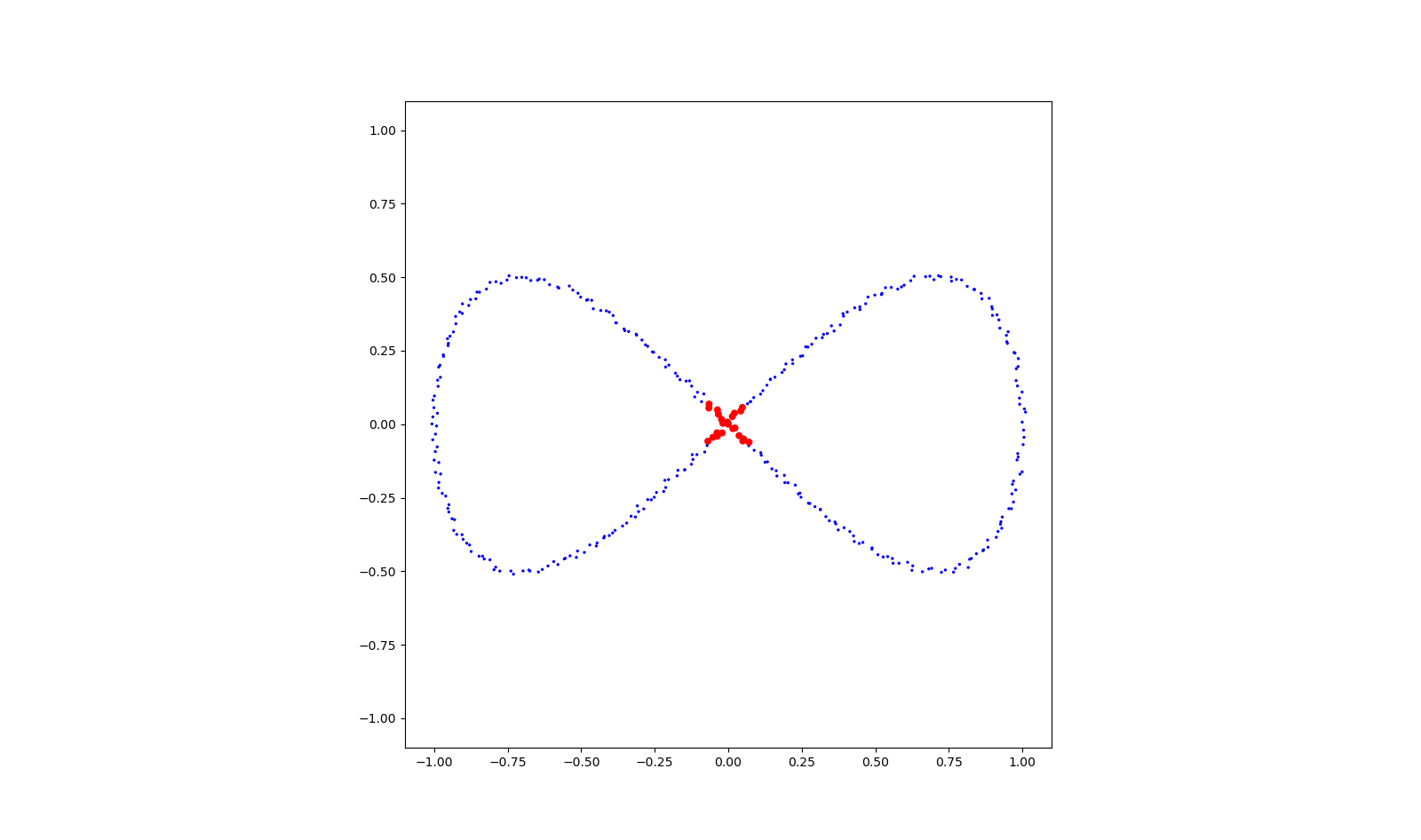}};
         \node(d1) at (-0.8\x, 1.25\x) {$d \approx 0.07$};
          \node(d1) at (0, 1.25\x) {$d \approx 0.035$};
           \node(d1) at (0.8\x, 1.25\x) {$d \approx 0.023$};
         \node[align = center] (z1) at (-1.3\x, 0.8\x){ $\heqto{\zeta}{ 3}$ };
         \node (z1) at (-1.3\x, 0){$\heqto{\zeta}{ 6}$};
          \node (z1) at (-1.3\x, -0.8\x){$\heqto{\zeta}{ 9}$};
    \end{tikzpicture}
\end{centering}
In practice, these choices of $\Phi$ and $u$ often tend to be too sensitive, especially when it comes to detecting curvature (see also \cref{ex:phi-strat_hausdorff}). In this example, however, they serve well to illustrate the convergence behavior in magnification parameter and Hausdorff distance. \\
First, note that for $\zeta = 3$, points in regions with high curvature are generally classified as belonging to $\ePhiStr[\zeta](\sampX)_p$. This is not surprising as for such comparatively small values, we may not expect magnifications at regular points to be close to the tangent spaces yet. 
By doubling $\zeta$, we may correctly classify these regions, provided that the sample quality is good enough to also approximate things at a local level (see the middle row). In this case, only an area around the singularity is classified as belonging to $\ePhiStr[\zeta](\sampX)_p$.
If we want to further shrink this area, and thus obtain a better approximation in stratified Hausdorff distance, we may again decrease $\zeta$ (see third row, in particular the most right picture). In case our sample quality is not sufficient, this may also classify several points far away from the singularity as singular though (see the second picture of the third row). \\ For less sensitive choices of parameter $u$ and alternative choice of $\Phi$, samples of lesser quality and smaller $\zeta$ may lead to good approximations in stratified Hausdorff distance. Consider, for example, \cref{fig:better_u} which was obtained with the same $\Phi$ and $u=0.4$. 
\begin{figure}
        \centering
       \includegraphics[width=0.4\textwidth]{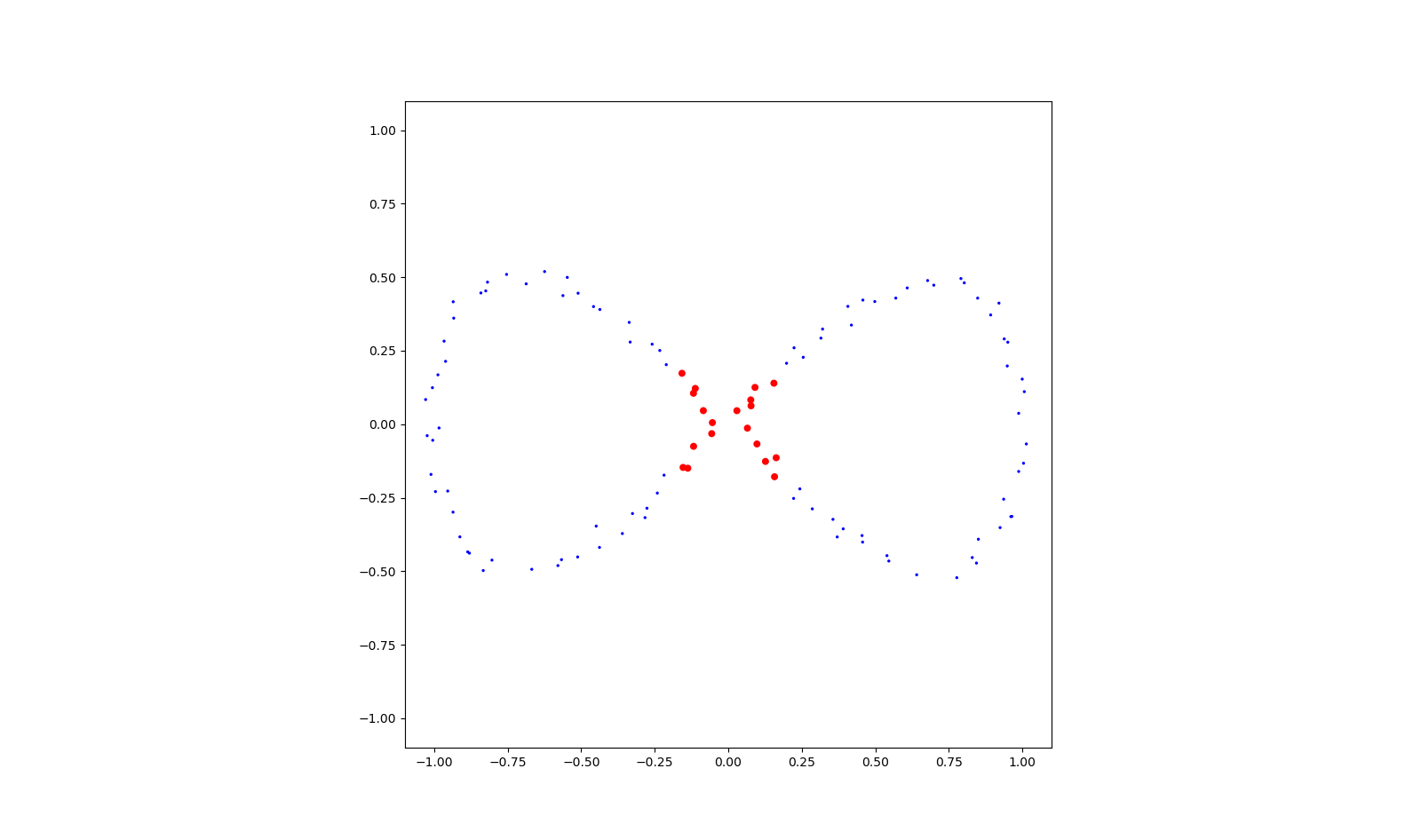}
        \caption{Approximated stratification of a sample around the lemniscate with $\zeta = 3$ and $u=0.4$.}
        \label{fig:better_u}
 \end{figure}
\end{example}

Finally, combining our two methods, i.e. persistent stratified homotopy types and stratification learning through tangent cones, we obtain a pipeline which associates to a non-stratified sample a persistent stratified homotopy type. The combination of \cref{prop:pers_strat_htpy_type_C-lipschitz} and \cref{thrm:recovery_thrm}, which is \cref{cor:composed_convergence}, guarantees that we may indeed approximate persistent stratified homotopy types of sufficiently regular Whitney stratified spaces from nearby samples, under the assumptions of \cref{thrm:recovery_thrm}. Using \cref{prop:thicken_stab_lok} we may then infer from these approximations information about the stratified homotopy type of $W$.
\begin{example}
As an illustration of the convergence of persistent stratified homotopy types obtained from non-stratified samples, consider \cref{fig:intro_barcodes}. It shows the barcode of the $0$-homology of the link part of the persistent stratified homotopy types associated to the stratified point clouds of \cref{ex:3x3_strat_sam}, going in a zigzag above the diagonal from the upper left to the lower right corner.
   \begin{figure}[htp]
       \centering
      \begin{centering}
       \begin{tikzpicture}
            \newdimen\x
        \x = 2.1cm   
            \node[align =center, text width = \x] (x1)  at (0,0){{\includegraphics[trim={1cm 0cm 1cm 1cm},clip ,width=\x]{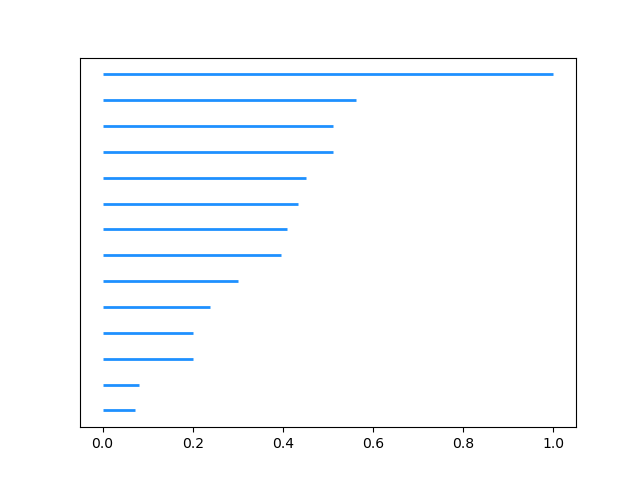}} \tiny $d = 0.07$ \\ $\zeta = 3$};
            \node[align =center, text width = \x] (x2)  at (1.1\x,0){{\includegraphics[trim={1cm 0cm 1cm 1cm},clip, width=\x]{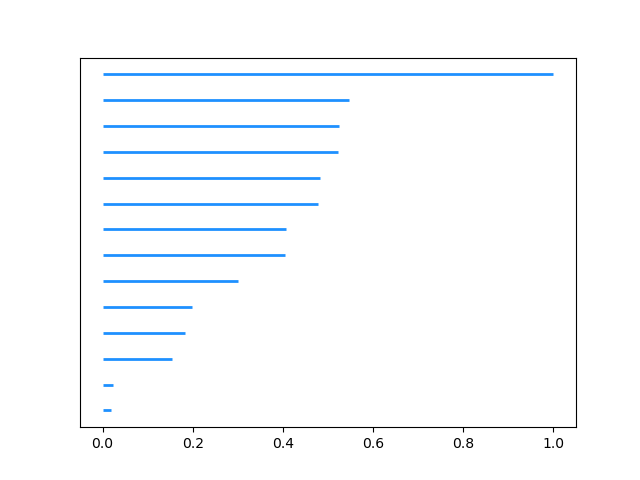}} \\ \tiny $d = 0.035$ \\ $\zeta = 3$ };
            \node[align =center, text width = \x] (x3)  at (2.2\x,0){{\includegraphics[trim={1cm 0cm 1cm 1cm},clip, width=\x]{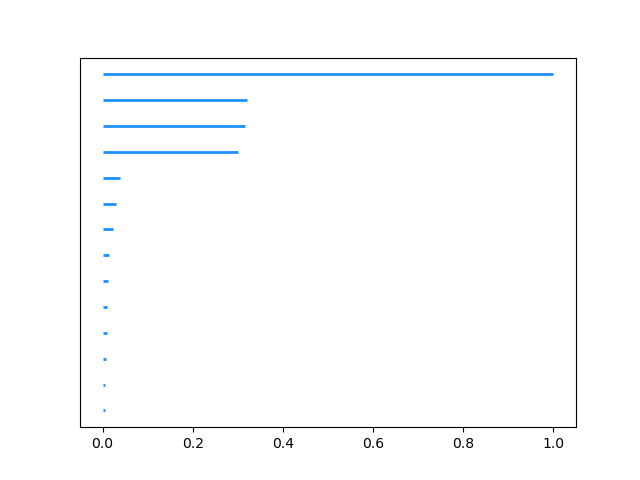}} \\ \tiny $d = 0.035$ \\ $\zeta = 6$};
            \node[align =center, text width = \x] (x4)  at (3.3\x,0){{\includegraphics[trim={1cm 0cm 1cm 1cm},clip, width=\x]{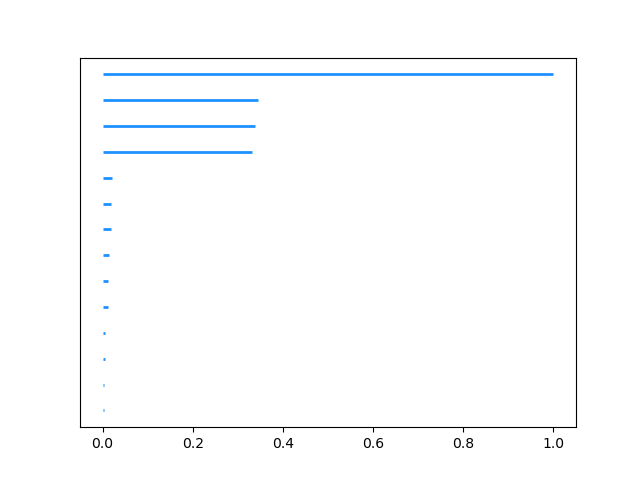}}\\ \tiny $d = 0.023$ \\ $\zeta = 6$};
            \node[align =center, text width = \x] (x5)  at (4.4\x,0){{\includegraphics[trim={1cm 0cm 1cm 1cm},clip, width=\x]{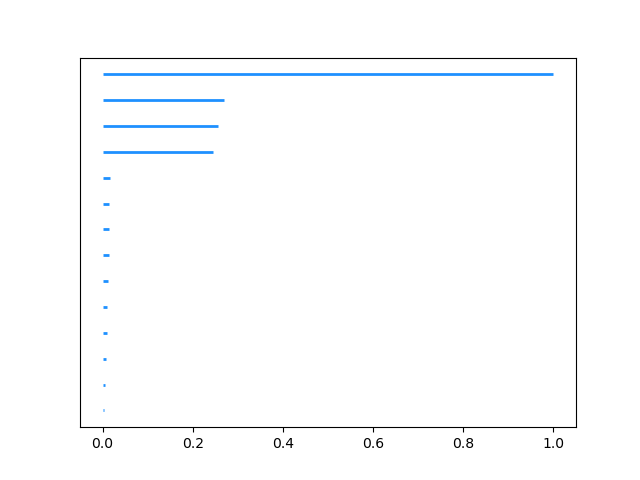}}\\ \tiny $d = 0.023$ \\ $\zeta = 9$};
            \end{tikzpicture}
   \end{centering}
       \caption{Barcodes (showing the $14$ longest bars) of the $0$-th homology of the link part of $\epersParam ( \ePhiStr(\sampX))$, with varying $\sampX$ and $\zeta$. Here $\sampX$ is a point cloud of Hausdorff distance to $V^0$ approximately $d$.}
       \label{fig:intro_barcodes}
   \end{figure}
   Note that for $\zeta = 3$, five disjoint regions are detected as singular and the link ends up having twelve path components. For $\zeta = 6$ and $\zeta = 9$, however, the stratified samples are close to $V^0$ in stratified Hausdorff distance and we detect four path components in the link. This is the expected number of path components from \cref{ex:intro_sublvl}.
\end{example}
\subsection{The case of more than two strata}\label{subsec:more_strata}
Let us end this introduction with some remarks on the case of multiple strata.
In fact, many of the constructions throughout this paper (such as the construction of diagrams representing stratified homotopy types) can be generalized to the case of more than two strata through a serious amount of inductive and technical effort. We have chosen to focus on the two strata case for the following reasons:
\begin{enumerate}
    \item The main goal of this paper is to establish a pipeline going all the way from a (non-stratified) sample of a stratified space to a persistent stratified homotopy type and investigate the properties of such a construction.
    In this sense, it is partially intended as a proof of concept, leaving room for many improvements and generalizations at several steps of the pipeline for future works.
    While some of the steps are fairly easy to replicate in a multi-strata scenario, this is not the case for all of them. In particular, learning stratifications from non-stratified data gets significantly harder when the underlying stratification poset is more complicated (this is due to examples such as the Whitney Umbrella, see for example \cite{helmer2021conormal} or \cite[p. 128-129]{banagl2007stratinvar}). Note that our results on convergence of magnifications to tangent cones are proven in the setting of more than two strata (\cref{prop:loc_sampl_convergence_of_magnifications}), and thus may allow generalizations for appropriate families of functions $\Phi$, for suitably nice spaces, avoiding examples such as the Whitney Umbrella. 
    We are aware that these types of stratification learning questions are currently the research focus of several other groups. 
    \item While it is certainly possible to generalize the definition of the persistent stratified homotopy type to more complicated posets (note that a lot of the abstract homotopy theory is already in place \cite{douteau2019stratified,douteauwaas2021,haine2018homotopy,ayala2018stratified}), this significantly increases the technical complexity of definitions and proofs involved, adding an inductive component. This adds another technical difficulty to an already somewhat lengthy paper, which we wanted to avoid at this point.
    \item In the case of multiple strata, there are several slightly different approaches to what the homotopy category of stratified spaces should be (compare \cite{haine2018homotopy} and \cite{douteauwaas2021}). While all of these approaches agree in the two strata case and on the class of Whitney stratified spaces, the investigation and comparison of the multi-strata case is still the content of ongoing research. New results and insights on the theoretical side may still greatly influence and simplify the transfer to topological data analysis, making it worthwhile to save the development of the multi-strata case for future projects. In particular, having inductive interpretations of the stratified homotopy theories defined in \cite{douteau2019stratified,douteauwaas2021,haine2018homotopy,ayala2018stratified} would significantly streamline the transfer to the setting of topological data analysis.
\end{enumerate}

\section{Stratified homotopy theory}\label{sec:strat_htpy_theo}
In this section, we summarize material concerning stratified spaces and their homotopy theory, as far as it is relevant to our investigations (\cref{subsec:strat_spaces,subsec:htpy_cat,subsec:strat_diag}). The exposition on stratified homotopy theory should be accessible for a reader familiar with basic notions in algebraic topology and category theory. Both for details and the complete model categorical picture we refer the reader to \cite{douteauwaas2021}, which contains a comprehensive overview. For more details on stratified spaces and their invariants consider, for example, \cite{banagl2007stratinvar}.
\subsection{Stratified spaces}\label{subsec:strat_spaces}
We begin by recalling some of the basic notions relevant to the theory of stratified space. Recall that the Alexandrov topology on a poset $P$, is the topology in which the closed sets are the sets which are closed below, under the relation on $P$.
\begin{definition}
A \define{stratified space} (over a poset $\pos$) is a pair $\stratSp = (\topSp, \phiStrat: \topSp \to \pos)$ where $\topSp$ is a topological space and $\phiStrat$ is continuous with respect to the Alexandrov topology on $\pos$. The map $\phiStrat$ is called the \define{stratification of $\stratSp$}. The fiber of the stratification over $p \in \pos$ \[\stratSp_p:=\phiStrat^{-1}\{p\}\] is called the \define{$p$-stratum of $\stratSp$}.
\end{definition}

\begin{example}\label{ex:basic_strat_spaces}
From an abstract point of view, any filtration $(\topSp_{\leq 0} \subset ... \subset \topSp_{\leq n}=X)$ by closed subsets of a topological space $\topSp$ induces a stratification over the poset $[n]= \{0 < ... <n\}$. However, stratified spaces also arise quite naturally in different fields of mathematics, and are often assumed to have manifold strata.
\begin{itemize}
    \item Let $(M, \partial M)$ be a compact manifold with boundary and let $\topSp$ be the space obtained by coning off the boundary of $M$, i.e. $\topSp = M \cup_{\partial M} C(\partial M)$, where \[
    C(Y) = \faktor{Y \times [0,1]}{ (y,0) \sim (y',0)}\] denotes the cone on a space $Y$. One obtains a stratification of $\topSp$ by the map
    \begin{align*}
        s\pp X \to  \{p,q\}; &&
        \begin{cases}
            x \mapsto q, & \textnormal{for }x \in X \setminus \{ \textnormal{cone point} \},\\
            x \mapsto p, & \textnormal{for } x = \textnormal{cone point}.
        \end{cases}    
    \end{align*}
    The resulting stratified space is locally Euclidean away from one isolated singularity, at which arbitrarily small neighborhoods are homeomorphic to the open cone $\mathring C(\partial M) = C(\partial M) \setminus \{1\} \times \partial M$.
    \item Given a smooth manifold $M$ with a compact Lie group $G$ acting smoothly and properly on $M$. The orbit space $M/G$ can then be stratified by orbit types (see, e.g., \cite[Chapter 4]{pflaum2001analytic} for more details).
    \item Any $n$-dimensional complex algebraic variety $\topSp$ can be equipped with the structure of a stratified space. A filtration by closed subsets is given by iteratively taking singular loci.
\end{itemize}

\end{example}
\begin{example}
The so-called pinched torus $PT^2$ can be described as the quotient space of the torus $T^2 = S^1 \times S^1$ by collapsing one circle $* \times S^1$ to a point, see \cref{fig:pinched_torus}. The image of this circle is the singular point, denoted $s$, of the pinched torus. The filtration $\{ s \} \subset PT^2$ induces a stratification over the poset $\{0 < 2\}$. The pinched torus is an example of a so-called \define{pseudomanifold}, an important class of stratified spaces that have been the subject of research to recover a form of generalized Poincar\'e duality for singular spaces (\cite{goresky1980intersection,goresky1983intersection}).
\begin{figure}
    \centering
    \includegraphics[width=0.7\textwidth]{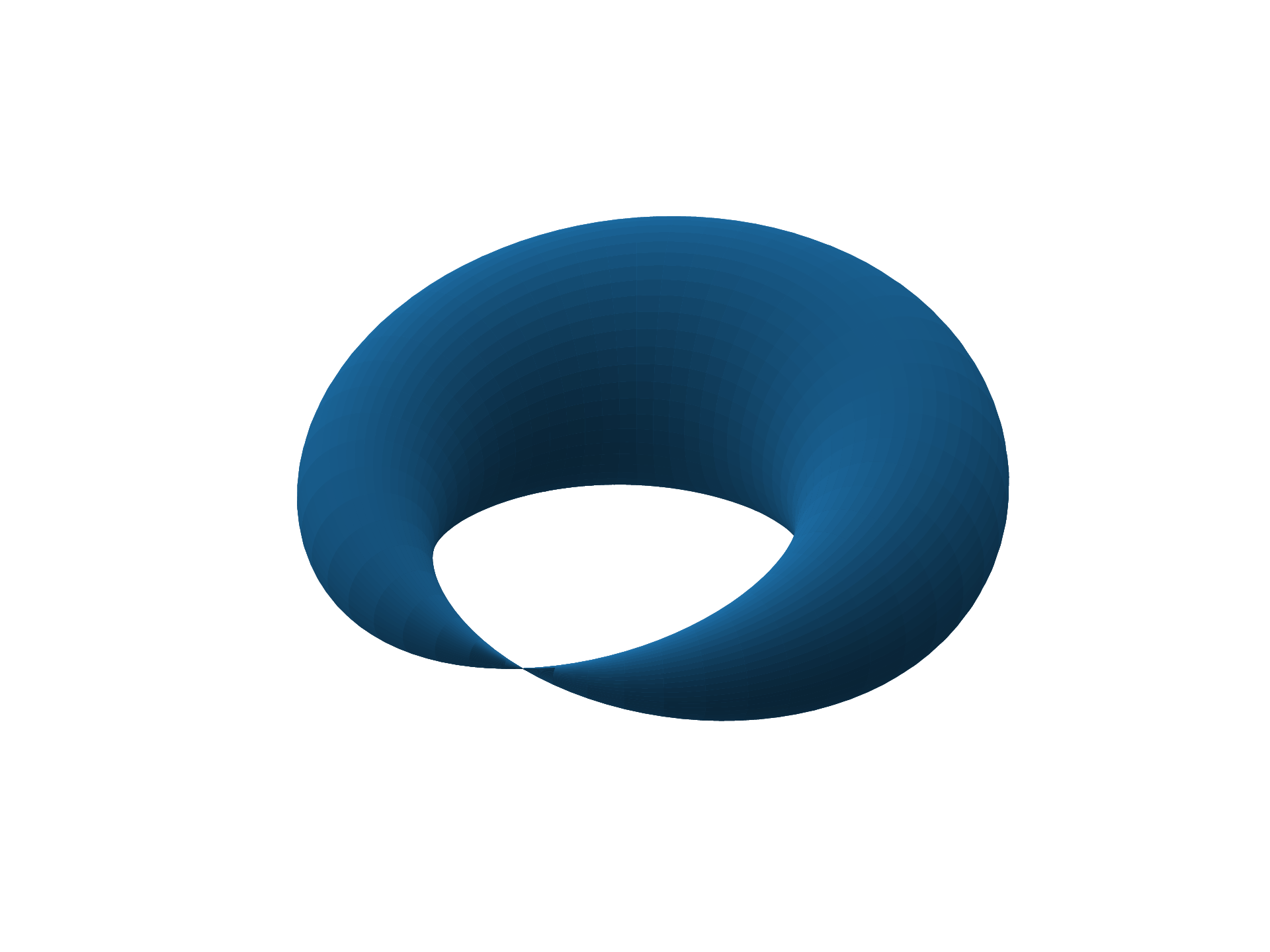}
    \caption{Pinched torus}
    \label{fig:pinched_torus}
\end{figure}
\end{example}

\begin{remark}
It is common to abuse notation insofar as one usually refers to the stratified space by its underlying topological space. Thus, we will freely use notation such as $x \in \stratSp$, when we mean $x \in \topSp$. However, as the second half of this paper is particularly concerned with learning stratifications, we will take care to differentiate rigorously between stratified and non-stratified objects then.
\end{remark}

\begin{definition}\label{def:strat_map}
A \define{stratum preserving map} between two $P$-stratified spaces $\topSp \rightarrow \pos$ and $\topSp[Y] \rightarrow \pos$ is a continuous map $f \colon \topSp \to \topSp[Y]$, making the diagram 

\[
    \begin{tikzcd}
        \topSp \arrow{rr}{f} \arrow{rd} & &
        \topSp[Y] \arrow{ld} 
        \\
       & \pos &
    \end{tikzcd}
\]
commute.
\end{definition}

\begin{notation}\label{not:strat_cat}
Stratified spaces over a poset $\pos$ together with stratum preserving maps define a category which we denote $\TopP$. Isomorphisms in $\TopP$ - i.e. stratum preserving homeomorphisms - will be denoted by $\cong_P$.
\end{notation}

\begin{remark}
There is a slight technical issue here insofar, as the homotopy theoretical perspective needs assumptions on the underlying topological spaces used. We assure the reader unfamiliar with the following technicalities that they can safely ignored them.
We generally denote by $\Top$ the category of $\Delta$-generated spaces, i.e. spaces which have the final topology with respect to maps coming from simplices (see \cite{dug2003delta} for details).
We generally assume all topological spaces involved to have this property. At times, this will mean that the topology on a space has to be slightly modified and replaced by a $\Delta$-generated one (for example $\mathbb{Q} \subset \mathbb R$ is not $\Delta$-generated, its $\Delta$-ification is given by a discrete countable space). However, since this operation does not change weak homotopy types, it is mostly irrelevant to our investigations of homotopy theory (see also \cite[Rem 2.10]{douteauwaas2021}). 
\end{remark}

\begin{notation}
Given a stratified space $\stratSp = (\topSp, \phiStrat: \topSp \to \pos)$ and $p \in \pos$ we write
\begin{align*}
\subLevel{\stratSp}{p} &:= \phiStrat^{-1}(\{q \leq p \}), \\
\subbLevel{\stratSp}{p} &:= \phiStrat^{-1}(\{q < p \}), \\
\supLevel{\stratSp}{p} &:= \phiStrat^{-1}(\{q \geq p \}), \\
\suppLevel{\stratSp}{p} &:= \phiStrat^{-1}(\{q > p\}).
\end{align*}
\end{notation}
For many theoretical as well as for our more applied investigations of stratified spaces, it is fruitful to impose additional regularity assumptions on the strata (such as manifold assumptions) and the way they interact. The notion central to this paper is the notion of a Whitney stratified space. These are characterized by the convergence behavior of secant lines around singularities. 
\begin{notation}\label{not:secant}
Given two distinct vectors $v,u \in \mathbb \supSp$, with $v \neq u$, we denote by $\secant{v}{u}$ the 1-dimensional subspace of $\supSp$ spanned by $v-u$.
\end{notation}
\begin{recollection}\label{recol:Whitney_strat}
A stratified space $\whs = \stratPair$ with $\topSp \subset \supSp$ locally closed is called \define{Whitney stratified}, if it fulfills the following properties.
\begin{enumerate}
    \item \define{Local finiteness}: Every point $x \in \topSp$ has a neighborhood intersecting only finitely many of the strata of $\whs$. 
     \item \define{Frontier condition:}\label{frontier_cond} $\whs_p$ is dense in $\subLevel{\whs}{p}$, for all $p \in \pos$.
    \item \define{Manifold condition:} $\whs_p$ is a smooth submanifold of $\supSp$, for all $p \in \pos$.
    \item \define{Whitney's condition (b):} Let $p, q \in \pos$ such that $p < q$ and let $x_n$, $y_n$ be sequences in $\whs_q$ and $\whs_{p}$ respectively, both convergent to some $y \in \whs_p$. Furthermore, assume that the secant lines $\secant{x_n}{y_n}$ converge to a $1$-dimensional space $l \subset \supSp$ and that the tangent spaces $\tangentSp[x_n]{\whs_q}$ converge to a linear subspace $\tau \subset \supSp$. Then $l \subset \tau$. (By convergence of vector spaces we mean convergence in the respective Grassmannians.)
\end{enumerate}
\end{recollection}
\begin{example}
Whitney's work (\cite{whitney1965local}, \cite{whitney1965tangent}) states that every algebraic and analytic variety admits a Whitney stratification. More general, Whitney stratifications can even be given to spaces such as semianalytic sets (see e.g. \cite{law1965ensembles}) or o-minimally definable sets (see e.g. \cite{le1998verdier}). Finally, if $\topSp$ is such that it has only isolated singularities and admits a Whitney stratification, then any stratification of $\topSp$, fulfilling frontier and boundary condition, with smooth strata is automatically a Whitney stratification. In particular, any definable set with isolated singularities and a dense open submanifold is canonically Whitney stratified with two strata.
Another class of Whitney stratified spaces arises from $G$-manifolds, already noted in \cref{ex:basic_strat_spaces}. For a proof, see \cite[Theorem 4.3.7]{pflaum2001analytic}.
\end{example}

Whitney's condition (b) has a series of immanent topological consequences, which ultimately led to the more general notion of a conically stratified space. The latter are (with some additional assumptions) one of the main objects of interest in the algebro-topological study of stratified spaces \cite{siebenmann1972deformation,goresky1980intersection,goresky1983intersection,quinn1988homotopically,lurie2012higher}.
In addition to the Whitney stratification assumption, we will frequently need additional control over how pathological the subsets of Euclidean space we allow for can be.
To obtain such additional control, we use the notion of a set $\topSp \subset \supSp$, definable with respect to some o-minimal structure (see \cite{van1998tame} for a definition). For the reader entirely unfamiliar with these notions it suffices to know that all semialgebraic or compact subanalytic sets have this property.
On the one hand, definability assumptions guarantee the existence of certain mapping cylinder neighborhoods (see \cref{ex:cyl_nbhds}) that allow thickenings that do not change the homotopy type (see \cref{lem:appendix_definably_thickenable}). At the same time, asserting additional control over the functions defining a set (polynomially bounded), has several consequences for the convergence behavior of tangent cones, already noted in \cite{hironaka1969normal,bernig2007tangent}. We will use these to recover stratifications from samples in \cref{sec:recover_strat}.
\begin{definition}
We say that a stratified space $\stratSp =  \stratPair$, with $\topSp \subset \supSp$ and $\pos$ finite, is \define{definable} (or \define{definably stratified}) if all of its strata are definable with respect to some fixed o-minimal structure.
\end{definition}
\subsection{Homotopy categories of stratified spaces}\label{subsec:htpy_cat}
Many of the algebraic invariants of stratified spaces - most prominently intersection homology - are invariant under a stratified notion of homotopy equivalence. 
\begin{definition}\label{def:strat_homotopic}
Let $f,f' \colon \stratSp \rightarrow \stratSp'$ be stratum preserving maps. We call $f$ and $f'$ \define{stratified homotopic}, if there exists a stratum preserving \[\htpy \colon (\topSp \times \uInt,\topSp \times \uInt \to \topSp \xrightarrow{\phiStrat} \pos) \rightarrow \stratSp' \] such that $\htpy_{\mid\topSp \times \{0\}} = f$ and $\htpy_{\mid\topSp \times \{1\}} = f'$. Furthermore, $f$ is called a \define{stratified homotopy equivalence}, if there exists another stratum preserving map $g\colon  \stratSp' \to \stratSp$ such that $f\circ g$ and $g \circ f$ are stratified homotopic to $\id[\stratSp']$ and $\id[\stratSp]$ respectively.
\end{definition}

\begin{remark}
Since we use different notions of equivalences of stratified spaces in this paper, we use the convention of speaking of \define{strict stratified homotopy equivalences} instead of stratified homotopy equivalences, to avoid any possibility of confusion. The class of all stratified spaces \define{strictly stratified homotopy equivalent} to a stratified space $\str$ is called the \define{strict stratified homotopy type of $\str$}.
\end{remark}
The use of strict stratified homotopy equivalence for topological data analysis faces one apparent issue.
Many of the justifications for the use of persistent approaches to the analysis of geometrical data rely on the fact that homotopy types of (sufficiently regular) spaces do not change under small thickenings (see for example \cite{niyogi2008finding}).
Unlike classical homotopy equivalence, however, stratified homotopy equivalence is a rather rigid notion.
\begin{example}\label{ex:thickenings_fail}
Consider the space $\topSp = S^1 \vee S^1$ embedded in $\mathbb R^2$ as a curve, shown in \cref{fig:S1vS1}. It features a singular point at the self-crossing. Denote the resulting stratified space over $\pos= \{ 0 <1 \}$ with the singularity sent to $0$ and the remainder to $1$ by $\str$.
While there generally seems to be no canonical way to thicken such a space, one possibility is to thicken both the total space as well as the singularity as in \cref{fig:thickening2}. The resulting thickened space $\str''$ is strictly stratified homotopy equivalent to the original curve with the singular stratum extended from a point to the crossing, denoted $\str'$, see \cref{fig:thickening1}.
However, $\str$ and $\str'$ (and hence $\str''$) are not strictly stratified homotopy equivalent. To see this, note that a stratified homotopy equivalence between $\str$ and $\str'$ would also have to be a homotopy equivalence of the underlying spaces. Such a map has to send a circle $S^1$ with degree $\pm 1$ onto another circle. But the image of any stratum preserving map between $\str$ and $\str'$ is (non-stratifiedly) contractible.
\begin{figure}
\centering
    \begin{minipage}{0.32\textwidth}
        \centering
        \includegraphics[width=1\textwidth,trim={3.5cm 3.5cm 3.5cm 3.5cm},clip]{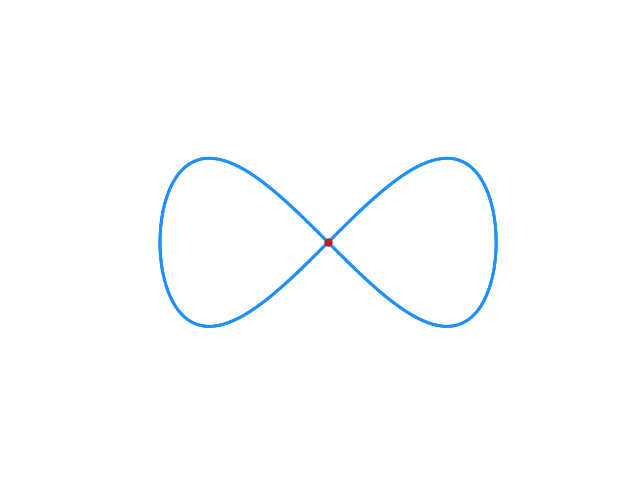}
\caption{Stratified singular curve, $\str$}
    \label{fig:S1vS1}
    \end{minipage}\hfill
    \begin{minipage}{0.32\textwidth}
        \centering
        \includegraphics[width=1\textwidth,trim={3.5cm 3.5cm 3.5cm 3.5cm},clip]{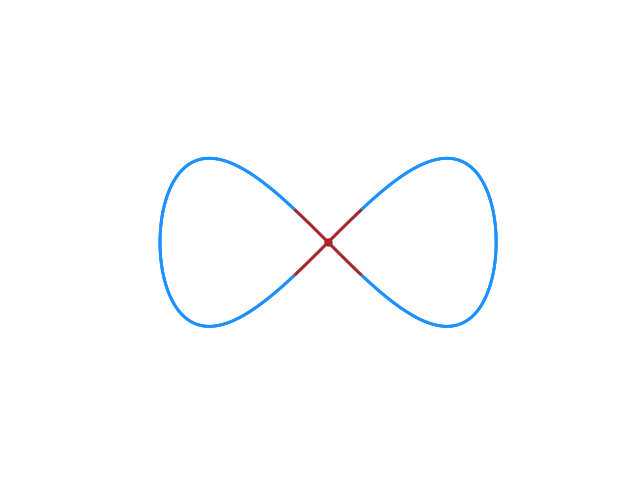}
    \caption{Alternative stratification, $\str'$}
    \label{fig:thickening1}
    \end{minipage}\hfill
    \begin{minipage}{0.32\textwidth}
    \centering
        \includegraphics[width=1\textwidth,trim={3.5cm 3.5cm 3.5cm 3.5cm},clip]{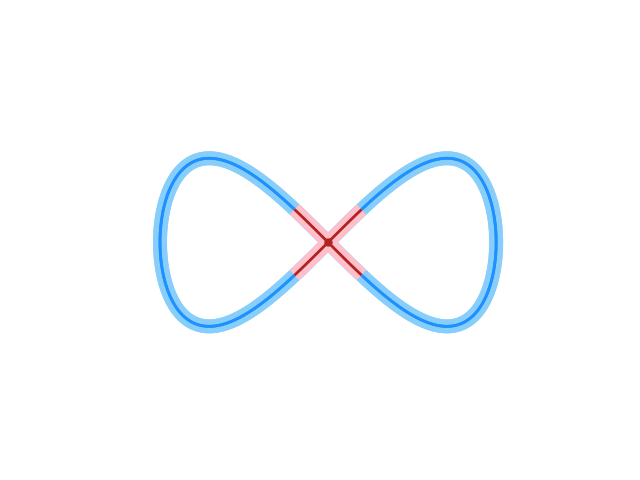}
    \caption{Stratified thickening, $\str''$}
    \label{fig:thickening2}
    \end{minipage}
\end{figure}
\end{example}
In some sense, the failure of stratified homotopy equivalence in \cref{ex:thickenings_fail} is due to the fact that the two thickenings are not sufficiently regular (i.e. Whitney stratified, or more generally conically stratified in the sense of \cite{lurie2012higher}) spaces anymore (this will become more apparent later on from \cref{thrm:fully_faithful_embedding} and \cref{fig:unclear_thickening}). Here, we already encounter the issue that to perform topological data analysis on \textit{nicely} stratified spaces, one generally needs to leave the \textit{nice} category. To make the intuition of why this phenomenon leads to the failure of stratified homotopy equivalence in \cref{ex:thickenings_fail} more rigorous, we need the notion of a homotopy link. These were first introduced in \cite{quinn1988homotopically} and can be thought of as a homotopy theoretical analog of the boundary of a regular neighborhood in the piecewise linear scenario. See also \cite{douteauwaas2021} for more geometrical intuitions.
\begin{definition}\label{def:holink}
Let $\stratSp$ be a stratified space and $p ,q \in \pos$ with $p<q$. The \define{homotopy link} of the $p$-stratum in the $q$-stratum is the space of so-called \define{exit paths}
\[
\stratlink( \stratSp) = \{ \gamma\colon \uInt \rightarrow \topSp \mid \gamma(0) \in \stratSp_p, \gamma(t) \in \stratSp_q, \forall t > 0 \}
\]
with its topology induced by $\stratlink(\stratSp) \subset \cont{0}{\uInt}{\topSp}{}$, where the latter denotes the space of continuous functions equipped with the compact open topology. The induced functors
\[
\TopP \to \Top 
\]
come with natural transformations 
\[
\str_p \leftarrow \stratlink( \stratSp) \to \str_q,
\]
given by the starting point and end point evaluation map.
\end{definition}
\begin{example}\label{ex:htpy_links}
Let us return to \cref{ex:thickenings_fail} to give an illustration of the homotopy link. For the original singular curve and both thickenings, the homotopy links are all homotopy equivalent to four isolated points (see \cref{fig:htpy_links}). This can be seen from \cref{con:equivalence_of_link_and_holink}, which states that the homotopy links are homotopy equivalent to the boundary of a cylinder neighborhood of the singular stratum. 
\begin{figure}
\centering
    \begin{minipage}{0.32\textwidth}
        \centering
        \includegraphics[width=1\textwidth,trim={3.5cm 3.5cm 3.5cm 3.5cm},clip]{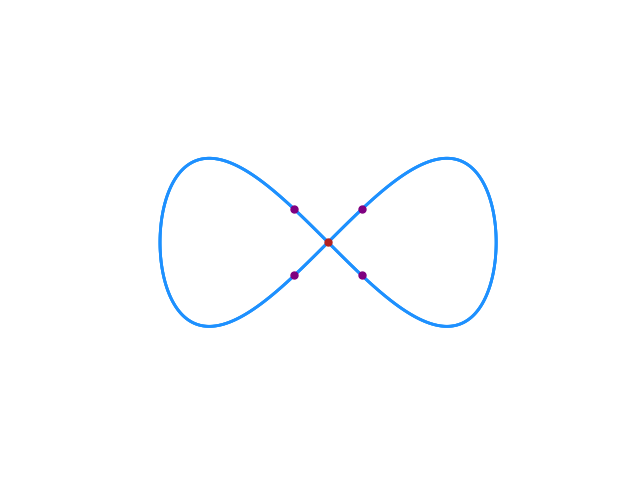}
    \end{minipage}\hfill
    \begin{minipage}{0.32\textwidth}
        \centering
        \includegraphics[width=1\textwidth,trim={3.5cm 3.5cm 3.5cm 3.5cm},clip]{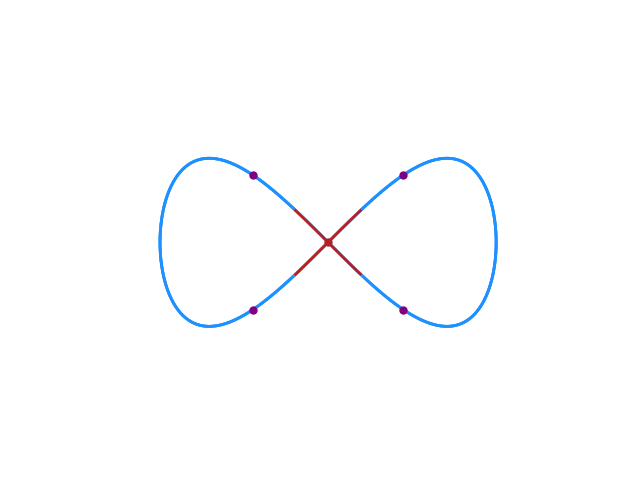}
    \end{minipage}\hfill
    \begin{minipage}{0.32\textwidth}
    \centering
        \includegraphics[width=1\textwidth,trim={3.5cm 3.5cm 3.5cm 3.5cm},clip]{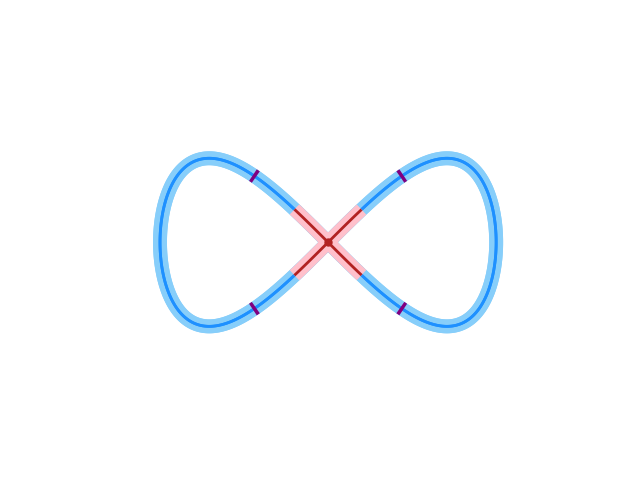}
    \end{minipage}
\caption{Geometric models of homotopy links marked in purple}
\label{fig:htpy_links}
\end{figure}
\end{example}
In \cite[Theorem 6.3]{miller2013strongly}, it was first shown that a stratum preserving between sufficiently regular stratified spaces is a stratified homotopy equivalence, if and only if it induces homotopy equivalences on all homotopy links and strata. This behavior is akin to the one described by the classical Whitehead theorem (see \cite{whitehead1949combinatorial1}, \cite{whitehead1949combinatorial2}) or more generally the behavior of cofibrant, fibrant objects in a model category. It is a general paradigm in abstract homotopy theory that to study a class of in some sense regular objects within a larger class of objects, up to a notion of equivalence, it can be useful to weaken that notion in a way, that it becomes less rigid on the whole class, but still agrees with the original notion on the class of regular objects. This is also the perspective on stratified homotopy theory that we take here that also allows us to circumvent the issue alluded to in \cref{ex:thickenings_fail}.
\begin{recollection}\label{recol:general_holink}
The definition of a homotopy link for pairs $\{ p < q\}$ generalizes to the case where $\{ p < q\}$ is replaced by a regular, i.e. strictly increasing, flag $\flagI = \{ p_0 < ... < p_n \}$. The resulting spaces are denoted
\[
\stratlink[\flagI](\stratSp).
\]
One then needs to replace the stratified interval $\uInt$ by a stratified simplex corresponding to $\flagI$. In the case of $\flagI= \{ p \}$ a singleton, this definition comes down to
\[
\stratlink[\flagI](\stratSp) = \stratSp_p.
\]
Since we are mainly concerned with the two strata case here, we refer the interested reader to \cite{douteauwaas2021} for rigorous definitions. 
\end{recollection}
\begin{definition}\label{def:weak_equiv}
A stratum preserving map $f\colon \stratSp \to \stratSp'$ in $\TopP$ is called a \define{weak equivalence of stratified spaces}, if it induces weak equivalences of topological spaces
\[
\stratlink[\flagI](\stratSp) \to \stratlink[\flagI](\stratSp'),
\]
for all regular flags $\flagI \subset \pos$. 
\end{definition}
\begin{notation}\label{not:localized_hoTop}
We denote by $\ho \TopP$ the category obtained by localizing $\TopP$ at the class of weak equivalences. The isomorphism class of $\str \in \ho \TopP$ is called the \define{stratified homotopy type of $\str$}. Isomorphisms in $\ho \Top_P$ will be denoted by $\simeq_P$.
\end{notation}
It is an immediate consequence of the fact that homotopy links map stratified homotopy equivalences to homotopy equivalences that any strict stratified homotopy equivalence is also a weak equivalence of stratified spaces. The converse is generally false.
\begin{example}\label{ex:weak_equiv}
Let us illustrate these concepts for the spaces from \cref{ex:thickenings_fail} where we already discussed that there is no strict stratified homotopy equivalence between the original curve and any of the described thickenings. However, all the spaces are weakly stratified homotopy equivalent. Indeed, this is already hinted at by the fact that we may find a homotopy equivalence between the respective regular and singular parts as well as the homotopy links as described in \cref{ex:htpy_links}. Consider \cref{fig:weak_equiv} for an illustration.
\begin{figure}
    \centering
    \includegraphics[width=1\textwidth]{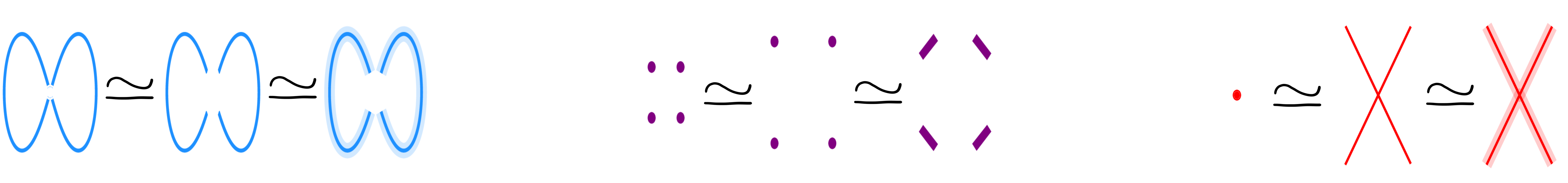}
    \caption{Regular strata, homotopy links and singular strata of the spaces in \cref{ex:thickenings_fail}}
    \label{fig:weak_equiv}
\end{figure}
\end{example}
Miller's result (\cite[Thm. 6.3]{miller2013strongly}) can in fact be strengthened to a fully faithful embedding of homotopy categories. Roughly speaking, a stratified space is called triangulable, if it admits a triangulation compatible with the stratification (for details see \cite{douteauwaas2021}). For the purpose of this paper, it suffices to know that Whitney stratified and (locally compact) definably stratified spaces even admit a PL-structure compatible with the stratification and are thus triangulable, see \cite{goresky1978triangulation}, \cite{shiota2005whitney}, \cite{czapla2012definable}. As a consequence of \cite[Theorem 1.2]{douteauwaas2021}, one then obtains the following result:
\begin{theorem}
{\cite[Theorem 1.2]{douteauwaas2021}} \label{thrm:fully_faithful_embedding}
Let $\WhitP \subset \TopP$ be the full subcategory of Whitney stratified spaces over $\pos$,
and $\simeq$ be the relation of stratified homotopy. Denote by $\quotient{\WhitP}{\simeq}$ the category obtained by identifiying stratified homotopic morphisms in $\WhitP$. Then the induced functor
\[
\quotient{\WhitP}{\simeq} \rightarrow \ho\TopP
\]
is a fully faithful embedding.
\end{theorem}

For our purpose, this result entails that for the study of stratified homotopy invariants of sufficiently regular stratified spaces through topological data analysis, one may as well work in the category $\ho \TopP$. As long as the spaces we investigate have these regularity properties, no information is lost by considering the stratified homotopy type instead of the strict stratified homotopy type. At the same time, \cref{prop:thicken_stab_lok,prop:pers_strat_htpy_type_C-lipschitz,thrm:recovery_thrm} point towards the fact that stratified homotopy types are well suited for applications in topological data analysis, ultimately fulfilling many of the relevant properties of the classical homotopy type.
\subsection{Stratification Diagrams}\label{subsec:strat_diag}
As noted in the previous section, for the passage to a persistent scenario, some notion of thickening of a stratified space is needed. 
In analogy to the classical scenario, this should assign to a stratified space $\str \subset \supSp$, a functor from the category given by the (positive) reals with the usual order $\mathbb R_{+}$ into some category representing stratified homotopy types $\mathcal C$. In the classical scenario, $\mathcal{C}$ is often taken to be the category of simplicial complexes (sets) using constructions such as the \v{C}ech or Vietoris-Rips complex. For now, let us refer to the image under such a functor $\pers{\str}$ as the \define{persistent stratified homotopy type} of $\str$, and similarly to the non-stratified construction using thickenings or \v{C}ech complexes as the \define{persistent homotopy type}.

This leaves us with the following question: How does one thicken a stratified subspace $\str \subset \supSp$ while fulfilling a series of stability and invariance properties that justify the use for topological data analysis (compare with \ref{enum:properties_of_PH1} to \ref{enum:properties_of_PH3}). We explain and show a series of such properties in \cref{sec:pers_strat}.
\begin{example}
In \cref{fig:unclear_thickening} we exhibit three different thickenings of the original space from \cref{ex:thickenings_fail}. The first thickening is neither weakly nor strictly stratified homotopy equivalent to the original curve (as can be seen by comparing homotopy links). The second thickening, being only weakly equivalent to the unthickened space, was discussed in \cref{ex:weak_equiv}. However, note that the inclusion of the original curve into it is not a stratum preserving. Hence, this notion of thickening does not allow for a persistent approach. For the third thickening, the inclusion of the original curve is even a strict stratified homotopy equivalence. However, it seems unclear how to systematically achieve such a thickening, particularly when working with samples.
\begin{figure}
\centering
    \begin{minipage}{0.32\textwidth}
        \centering
        \includegraphics[width=1\textwidth,trim={3.5cm 3.5cm 3.5cm 3.5cm},clip]{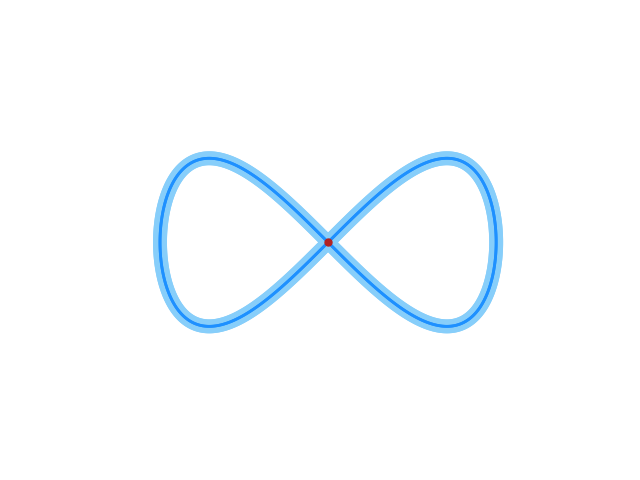}
    \end{minipage}\hfill
    \begin{minipage}{0.32\textwidth}
        \centering
        \includegraphics[width=1\textwidth,trim={3.5cm 3.5cm 3.5cm 3.5cm},clip]{Figures/thickening2.png}
    \end{minipage}\hfill
    \begin{minipage}{0.32\textwidth}
    \centering
        \includegraphics[width=1\textwidth,trim={3.5cm 3.5cm 3.5cm 3.5cm},clip]{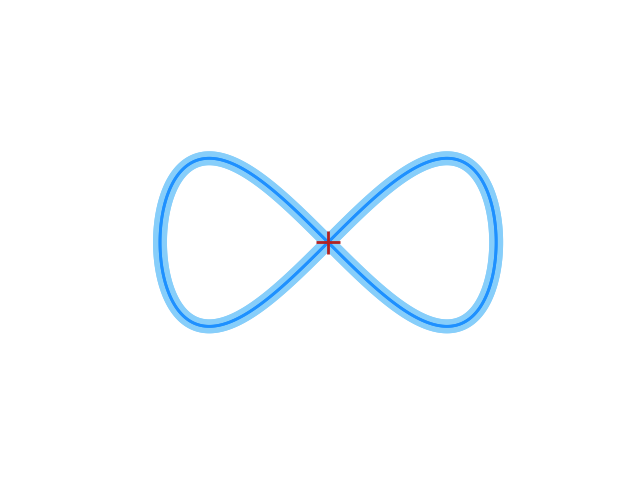}
    \end{minipage}
    \caption{Three possible thickenings}
    \label{fig:unclear_thickening}
\end{figure}
\end{example}
As illustrated in detail in \cref{sec:pers_strat}, thickenings can be done successfully by representing stratified homotopy types by so-called stratification diagrams.
\begin{definition}\label{def:diags}
We denote by $\mathrm{R}(P)$ the category with objects given by regular (i.e. strictly increasing) flags $\flagI = \{ p_0 < \dots < p_k\}$ in $\pos$ and morphisms given by inclusion relations of flags. We denote by \[ \Diag := \mathrm{Fun}(\mathrm{R}(P)^{\mathrm{op}},\Top) \] the category of $\mathrm{R}(P)^{\op}$ indexed diagrams of topological spaces. We call elements of $\Diag$ \define{(stratification) diagrams}.
\end{definition}
\begin{definition}\label{def:weak_eq_diags}
A morphism $f\colon \diag \to \diag'$ in $\Diag$, for which $f_{\flagI}$ is a weak equivalence at all $\flagI \in \mathrm{R}(P)$ is called a \define{weak equivalence of (stratification) diagrams}.
\end{definition}
\begin{notation}
We denote by $\ho \Diag$ the category obtained by localizing $\Diag$ at weak equivalences of diagrams.
\end{notation}
For our purposes, the important result on stratification diagrams is that they can equivalently be used to describe stratified homotopy types. This is due to the following result.
\begin{recollection}\label{recol:diag_are_equ}(For details see \cite{douteau2019stratified, douteauwaas2021}).
   (Generalized) homotopy links induce a functor 
   \begin{align*}
       \diagFun \pp \TopP &\to \Diag \\
   \stratSp &\mapsto \{ \flagI \mapsto \stratlink[\flagI]{(\stratSp)} \}.
   \end{align*}
    By definition, a stratum preserving map is a weak equivalence, if and only if its image under $\diagFun$ is a weak equivalence. In particular, one obtains an induced functor
    \[
    \diagFun\pp \ho \TopP \to \ho\Diag
    \]
    which turns out to be an equivalence of categories. In this sense, the stratification diagram encodes the same homotopy theoretic information as the original space. We will use this equivalence to identify these two homotopy categories and often not distinguish between a stratified space and its stratification diagram.
\end{recollection}
Homotopy links (and thus also stratification diagrams) defined as subspaces of mapping spaces are, at first glance, objects unsuited to a computational or algorithmic approach. To obtain more geometrical and combinatorially interpretable models of the latter, we will also use another equivalent description of stratified homotopy types, which occur naturally, particularly when trying to quantitatively recover stratifications from non-stratified data in \cref{sec:recover_strat}. Since our TDA investigation is mainly concerned with the two strata case, we will only consider $\pos = \{ p < q\}$ for the remainder of this section and only give definitions in this scenario. The relevant observation (see \cite{douteau2019stratified}) is that instead of considering the poset $\pos$ as a space with Alexandrov topology, we may instead consider it as a simplicial complex via its nerve $\nerve (\pos)$ (with vertices the elements of $\pos$ and simplices given by flags) and then consider its realization. In the particular case $\{ p < q\}$, the resulting space is canonically homeomorphic to $[0,1]$, with $p$ corresponding to $0$ and $q$ corresponding to $(0,1]$. This leads to the following definition:
\begin{definition}
A \define{strongly stratified space} (over $\pos= \{p < q\}$) is a pair \[\stratSp = (\topSp, \phiStrat: \topSp \to \uInt)\] where $\topSp$ is a topological space and $\phiStrat$ is continuous.
A \define{strongly stratum preserving map} $f: \str = (\topSp, \phiStrat) \to (\topSp', \phiStrat') = \str'$ is a map of topological spaces $f: \topSp \to \topSp'$ making the diagram
\[
\begin{tikzcd}
X \arrow[rd, "\phiStrat"'] \arrow[rr, "f"] && X' \arrow[ld, "\phiStrat'"] \\
& \uInt &
\end{tikzcd}
\]
commute.
\end{definition}
\begin{remark}
    The name, strongly stratified space $\stratSp = (\topSp, \phiStrat: \topSp \to \uInt)$ relates to the fact, that we may recover a stratified space by postcomposing with the stratification of $[0,1]$ given by 
    \begin{align*}
        [0,1] &\to \{p <q \} \\
        t &\mapsto \begin{cases}
            p & t = 0;\\
            q & t > 0 
        \end{cases}.
    \end{align*}
    In this sense, a strong stratification is a stronger notion than a stratification, which is obtained by storing the additional information of a parametrization of a neighborhood around the singular stratum.
\end{remark}
\begin{notation}
 We denote by $\TopNP$ the category with objects given by strongly stratified spaces and morphisms given by strongly stratum preserving maps. Isomorphisms in this category - i.e. strongly stratum preserving homeomorphisms - will be denoted by $\cong_{\nerve(\pos)}$.
\end{notation}
In a TDA scenario, where one usually works with metric spaces, strong stratifications arise naturally from stratifications.
\begin{example}\label{ex:strat_metric_to_strong_strat_metric}
Let $\str = ( \topSp, \phiStrat)$ be a stratified space equipped with a metric $\dmet{-}{-}$ on $\topSp$. Then, $\str$ can be equipped with the structure of a strongly stratified space, compatible with the original stratification. The strong stratification map is given by the minimum of the distance to the singular stratum function and $1$, i.e. by
\begin{align*}
    \dsingsampX[\str] \pp \topSp &\to [0,1] \\
    x &\mapsto \min \{ \dmet{x}{\str_p}, 1 \}.
\end{align*}
\end{example}
The central examples of particularly well-behaved strongly stratified spaces are those that have the structure of a mapping cylinder close to the singular stratum (see \cref{def:cyl_strat}). The structure of such spaces near the singular stratum is specified by the following example.
\begin{example}\label{ex:def_mapping_cyl_strat}
Given a map of topological spaces $r \pp L \to \topSp$, we can consider the mapping cylinder of $r$ 
\[
\stratCyl := L \times [0,1] \cup_{L \times 0, r} \topSp
\]
equipped with the teardrop topology \cite[Definition 2.1]{quinn1988homotopically} as a strongly stratified space via
\begin{align*}
    \pi_{\uInt}: \stratCyl &\to \uInt \\
    [(x,t)] &\mapsto t.
\end{align*}
Note that if the above $r$ is a proper map between locally compact Hausdorff spaces, then the usual quotient space topology agrees with the teardrop topology on the mapping cylinder \cite{hughes1999stratifications}. When working with metric spaces, there is the following criterion for a map
\[
f:\stratCyl \to Z
\]
into a metric space $Z$ to be continuous. The map $f$ is continuous, if and only if its restrictions to $L \times (0,1]$ and $\topSp$ are continuous, and the family of maps $f(-,t): L \to Z$ with $t >0$, converges uniformly to $f\mid_{X} \circ r$, as $t \to 0$ (consider \cite[Definition 2.1]{quinn1988homotopically}).
\end{example}
Similar to the relation between diagrams and stratified spaces, strongly stratified spaces can also be used to describe stratified homotopy types, as explained in the following recollection.
\begin{recollection}\label{recoll:hoTop}
    We have only described the construction of $\TopNP$ in the case of $\pos = \{ p < q \}$ here. For the more general case see \cite{douteau2019stratified,douteauwaas2021}. 
   Similarly to the stratified case, the strongly stratified category can be equipped with a notion of weak equivalence, leading to a homotopy category $\ho \TopNP$. The forgetful functor
   \[
   \TopNP \to  \TopP ,
   \]
   obtained by post composing the strong stratification with the stratification of the interval \[\uInt \to \{ p < q\}\] given by taking $0$ as the $p$-stratum, then (by passing to derived functors with respect to the model structures explained in \cite{douteau2021homotopy}) induces an equivalence of homotopy categories
   \[
    \ho\TopNP \to \ho\TopP.
   \]
   We will often treat strongly stratified spaces as stratified spaces under this forgetful functor. The equivalence of homotopy categories guarantees that no homotopy theoretical information is lost.
   
   We will not be making mathematical use of this result here. Nevertheless, it conceptually explains the multiple occurrences of strongly stratified spaces in our investigations of strongly stratified homotopy types. 
\end{recollection}
As in the stratified scenario we make frequent use of some short notation to access the analogs of strata in the strongly stratified case.
\begin{notation}\label{not:strong_levelsets}
Let $\str$ be a strongly stratified space and $v' \leq v \in [0,1]$. We use the following notation:
\begin{align*}
\str_v&:= \phiStrat^{-1}\{v\}, \\
\subLevel{\stratSp}{v} &:= \phiStrat^{-1}[0,v], \\
\supLevel{\stratSp}{v'} &:= \phiStrat^{-1}[v',1], \\
\supbLevel{\str}{v'}{v} &:= \phiStrat^{-1}[v',v].
\end{align*}
For values of $v,v'$ outside of $\uInt$ we define these as above, using the closest allowable value.
\end{notation}
It turns out that for particularly nice strongly stratified spaces, these sub- and superlevel sets can be used to recover the stratification diagram, cf. \cite{quinn1988homotopically, miller1994expansions, douteauwaas2021}. However, for the sake of completeness, we include details of this behavior with \cref{ex:cyl_nbhds}. As already alluded to above, these are stratified spaces for which the strata have cylinder neighborhoods.
\begin{definition}\label{def:cyl_strat}
We say a stratified space $\stratSp$ over $\pos=\{p <q \}$ is \define{cylindrically stratified}, if there exists a neighborhood $N$ of $\topSp_p$ and a space $L$ and a map of spaces $r\pp L \to \topSp_p$, such that 
\[
N \cong_P \stratCyl,
\]
where $\stratCyl$ denotes the stratified mapping cylinder of $r$ from \cref{ex:def_mapping_cyl_strat}.
We say a strongly stratified space $\str= (\topSp , s\pp \topSp \to [0,1)])$ is \define{cylindrically stratified}, if it is cylindrically stratified as a stratified space and there is a homeomorphism $f \colon s^{-1}(0,1) \xrightarrow{\sim }\str_{\frac{1}{2}} \times (0,1)$, making the diagram
\begin{diagram}
    s^{-1}(0,1) \arrow[rd, "{s\mid_{s^{-1}(0,1)}}"'] \arrow[rr, "\sim"', "f"] & & {\str_{\frac{1}{2}}} \times (0,1) \arrow[ld, "{ \pi_{(0,1)}}"] \\
   & (0,1) &
\end{diagram}
commute (i.e. a strongly stratum preserving homeomorphism, with respect to the strong stratifications induced by $s$ and $\pi_{[0,1]}$.)
\end{definition}
\begin{remark}\label{rem:def_cyl_strat}
Note, that the definition of a cylindrically stratified space in the strong case is slightly weaker than assuming a strongly stratified mapping cylinder neighborhood. We choose this definition for our purposes as it has precisely the same consequences and is much easier to verify.
Nevertheless, for compact $\str$, it follows by an application of the two-out-of-six property, as in \cref{lem:appendix_definably_thickenable}, that the inclusions
\[
\str_p \hookrightarrow \subLevel{\str}{v}
\]
for $0\leq v <1$, are homotopy equivalences.
\end{remark}
\begin{definition}
 A \define{cylindrically stratified metric space} $\stratSp$ over $\pos=\{p <q \}$ is a stratified space equipped with a metric $\dmet{-}{-}$, which is cylindrically stratified when considered as a strongly stratified space, with respect to the strong stratification induced by the metric (\cref{ex:strat_metric_to_strong_strat_metric}).
\end{definition}
It turns out that many of the stratified spaces, that we are interested in, are cylindrically stratified.
\begin{example}\label{ex:cyl_nbhds}
Whitney stratified spaces, equipped with the metric induced by the inclusion into $\supSp$, are cylindrically stratified up to a rescaling. They even admit neighborhoods that are strongly stratum preserving homeomorphic to a strongly stratified mapping cylinder of a fiber bundle (in particular, they are conically stratified). This is a classical result, found for example already in \cite{thom1969ensembles,mather1970notes}. We sketch a proof here for the sake of completeness.\\
Let $\stratSp= ( \topSp, \topSp \to \{ p <q \})$ be a Whitney stratified space with $\stratSp_p$ compact and $\topSp \subset \supSp$. 
By passing to a sufficiently small neighborhood of $\stratSp_p$ we may assume $X$ to lie in a (standard) tubular neighborhood $N$ of $\str_p$ in $\supSp$, such that the retraction map 
\begin{align*}
    r: N &\to \stratSp_p \\
x &\mapsto y_m,  
\end{align*}
where $y_m$ minimizes $\dmet{x}{y}$, is well defined and smooth. 
Next, consider the distance to $\stratSp_p$ map
\begin{align*}
    \dsingsampX[\stratSp] \pp \topSp &\to \mathbb R \\
        x &\mapsto \dmet{x}{\stratSp_p}.
\end{align*}
It is then a consequence of Thom's first isotopy lemma (which in this two strata case amounts to Ehresmann's Lemma \cite{ehresmann1951fibration}, see e.g.\cite{thom1969ensembles} and \cite[Thm. 6.7]{banagl2007stratinvar} for a modern source), that the map
\begin{align*}
    \topSp \cap N &\to \str_p \times \mathbb R \\
    x &\mapsto (r(x), \dsingsampX[\stratSp](x))
\end{align*}
restricts to a fiber bundle over $\str_p \times (0, \varepsilon]$, for $\varepsilon$ small enough.
If we denote by $\neiEps[\varepsilon]{\stratSp_p}$ a closed $\varepsilon$-neighborhood of $\stratSp_p$ in $X$ and  set $L= \dsingsampX[\stratSp]^{-1}(\varepsilon)$ this means that there is a homeomorphism
\[
f: \neiEps[\varepsilon]{\stratSp_p} \setminus \stratSp_p \xrightarrow{\sim} L \times (0, \varepsilon] 
\]
such that the diagram
\begin{equation}\label{diag:much_commute}
\begin{tikzcd}
\neiEps[\varepsilon]{\stratSp_p} \setminus \stratSp_p \arrow{rr}{f} \arrow[swap]{rd}{r \times \dsingsampX[\stratSp]} && L \times (0, \varepsilon] \arrow{ld}{r \times \pi_{(0, \varepsilon]}}\\
&\str_p \times (0, \varepsilon] &
\end{tikzcd}
\end{equation}
commutes.
By rescaling, we may assume without loss of generality that $\varepsilon=1$ and let $N = \neiEps[1]{\stratSp_p}$, the closed neighborhood of points with distance $\leq 1$ to $\stratSp_p$.\\
Now, consider the map 
\begin{align*}
g\pp \stratCyl \to  N; &&
\begin{cases}
(x,t) \mapsto f(x,t), & \textnormal{for }t > 0,\\
[(x,0)]=[y] \mapsto r(x)=y, & \textnormal{for }t=0.
\end{cases}    
\end{align*}
$g$ is clearly bijective and continuous on $\stratSp_p$ and $L \times (0,1]$. Furthermore, by the commutativity of Diagram \ref{diag:much_commute}, for $t \to 0$, $f(-,t): L \to N$ converges uniformly to $r\mid_L$. By the alternative characterization of the mapping cylinder topology in \cref{ex:def_mapping_cyl_strat}, it follows that $g$ is a continuous bijection, from a compactum to a Hausdorff space, and thus a homeomorphism. 
\end{example}
\begin{example}\label{ex:def_are_cylind_strat}
Compact definably stratified spaces $\str = (\topSp, s:\topSp \to \{ p < q\})$, are (up to a rescaling) cylindrically stratified. Indeed, note first that they are cylindrically stratified as topological spaces. This follows from the fact that they are triangulable in a way that is compatible with the strata (see \cite{van1998tame}). In particular, $\str_p$ always admits a mapping cylinder neighborhood given by a regular neighborhood in the piecewise linear sense. Furthermore, note that the map
\[
 \mathrm{d}_{\str_p}:\topSp \to \mathbb R
\]
again is definable. Thus, by Hardt's Theorem for definable sets (see \cite{van1998tame}), it restricts to a trivial fiber bundle over $(0, \varepsilon]$, for $\varepsilon$ sufficiently small. In other words, after rescaling, we indeed have a homeomorphism 
\[ 
\mathrm{d}_{\str_p}^{-1}(0,1) \to \str_{\frac{1}{2}} \times (0,1).
\]
over $(0,1)$.
\end{example}
\begin{remark}
We will generally consider all compact definably or Whitney stratified spaces to be appropriately rescaled, such that they are cylindrically stratified. Similar assumptions will be made for definably stratified spaces when using \cref{lem:appendix_definably_thickenable}.
\end{remark}
Finally, the following construction, together with \cref{prop:holinks_are_links}, tells us that stratification diagrams of cylindrically stratified spaces have more interpretable geometric models, usable for TDA.
\begin{construction}\label{con:equivalence_of_link_and_holink}
Given a stratified mapping cylinder $\stratCyl$ for $r: L \to Y$ a map of metrizable spaces, we may consider the map
\begin{align*}
   \alpha \pp L &\to \stratlink{\stratCyl} \\
    x &\mapsto \{ t \mapsto [(x,t)]\},
\end{align*}
mapping a point $x$ to the corresponding line segment in $\stratCyl$. A homotopy inverse to this map is given by 
\begin{align*}
   \beta \pp \stratlink{\stratCyl} &\to L\\
    \gamma &\mapsto \pi_{L}(\gamma(1)).
\end{align*}
Clearly, $\beta \circ \alpha =1_{L}$. A homotopy $\alpha \circ \beta \simeq 1_{\stratlink{\stratCyl{r}} }$ is given by
\begin{align*}
    \stratlink{\stratCyl} \times \uInt  &\to \stratlink{\stratCyl} \\
    (\gamma, s) \mapsto \{ t &\mapsto ( \pi_L( \gamma(s + (1-s)t ), t).
\end{align*}
Compare \cite{douteauwaas2021}, \cite{friedman2003stratified} and \cite{quinn1988homotopically} for similar, more detailed arguments covering the continuity of such maps. Now, if $\str$ is a metrizable, cylindrically stratified space  over $\pos = \{p < q\}$ and $N\cong \stratCyl$ is a stratified mapping cylinder neighborhood of $\stratSp_p$ with boundary $L$, then the inclusion
\[
\stratlink{N} \hookrightarrow \stratlink{\stratSp}
\]
is a (weak) homotopy equivalence. Essentially, the idea of the proof is to continuously retract paths in $\stratSp$ into $N$ (see \cite[Appendix]{friedman2003stratified} for details under slightly stronger assumptions). In particular, we then have a commutative diagram 
\[
\begin{tikzcd}
\str_q  \arrow[rr, equal] && \str_q \\
L\times \{v\} \arrow[u, hook] \arrow[d, "r"] \arrow[r, hook, "\simeq"]& \stratlink{(N)} \arrow[r, hook, "\simeq"] & \stratlink{(\str)} \arrow[d], \arrow[u] \\
\str_p \arrow[rr, equal]&& \str_p  \spacecomma
\end{tikzcd}
\]
for $v \in (0,1]$
\end{construction}
\begin{proposition}\label{prop:holinks_are_links}
Let $\str$ be a compact, cylindrically stratified metric space and $\diagParamTup$, such that $0 <\diagParamLow \leq \diagParamUp<1$. Then there is an isomorphism in $\ho \Diag$
\[
\{ \subLevel{\str}{\diagParamUp} \hookleftarrow \supbLevel{\str}{\diagParamLow}{\diagParamUp} \hookrightarrow  \supLevel{\str}{\diagParamLow}  \} \simeq \diagFun(\str).
\]
\end{proposition}
\begin{proof}
Let $s$ be the strong stratification induced by the metric on $\stratSp$.
By assumption, $\stratSp$ admits a mapping cylinder neighborhood $N \cong \stratCyl$, for some map $r\pp L \to \str_p$. 
Denote $\tilde{s}\pp N \to [0,1]$, the alternative strong stratification induced by this choice of mapping cylinder neighborhood.
Since we assume that $\str_p$ is compact, we may assume, without loss of generality, that $N \subset s^{-1}[0,1)$.
By \cref{con:equivalence_of_link_and_holink}  (using the same notation), it suffices to expose a (canonical) zigzag of weak equivalence to the diagram
\[
\{ \str_p \leftarrow L \times \{v\} \hookrightarrow \str_q \},
\]
for some $v \in (0,1]$. Such a zigzag between diagrams is given as follows:
\begin{diagram}
\{ \subLevel{\str}{\diagParamUp} &  \supbLevel{\str}{\diagParamLow}{\diagParamUp} \arrow[r, hook] \arrow[d, "\simeq"] \arrow[l, hook']& \supLevel{\str}{\diagParamLow} \}\\
\{ s^{-1}[0,1) &  s^{-1}(0,1) \arrow[r, hook] \arrow[l, hook']& s^{-1}(0,1] 
\}  \\
\{ s^{-1}[0,1) &  L \times \{v\} \arrow[r, hook] \arrow[u, "\simeq"]\arrow[l, hook'] \arrow[d, "\simeq"]& s^{-1}(0,1] = S_q \} \\
\{ \str_p &  L \arrow[r, hook] \arrow[l, "r"]& S_q \}.\\
\end{diagram}
We describe the morphisms of diagrams from top to bottom, and show that they are weak equivalences.
The first morphism is given by inclusions. Since we have assumed that $s^{-1}(0,1)$ has the shape of an open cylinder $L' \times (0,1)$, this morphism is clearly given by weak equivalences at $\{p\}, \{q\}$ and $\{p < q\}$.
The next morphism is again given by inclusions. To see that it is a weak equivalence, we need to show that \[ L \times \{v\} \hookrightarrow \tilde{s}^{-1}(0,1] \hookrightarrow s^{-1}(0,1) \cong L' \times (0,1)\] is a weak equivalence. 
Since the first of these maps is a weak equivalence, it suffices to show that the second is one too. Since $\str_p$ is compact we find $\varepsilon, \varepsilon' >0$ such that 
\[
\tilde{s}^{-1}(0,\varepsilon') \subset s^{-1}(0,\varepsilon) \subset \tilde{s}^{-1}(0,1) \subset s^{-1}(0,1).
\]
Now, note that since all sets involved are given by open cylinders (on $L$ and $L'$ respectively), these inclusions fulfill the requirements for the two-out-of-six property of homotopy equivalences. In particular, all maps involved are weak equivalences (even homotopy equivalences).
Finally, the last morphism is constructed as follows. Both at $q$ and $\{p <q \}$ it is given by the identity. Assume that $v \in (0,1]$ was taken such that \[L \times \{v\} \cong \tilde{s}^{-1}\{v\} \subset s^{-1}(0,\varepsilon] \subset \tilde{s}^{-1}(0,1),\] for some $\varepsilon >0$ sufficiently small. Note that this is indeed possible by the compactness of $\str_p$. 
Then, at $\{p \}$  the morphism is given by the composition 
\begin{align*}
    s^{-1}[0,1) & \to s^{-1}[0, \varepsilon] \hookrightarrow N \cong \stratCyl \to \str_p
\end{align*}
where the first of these maps is given by 
\[
(x,t) \mapsto (x, \min\{t,\varepsilon\})
\]
under the identification $s^{-1}(0,1) \cong L' \times (0,1)$. By the assumption that $L \times \{v\}$ maps into $s^{-1}[0,\varepsilon]$, this map induces a morphism of diagrams. It remains to show that it is a weak equivalence. 
By the cylinder structure of $s^{-1}(0,1)$, the first map in this composition is a homotopy equivalence. The same holds true for the second map by a two-out-of-six argument, completely analogous to the one performed when comparing $L$ and $L' \times (0,1)$. Finally, the last map is the retraction of a mapping cylinder and thus also a homotopy equivalence. Combining this, we have shown that the final morphism is also a weak equivalence of diagrams.
\end{proof}

\section{Persistent stratified homotopy types}\label{sec:pers_strat}
In this section, we introduce the notion of persistent stratified homotopy type (\cref{subsec:def_pers_strat}) and investigate its stability properties (\cref{subsec:stab_pers_type}), in the particular case of Whitney stratified spaces with two strata (\cref{subsec:stab_pers_type_whit}).
Before we focus on the specific case of persistent \textit{stratified} homotopy types, let us first make a little more precise what we mean by
persistent homotopy types and their role in topological data analysis.
As already alluded to, for the case of persistent homology in the introduction, the perspective we take here on persistent approaches to TDA is that they can usually be decomposed into a two-step process. Conceptually, it can be described as follows: \\
\\
Let $\textbf{S}$ denote some categories of objects representing datasets which contain geometrical information. For example, we can take the category given by all subsets of $\mathbb R^N$, with morphisms given by inclusions. Let $\textbf{T}$ denote some category of geometrical and or combinatorial objects, for example the categories $\Top$, $\textbf{sSet}$ (the category of simplicial sets), $\TopP$, $\Diag$, equipped with a class of morphisms called weak equivalences. Let $\textbf{A}$ denote some category of algebraic objects, for example, the category of vector spaces over some fixed field. Finally, let
\[
H \pp \textbf{T} \to \textbf{A}
\]
be a functor that sends weak equivalences to isomorphisms, for example for $\textbf{T} = \Top$ this could be a homology functor.
We denote $\mathbb R_+$ the category given by  the poset of nonnegative reals.
\begin{notation}\label{not:pers_H}
Given any category $\mathcal C$ and another category $U$ (most prominently $\mathbb R_+$). We denote by $\mathcal C^{U}$ the category of functors from $U$ to $\mathcal C$, with morphisms given by natural transformations.
\end{notation}
Then, a persistent version of $H$ is constructed by taking a composition
\begin{align*}
    \textbf{S} \xrightarrow{\mathcal P} \textbf{T}^{\mathbb{R}_+} \xrightarrow{H^{\mathbb{R}_+}} \textbf{A}^{\mathbb{R}_+},
\end{align*}
for some functor $ \textbf{S} \xrightarrow{\mathcal P} \textbf{T}^{\mathbb{R}_+}$ turning objects in $\textbf{S}$ into persistent objects in $\textbf{T}$, i.e. elements of $\textbf{T}^{\mathbb{R}_+}$. Examples of such functors include sending a subspace of $\mathbb R^N$ to the family of its $\varepsilon$ thickenings, filtered \v{C}ech- or Vietoris-Rips complexes. 
Since $H$ sends weak equivalences into isomorphisms, we obtain a factorization
\begin{diagram}
\textbf{S} \arrow[r] \arrow[rd, "\mathcal P"] \arrow[rr, "PH", bend left = 60]& \textbf{T}^{\mathbb{R}_+} \arrow[r, "H^{\mathbb R_+}"] \arrow[d]&\textbf{A}^{\mathbb{R}_+} \\
& \ho \textbf{T}^{\mathbb{R}_+} \arrow[ru, dashed] &
\end{diagram}
\begin{notation}\label{not:hoTR}
All throughout this paper, $\ho \textbf{T}^{\mathbb R_+}$ denotes the category obtained by localizing $\textbf{T}^{\mathbb{R}_+}$ at such natural transformations, which are given by a weak equivalence at each $\alpha \in \mathbb R_+$. Such natural transformations will be called weak equivalences of functors. We will also refer to isomorphisms in the homotopy category $\ho \textbf{T}^{\mathbb{R}_+}$ (which are always represented by zigzags of weak equivalences in $\textbf{T}^{\mathbb{R}_+}$) as weak equivalences. The same notation and language is used when $\mathbb R_+$ is replaced by a more general indexing category. (The reader wondering about the relation of this construction to taking the homotopy category first and then passing to persistent objects is referred to \cref{rem:order_ho_fun}.) 
\end{notation}
 The functor $\mathcal P$ can then be understood as assigning to an object in $\textbf{S}$ a persistent homotopy type. By a slight abuse of language, we thus refer to $\mathcal P(\sampX)$ as the \define{persistent homotopy type of } $\sampX \in \textbf{S}$, even though it depends on a choice of functor $\textbf{S} \to \textbf{T}^{\mathbb{R}_+}$. (Note, however, that this abuse of language is no more incorrect than speaking of \textit{the persistent homology}, which also depends on choices such as \v{C}ech- or Vietoris-Rips complexes.) Then, as we illustrated in the introduction to the scenario of classical persistent homology, many properties of the functor $PH$ may already be understood by studying the functor $\mathcal P$. At the same time, the advantage of such a modular approach is that it quickly allows obtaining results for all sorts of choices of $H$.
\subsection{Definition of the persistent stratified homotopy type}\label{subsec:def_pers_strat}
Let us now move to the specific case of persistent stratified homotopy types. The goal is to expose a functor $\epers$ with values in the category $\ho \TopP^{\mathbb R_+}$. Furthermore, for the values of such functor to deserve the name, \textit{persistent stratified homotopy types}, we need to show it fulfills properties analogous to the properties \ref{enum:properties_of_PH1}, \ref{enum:properties_of_PH2} and \ref{enum:properties_of_PH3} of the non-stratified persistent homotopy type. 
Let us summarize the pipeline suggested by \cref{subsec:htpy_cat,subsec:strat_diag}.\\
We start with a Whitney stratified or definably stratified space $\str \subset \supSp$ (or later on a sample of any of the latter) aiming to obtain a persistent version of its stratified homotopy type. By \cref{recol:diag_are_equ} the stratified homotopy type of $\str$ is equivalently described by its stratification diagram $\diagFun{(\str)}$. It is an immediate consequence of \cite[Thm. 2.12]{douteau2021homotopy} (which is a stronger version of \cref{recol:diag_are_equ}) that the homotopy link functor from \cref{recol:diag_are_equ} sending a stratified space to a stratification diagram, induces an equivalence of categories 
\begin{diagram}
\ho \TopP^{\mathbb R_+} \xrightarrow{\sim} \ho \Diag ^{\mathbb R_+}.
\end{diagram}
Hence, we may equivalently expose a functor $\epers$ valued in $\ho \Diag ^{\mathbb R_+}$.
To do so, we need to obtain a geometric description of the stratification diagram $\diagFun(S)$, which admits thickenings. By \cref{ex:cyl_nbhds,ex:def_are_cylind_strat}, $\str$ naturally admits the structure of a cylindrically stratified space (up to a rescaling). Then, by \cref{prop:holinks_are_links}, up to a weak equivalence we may recover the stratification diagram $\diagFun{(\str)}$ of $\str$ by the diagram of sub- and superlevel sets
\[
 \subLevel{\str}{\diagParamUp} \hookleftarrow \supbLevel{\str}{\diagParamLow}{\diagParamUp} \hookrightarrow  \supLevel{\str}{\diagParamLow}. 
\]
The diagram obtained in this fashion has the advantage that it admits an obvious notion of thickening.
One simply thickens the parts of the diagram contained in $\supSp$ separately. \\
\\
Let us now set up the framework to analyze the stability properties of these constructions and their interactions with sampling. For the remainder of this subsection let $\pos= \{ p <q \}$ be a poset with two elements.
Next, we define a series of spaces of subspaces of $\supSp$. One should mainly think of elements of these spaces as samples, sampled nearby a continuous space, whose convergence behavior we are investigating. Nevertheless, in the generality below all sorts of complicated, non-finite sets are permitted. We will follow the naming convention of using blackboard bold letters like $\sampX$, when suggesting an object conceptually takes the role of a sample. We use usual letters, like $\topSp$, when an objects takes the role of a space nearby which samples are taken. Of course, this convention does not apply to the fields $\mathbb Q, \mathbb R$ and $\mathbb C$.
\begin{remark}
Throughout the remainder of this paper, we will be defining several distances on categories and spaces whose elements are themselves some version of spaces. In this context, the term metric will mean \define{symmetric Lawvere metric}, that is, we allow for the value $\infty$, and do not require nonidentical elements to have positive distance. 
\end{remark}
\begin{notation}\label{not:thickening}
    Throughout the remainder of this paper, $\norm{-}$ always denotes the Euclidean norm on $\supSp$.
    Given a subset $\sampX \subset \supSp$, and some non-negative real number $\alpha  \geq 0$, we denote by $\thicken{\sampX}$ the $\alpha$ thickening of $\sampX$, given by
    \[
    \{ y \in \supSp \mid \exists x \in \sampX: \norm{x-y}  \leq \alpha \}.
    \]
    We will take care, to always use greek letters for thickenings, as to avoid any possible confusion with the indexing letters denoting strata.
\end{notation}
\begin{definition}\label{def:sam}
Denote by $\Sam$ the space of subspaces of $\supSp$, $\{\sampX \subset \supSp \}$, equipped with the (extended pseudo) metric given by the Hausdorff-distance. That is, for $\sampX, \sampX' \in \Sam$, we set 
\[
\dmet[\HD]{\sampX}{\sampX'} = \inf\{\alpha > 0 \mid \sampX \subset \thicken{\sampX'},\sampX' \subset \thicken{\sampX} \}.
\]
We refer to $\Sam$ as the \define{space of samples} (of $\supSp$). 
\end{definition}
Next, we define a metric on the set of stratified samples.
\begin{definition}\label{def:samP}
Denote by $\SamP$ the space of pairs in $\supSp$,
\[ \{ (\sampX, \singsampX) \mid \singsampX \subset \sampX\} \subset \Sam^2,\] 
equipped with the (extended pseudo) metric induced by the inclusion.\\
That is, for $\sampX, \sampX' \in \Sam_P$, we set 
\[
\dmet[\SamP]{(\sampX, \singsampX)}{(\sampX', \singsampX')} := \max\{ \dmet[\HD]{\sampX}{\sampX'}, \dmet[\HD]{\singsampX}{\singsampX[\sampX']}\}. 
\]
We also refer to $\SamP$ as the \define{space of ($\pos$-)stratified samples} (of $\supSp$).
\end{definition}
We can think of $\SamP$ as a metricized, sample version of the category $\TopP$. 
Next, we need an analogous construction for the category $\Diag$.
\begin{definition}\label{def:DSam}
Denote by $\DSam$ the space of triples of subspaces of $\supSp$ 
\[ \{(\sampD_{p},\sampD_{\{p < q\}},
\sampD_{q}) \mid \sampD_{q} \supset \sampD_{\{p<q\}} \subset \sampD_{p}, \} \subset \Sam^{3},
\] equipped with the subspace metric.
That is, for $\sampD, \sampD' \in \DSam$ we set
\[
\dmet[\DSam]{\sampD}{\sampD'}:= \max_{\flagI \in \mathrm{R}(\pos)}{\dmet[\HD]{\sampD_{\flagI}}{\sampD'_{\flagI}}}.
\]
We also refer to $\DSam$ as the \define{space of stratification diagram samples} (of $\supSp$).
\end{definition}
Finally, we repeat a similar process for $\TopNP$.
\begin{definition}\label{def:sampNP}
Denote by $\SamNP$ the space \[\{(\sampX, \phiStrat[\sampX]\colon \sampX \to [0,1]) \mid \sampX \subset \supSp \},\] equipped with the (extended pseudo) metric given as follows.
Embed $\SamNP$ into the space of subspaces of $\supSp \times [0,1]$, equipped with the Hausdorff distance on the product metric, by sending $\phiStrat$ to its graph. The metric on $\SamNP$ is then given by the subspace metric under this embedding.
\\
Equivalently, this comes down to the following. For $(\sampX,s), (\sampX',s'
)\in \SamNP$,  we set 
\begin{align*}
\dmet[\SamNP]{(\sampX, \phiStratsampX)}{(\sampX', \phiStratsampX')} := \max_{\sampX_0, \sampX_1 \in \{\sampX, \sampX'\}^2} \{
\inf \{ &\varepsilon>0 \mid  \forall x \in \sampX_0 \exists y \in \sampX_1:\\
&\norm{x-y} ,\vert \phiStratsampX_0(x) - \phiStratsampX_1(y)\vert  \leq \varepsilon \}\}.
\end{align*}
We also refer to $\SamNP$ as the \define{space of strongly ($P$-)stratified samples} (of $\supSp$).
\end{definition}
\begin{notation}
For the remainder of this work, we denote $\diagParam = \diagParamTup \in (0,1)^2$ with $\diagParamLow < \diagParamUp $, and $\spaceParam \in (0,1)$. Furthermore, we equip $(0,1)^2$ with the usual order in the second, and the opposite order in the first component, that is we write
\[
 \diagParam \leq \diagParam' \iff \diagParamLow' \leq \diagParamLow \mathrm{\quad and \quad} \diagParamUp \leq \diagParamUp'.
\]
We denote by $\paramSpace \subset (0,1)^2$ the subposet of $(0,1)^2$ with this order, consisting of $\diagParamTup$ with $0 < \diagParamLow < \diagParamUp < 1$. 
\end{notation}
Finally, let us model the several constructions connecting stratified spaces, strongly stratified spaces and stratification diagrams we described in \cref{sec:strat_htpy_theo} in this framework.
\begin{construction}\label{con:connecting_maps}
Consider the following three maps.
\[
\begin{tikzcd}
\SamP \arrow[r, shift left, "\estrongStr" ]  & \arrow[l, shift left, "\eforgetStr"] \SamNP \arrow[r, "\ediagification"] & \DSam
\end{tikzcd}
\]
These are defined via:
    \begin{align*}
        ( \sampX, \singsampX) &\xmapsto{\estrongStr} (\sampX, \min \{\dsingsampX, 1\}), \\
       \stratSamp = ( \sampX, \phiStratsampX ) &\xmapsto{\eforgetStr} (\sampX, \subLevel{\stratSamp}{\spaceParam}),\\
    \stratSamp  = ( \sampX, \phiStratsampX ) & \xmapsto{\ediagification} (\subLevel{\stratSamp}{\diagParamUp}, \supbLevel{\stratSamp}{\diagParamLow}{\diagParamUp}, \supLevel{\stratSamp}{\diagParamLow}) \spaceperiod
    \end{align*}
The map $\estrongStr$ corresponds to the assignment of a strong stratification to a stratified metric space (see \cref{ex:strat_metric_to_strong_strat_metric}).
$\eforgetStr$ gives a family of models for the forgetful functor, $\TopNP \to \TopP$, described in \cref{recoll:hoTop}. 
Finally, by \cref{prop:holinks_are_links}, $\ediagification$ (composed with $\estrongStr$) provides a model for the functor assigning to a stratified space its stratification diagram, $\diagFun \pp \TopP \to \Diag$ (see \cref{recol:diag_are_equ}). 
\end{construction}
\begin{example}
Consider the three pictures \cref{fig:PT_ex_fun,fig:PT_ex_fun_str,fig:PT_ex_fun_diag}. \cref{fig:PT_ex_fun} shows the pinched torus $PT$ as a stratified subspace of $\mathbb R^3$, with the singularity marked in red. \cref{fig:PT_ex_fun_str} shows $\strongStr{PT}$, where the color coding indicates the strong stratification. Finally, \cref{fig:PT_ex_fun_diag} shows the image under $\ediagification$ for $\diagParam= (0.2,0.4)$. Specifically, the union of the red and purple part give the $p$-part of the diagram, the purple part the $\{p<q\}$-part, and the union of the purple and the blue one the $q$-part.
\begin{figure}
    \centering
    \begin{minipage}{0.32\textwidth}
        \centering
        \includegraphics[width=1\textwidth]{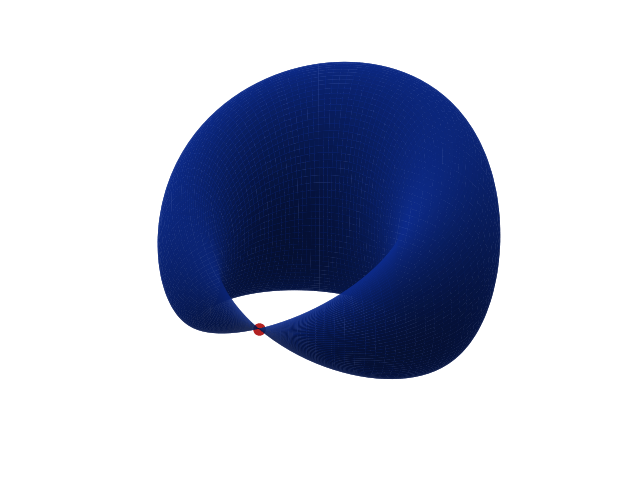}
        \caption{The pinched torus $PT$ as an element of $\SamP$}
        \label{fig:PT_ex_fun}
    \end{minipage}\hfill
    \begin{minipage}{0.32\textwidth}
        \centering
        \includegraphics[width=1\textwidth]{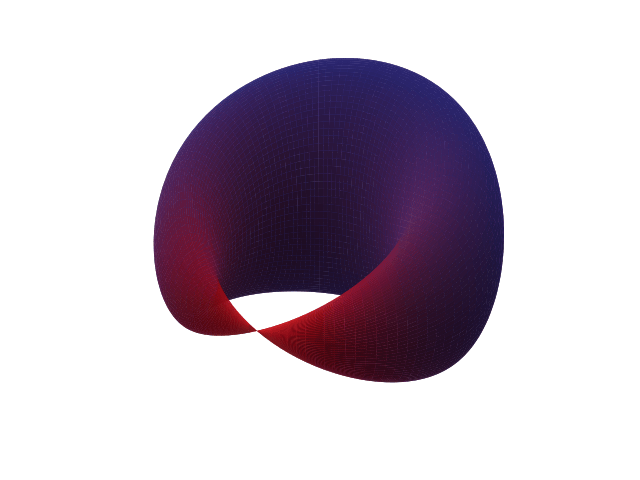}
        \caption{Illustration of $\strongStr{PT}$. The colouring indicates the strong stratification.}
        \label{fig:PT_ex_fun_str}
    \end{minipage}
    \begin{minipage}{0.32\textwidth}
        \centering
        \includegraphics[width=1\textwidth]{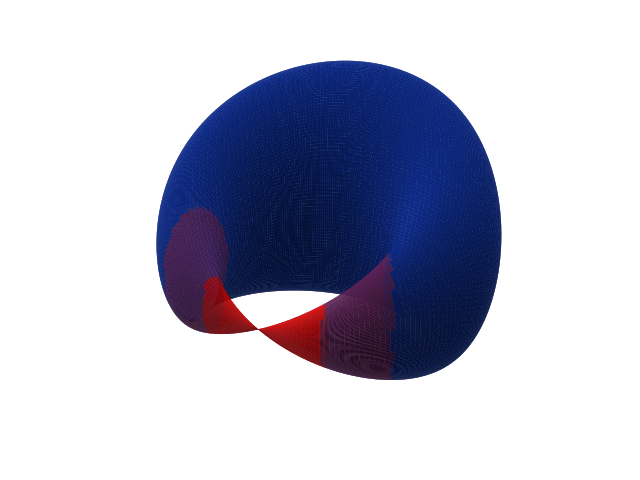}
        \caption{ Illustration of $\diagification{\strongStr{PT}}$.}
        \label{fig:PT_ex_fun_diag}
    \end{minipage}
\end{figure}

\end{example}
We will later make use of the following immediate elementary relation between $\ediagification$ and $\eforgetStr$.
\begin{lemma}\label{lem:forget_through_diag}
Let $\diagParam = (\diagParamLow, \spaceParam) \in \Omega$. Then,
\[
\forgetStr{\stratSamp} = ( \diagification{\stratSamp}_p \cup \diagification{\stratSamp}_{\{p < q\}} \cup \diagification{\stratSamp}_q, \diagification{\stratSamp}_p)
\]
for all $\stratSamp \in \SamNP$.
\end{lemma}
\begin{remark}\label{rem:sam_spaces_posets}
Note, that all of the described sample spaces naturally admit the structure of a poset. In the case of $\Sam$, $\SamP$ and $\DSam$ it is simply given by inclusions. In case of $\SamNP$, it is obtained by treating elements of $\SamNP$ as their graph, i.e. as a subset of $\supSp \times \uInt$ and then using the inclusion relation. Equivalently, this means
\[
 ( \sampX , s) \leq ( \sampX', s' ) \iff ( x \in \sampX \implies ( x \in \sampX' \And s(x)= s'(x)) ).
\]
In this fashion, the spaces of samples may also be treated as categories and the maps of \cref{con:connecting_maps} are functors. 
Furthermore, from this perspective we can treat $\SamP$ as a subcategory of $\TopP$ treating the equivalence in \cref{prop:ALT_holinks_are_links} as a natural equivalence. We will not make much use of this perspective here. However, it allows for notation such as $\Sam^{I}$, where $I$ is some indexing category to make sense, and we will use this freely. Furthermore, from this perspective the metrics on $\SamP$ and $\DSam$ are induced by the flow given by componentwise thickening (see  \cref{ex:hausdorff_distance} for details).
\end{remark}
\begin{notation}\label{not:htpy_const}
In the remainder of the section, we will frequently state that certain functors are \define{homotopically constant}, which means the following: If $\textbf{T}$ is a category with weak equivalences and $U$ some indexing category, then we say $F \in \ho \textbf{T}^{U}$ is homotopically constant with value $T \in \textbf{T}$, if there is an isomorphism in $\ho \textbf{T}^{U}$
\[
F \simeq T,
\]
where we treat $T$ as as object of the functor category $\textbf{T}^U$ by sending it to the constant functor of value $T$. Note that this implies in particular that all structure maps of $F$ are weak equivalences.
\end{notation}
The categorical perspective on the sampling spaces can be used to define a parameter independent version of 
$\ediagification$.
\begin{construction}
Note that for $\diagParam \leq \diagParam'$ we have natural inclusions
\[
 \ediagification \hookrightarrow \ediagification[\diagParam'].
\]
These induce a map 
\[
 \efreediagification\colon \SamNP \to \DSam^{\paramSpace}.
\]
\end{construction}
\cref{prop:holinks_are_links} may then be rephrased as follows. 
\begin{proposition}\label{prop:ALT_holinks_are_links}
Let $\str \in \SamP$ be cylindrically stratified. Then, for all $\diagParam \in \paramSpace$ we have a weak equivalence
\begin{align*}
    \diagFun( \str) \simeq \ediagification  \circ \estrongStr (\str).
\end{align*}
In fact, even more, there is an isomorphism in the homotopy category $\ho \Diag^{\paramSpace}$ 
\[
\diagFun( \str) \simeq \efreediagification  \circ \estrongStr (\str).
\]
\end{proposition}
\begin{proof}
Note, that in the proof of \cref{prop:holinks_are_links} we in fact first constructed a weak equivalence of $\ediagification  \circ \estrongStr (\str)$ with a diagram independent of $\diagParam$. It is immediate from the construction there, that this weak equivalences induces a weak equivalence from $\efreediagification \circ \estrongStr(\str)$ to a constant functor. The second part of the proof then shows that this constant functor is weakly equivalent to the constant functor with value $\diagFun{(\str)}$.
\end{proof}
In particular, under the equivalence of homotopy categories $\ho \TopP \cong \ho \Diag$, $\str$ and $\ediagification \circ \estrongStr (\str)$ represent the same stratified homotopy type. As a consequence, to define persistent stratified homotopy types, we can thicken stratification diagrams instead of stratified spaces.
\begin{construction}\label{con:thickening_diag}
Define \define{the thickening of $\sampD$ by $\alpha \geq 0$} via:
\[
    \thicken{\sampD} := ( \thicken{(\sampD_{q})}, \thicken{(\sampD_{ \{p < q \}})}, \thicken{(\sampD_{q})} ).
\]
For $\alpha \leq \alpha'$ there are the obvious inclusions of diagrams \[ \thicken{\sampD} \hookrightarrow \thicken[\alpha']{\sampD} \]
We thus obtain a map (functor from the categorical perspective) 
    \begin{align*}
        \DSam \to \DSam^{\mathbb R_+} \\
         \sampD \mapsto \{ \alpha \mapsto \thicken{\sampD} \}
    \end{align*}
with the structure maps given by inclusions. We may then treat the sample diagrams as elements of $\Diag$, ultimately obtaining the composition:
\begin{align*}
 \epersd \colon \DSam \to \DSam^{\mathbb R_+} \to  \ho \Diag^{\mathbb R_+} \simeq \ho \TopP^{\mathbb R_+} .
\end{align*}
\end{construction}
We now have everything in place to define persistent stratified homotopy types.
\begin{definition}\label{pers_strat_htpy_type}
The \define{persistent stratified homotopy type} of a stratified sample $\stratSamp \in  \SamP$ (depending on the parameter $\diagParam$) is defined as the image of $\stratSamp$ under the composition
\[
    \epersParam\colon\SamP \xrightarrow{\estrongStr} \SamNP \xrightarrow{\ediagification} \DSam \xrightarrow{\epersd} \pStrat, 
\]
where the final map is the one defined in \cref{con:thickening_diag}. 
\end{definition}
\begin{remark}\label{rem:computability_of_strat_pers}
Note, that by construction, $\epersParam$ fulfills an analog of Property \ref{enum:properties_of_PH1}, i.e. it admits a combinatorial interpretation which, for finite samples, can be stored on a computer. 
Indeed, by construction the persistent stratified homotopy type of $\stratSamp \in \SamP$ is equivalently represented by the image of $\stratSamp$ under
\[
\epersParam \pp \SamP \xrightarrow{\estrongStr} \SamNP \xrightarrow{\ediagification} \DSam \xrightarrow{\epersd} \ho \Diag^{\mathbb R_+},
\]
named the same by abuse of notation.
Then, it is a consequence of the classical nerve theorem (see e.g. \cite[Prop. 4G.3]{hatcher2002algtop} or \cite{borsuk1948nerve}) that for $\stratSamp \in \SamP$, $\persParam{\stratSamp}$ is equivalently represented by the diagram of \v{C}ech complexes
\begin{align*}
    \alpha \mapsto \{ \flagI \mapsto \textnormal{\v{C}}_{\alpha}( \ediagification(\estrongStr( \stratSamp))_{\flagI}) \}
\end{align*}
where $\textnormal{\v{C}}_{\alpha}$ denotes the \v{C}ech complex with respect to $\alpha$. When $\stratSamp$ is finite, this data can be stored on a computer and algorithmically evaluated.
\end{remark}
\begin{definition}\label{def:parameter-free}
The (parameter-free) \define{persistent stratified homotopy type} of a stratified sample $\stratSamp \in  \SamP$  is defined as the image of $\stratSamp$ under the composition
\begin{align*}
    \epers\colon \SamP  \xrightarrow{\estrongStr} \SamNP \xrightarrow{\efreediagification} \DSam^{\paramSpace} \xrightarrow{\epersd^{\paramSpace}} (\pStrat)^{\paramSpace} \cong \opStrat.
\end{align*}
\end{definition}
Then, the following two results guarantee that for sufficiently regular stratified spaces the homotopy type does not change under small thickenings (this is Property \ref{enum:properties_of_PH3}, see \cite{niyogi2008finding} for the analogous result in the non-stratified smooth setting). This justifies the use of persistent stratified homotopy types as a means to infer stratified homotopic information. Recall that the weak feature size of a subspace $\topSp$ of $\supSp$ (\cite{chazal2005weak}), is a non-negative real number $\varepsilon$ associated to $\topSp$, which has the property that the natural inclusion $\thicken{\topSp} \hookrightarrow \thicken[\alpha']{\topSp}$ is a homotopy equivalences, for $0 < \alpha \leq \alpha' < \varepsilon$.
\begin{proposition}\label{prop:thicken_stab_lok}
Let $\stratSp \in \SamP$ be a compact, definable stratified metric space. Then, for any $v \in \paramSpace$, there exists an $\varepsilon >0$, such that the structure map
\[
\persParam{\stratSp}(\alpha) \to \persParam{\stratSp}(\alpha')
\]
is a weak equivalence, for all $0 \leq \alpha \leq \alpha' < \varepsilon$. In particular, 
\[
\persParam{\stratSp}\mid_{[0,\varepsilon)} \simeq S.
\]
In other words, the persistent stratified homotopy type of $S$ at $\diagParam$ restricted to $[0,\varepsilon)$, is weakly equivalent to the constant functor with value $\str$. Furthermore, $\varepsilon$ can be taken to be the minimum of the weak feature size of the entries of $\diagification{\topSp}$ (see \cite{chazal2005weak}), and the latter is positive.
\end{proposition}
\begin{proof}
Note that by definition of a weak equivalence in the category of stratification diagrams, this statement really just says there exists an $\varepsilon >0$, such that for each flag $\flagI$ in $P$ the inclusions 
\[ \diagification{\stratSp}_{\flagI} 
\hookrightarrow
\thicken[\alpha]{(\diagification{\stratSp}_{\flagI})} \]
are weak equivalences, for $\alpha \leq \varepsilon$. Note, however, that by the definability assumption $\diagification{\stratSp}_{\flagI}$ is again definable. Hence, this follows from the fact that the homotopy type of compact definable sets is invariant under slight thickenings
(see \cref{lem:appendix_definably_thickenable} for the precise statement and the fact that $\varepsilon$ can be taken as the minimum of the weak feature size of the entries of $\diagification{\topSp}$).
For $\alpha = 0$, we have 
\[
\persParam{\stratSp}(0) \simeq \stratSp
\]
by \cref{prop:ALT_holinks_are_links}. 
\end{proof}
\begin{proposition}\label{prop:thicken_stab_glob}
Let $\stratSp \in \SamP$ be a compact, definably stratified space. Then (up to a linear rescaling), the persistent stratified homotopy type
\[
    \pers{\stratSp}\colon \paramSpace \times \mathbb R_{+} \to \ho \TopP
\]
is homotopically constant with value $\stratSp$ on an open neighborhood of $\paramSpace \times \{ 0\}$. 
\end{proposition}
\begin{proof}
That the functor is homotopically constant with value $\stratSp$ on $ \paramSpace \times \{ 0\}$ is the content of \cref{prop:ALT_holinks_are_links}. Let $\omega_v$ denote the minimum of the weak feature sizes (compare to \cite{chazal2005weak}) of the entries of $\diagification{\stratSp}$. An elementary argument shows that $\omega_v$ varies continuously in $\diagParam$. By \cref{lem:appendix_definably_thickenable} all weak feature sizes involved are positive.
We take 
\[
U := \{ (\diagParam, \alpha) \mid \textnormal{$\alpha$ is smaller than } \omega_v \}.
\]
From \cref{prop:thicken_stab_lok} it follows, that all structure maps of $\pers{\stratSp}$ on $U$ in direction $\mathbb R_+$ are weak equivalences. From this, it already follows that all structure maps of $\pers{\stratSp}\mid_{U}$ are weak equivalences. For the slightly stronger result that this already implies that $\pers{\stratSp}\mid_{U}$ is homotopically constant, see \cref{lem:appendix_constant_diagram}.
\end{proof}
\subsection{Metrics on categories of persistent objects}\label{subsec:flows}
One of the central requirements for the use of persistent homology in practice is the fact that it is stable with respect to Hausdorff and interleaving distance (\ref{enum:properties_of_PH2}, first shown in \cite{chazal2008interleave}).
Investigating the use of metrics in persistent scenarios and the stability of functors with respect to them has since been the content of ongoing research (\cite{bramer2020atomic,hofer2017deep,lesnick2015interleave,bjerkevik2021lpdist,bubenik2020algwasser}). The stability of persistent homology with respect to the interleaving distance may, however, already be phrased at the level of persistent homotopy types (even on the level of persistent spaces), as we explain in the remainder of this subsection.
The goal of \cref{subsec:stab_pers_type} is to investigate the stability behavior of the persistent stratified homotopy type. To do so, we make us of the notion of a flow introduced \cite{silva2018interleave}. For the sake of conciseness, we recall a slightly less general definition here. 
\begin{recollection}[\cite{silva2018interleave}]\label{recol:flow}
    A \define{strict flow} on a category $\mathcal{C}$ is a strict monoidal functor $(-)_{-}\pp \mathbb R_{+} \to \textnormal{End}(\mathcal C)$. In other words, to each $\varepsilon \in \mathbb R_+$ we assign an endofunctor $(-)_{\varepsilon}$ and whenever $\varepsilon \leq \varepsilon'$ we assign (functorially) a natural transformation $s_{\varepsilon \to \varepsilon'}\pp (-)_{\varepsilon} \to (-)_{\varepsilon'}$. Being strict monoidal means that $(-)_0 = 1_{\mathcal C}$, 
    $(-)_{\varepsilon'} \circ (-)_{\varepsilon} = (-)_{\varepsilon + \varepsilon'}$ and $(s_{\varepsilon \leq \varepsilon'})_{\delta} = s_{\varepsilon + \delta \leq \varepsilon' + \delta}$. Generally, one should think of flows as a notion of shift on $\mathcal C$. 
    Then, just as in the scenario of the interleaving distance for persistence modules \cite{chazal2008interleave}, one says that $X,Y \in \mathcal C$ are \define{$\varepsilon$-interleaved} if there are morphisms $f\pp X \to Y_{\varepsilon}$ and $g \pp Y \to X_{\varepsilon}$ and such that the diagram
    \begin{diagram}
     X \arrow[d, "f"'] \arrow[rd] & Y \arrow[d, "g"] \arrow[ld] \\
     {Y_\varepsilon}\arrow[d, "{g_{\varepsilon}}"'] \arrow[rd, crossing over] & {X_{\varepsilon}} \arrow[d, "{f_{\varepsilon}}"] \arrow[ld] \\
     {X_{2\varepsilon}} & {Y_{2\varepsilon}} 
    \end{diagram}
    commutes (all unlabelled morphisms are given by the flow). One then obtains a (symmetric Lawvere) metric space by setting
    \begin{equation}
        \dmet[\In]{X}{Y} := \inf \{ \varepsilon \geq 0 \mid X,Y \textnormal{ are $\varepsilon$-interleaved} \}.
    \end{equation}
    An immediate consequence of this definition is that any functor $F\pp \mathcal C \to \mathcal D$ between categories with a (strict) flow that fulfills $(-)_{\varepsilon} \circ F = F \circ (-)_{\varepsilon}$ for all $\varepsilon \in \mathbb R_+$ is necessary $1$-Lipschitz with respect to the interleaving distances, that is it fulfills \[\dmet[\In]{F(X)}{F(Y)} \leq \dmet[\In]{X}{Y},\] for $X,Y \in \mathcal{C}$.
\end{recollection}
\begin{example}\label{ex:int_distance}
Consider $\mathbb{R}^k$ equipped with the product partial ordering, given by $u \leq v :\iff \forall i \in \{1,\dots,n \} \colon u_i \leq v_i$.
Let $U \subset \mathbb R^k$ be a generalized interval (i.e. a subset of $\mathbb R^k$, such that $u,v \in U, u\leq w\leq v$ implies $w \in U$). Furthermore, let $\textbf{T}$ be a category with a terminal object $*$. 
    Then any functor category $\textbf{T}^{U}$ is naturally equipped with a shift type flow, given by 
        \begin{align*}
            X_\varepsilon(u):= 
                \begin{cases}
    X(u + \varepsilon(1,\dots,1)) &\textnormal{, for $u +  \varepsilon(1,\dots,1) \in U$;}\\
    * &\textnormal{, for $u + \varepsilon(1,\dots,1) \notin U$}.
        \end{cases}
        \end{align*}
    If $\textbf{T}$ is equipped with a notion of weak equivalence (which includes all isomorphisms), then by construction the flow respects weak equivalences in $\textbf{T}^{U}$. Thus, it descends to a flow on the homotopy category $\ho \textbf{T}^{U}$ obtained by localizing weak equivalences of functors. In particular, this construction equips the persistent homotopy category $\ho \Top^{\mathbb R_+}$ with a symmetric Lawvere metric, such that the functor 
    \[
    \Top^{\mathbb R_+} \to \ho \Top^{\mathbb R_+}
    \]
    is $1$-Lipschitz.
    More generally, the same construction works for the cases $\textbf{T} = \TopP, \Diag$. We call distances of this type \define{interleaving distances}.
    Furthermore, for any functor between two such categories $\textbf{T},\textbf{T}'$, which descends to the homotopy category, the induced functors 
    \begin{align*}
        \textbf{T}^{\mathbb R_+} &\to \textbf{T}'^{\mathbb R_+}, \\
        \ho \textbf{T}^{\mathbb R_+} &\to \ho\textbf{T}'^{\mathbb R_+} 
    \end{align*}
    commute with shifting  and are thus $1$-Lipschitz. Whenever we refer to a metric on such a functor category (or its homotopy category), we will be referring to the interleaving distance.
\end{example}
\begin{example}\label{ex:hausdorff_distance}
Denote by $\Sam$ the category of subsets of $\supSp$, with morphisms given by inclusions. If we take 
$
\thicken[\varepsilon]{\sampX}
$
to be given by an $\varepsilon$ thickening,
for $\varepsilon \in \mathbb R_{+}$ then this construction defines a strict flow on $\Sam$. The distance induced by the flow is the Hausdorff distance (compare \cite{silva2018interleave}). Clearly, the functor 
\begin{align*}
    \mathcal{P}\pp \Sam &\to \Top^{\mathbb R_+} \\
    \sampX &\mapsto \{ \varepsilon \mapsto \sampX_{\varepsilon} \}
\end{align*}
commutes with flows. In particular, this gives an, albeit somewhat unnecessarily abstract, proof of the fact that persistent homotopy types are stable with respect to Hausdorff and interleaving distance, which immediately implies the stability of persistent homology (\cite{chazal2008interleave}). More generally, if we define thickening componentwise, then the Hausdorff style distances on $\SamP$ and $\DSam$ (\cref{def:samP,def:DSam}) are also induced by the thickening flow. 
\end{example}
In the next subsection, we are going to make frequent use of \textit{tautological} stability results as in \cref{ex:int_distance,ex:hausdorff_distance}. For example, from the flow perspective, just as in the non-stratified scenario, one immediately obtains:
\begin{lemma}\label{lem:pers_diag_lip}
$\epersd \pp \DSam \to \pStrat $ is $1$-Lipschitz.
\end{lemma}
\begin{example}\label{def:global_interleaving_distance}
    We may identify $\paramSpace$ with the subset \[U=\{ (x,y) \in (-1,0) \times (0,1) \mid -x < y \}\] by mapping $\diagParamTup \mapsto (-\diagParamLow, \diagParamUp )$. Note that this defines an order preserving map. Then the construction in \cref{ex:int_distance} defines a flow distance on $\opStrat$.
\end{example}
\subsection{Stability of persistent stratified homotopy types}\label{subsec:stab_pers_type}
As we have seen in \cref{subsec:flows}, especially \cref{ex:hausdorff_distance}, the stability of persistent homotopy type with respect to the Hausdorff distance is almost tautological. When beginning with a stratification diagram of samples, the situation is similarly easy for the stratified scenario (see \cref{lem:pers_diag_lip}). 
The situation for persistent stratified homotopy types when starting from a stratified sample is a little more subtle. This stems from the fact that taking sublevel sets with respect to some function (or equivalently intersecting with a closed subset) is generally not even a continuous operation with respect to the Hausdorff distance.
\begin{example}\label{ex:sublevel_not_stable}
Let $N=1$, i.e. $\supSp = \mathbb R$ and $A = (-\infty,0]$. Let $\sampX = \{ 0 \}$ and $\sampX_{n}= \{\frac{1}{n}\}$. $\sampX_{n}$ converges to $\sampX$ in the Hausdorff distance. However, the intersections $A \cap \sampX_n$ are empty, while $A \cap \sampX= \sampX$. In particular, we have 
\[
\dmet[\HD]{A \cap \sampX_n}{A \cap \sampX} = \infty,
\]
for all $n > 0$.
\end{example}
The problem in \cref{ex:sublevel_not_stable} is that $\sampX$ is lacking a certain amount of homogeneity while passing into the interior of $A$. Indeed, if we replaced $ \sampX$ by $[-\varepsilon,0]$ then no such phenomena occur and one can easily show that $A \cap -$ is continuous in $[-\varepsilon,0]$. As a more involved incarnation of this phenomenon, we now show stability of persistent stratified homotopy types, when sampling around a compact cylindrically stratified metric space.\\
A peculiarity of the stratified setting is that stability of persistent stratified homotopy types does generally not hold globally, as it does in the non-stratified case, but only at sufficiently regular elements of $\SamP$. To capture this notion of stability, we need the following notion of local Lipschitz continuity.
\begin{definition}\label{def:lipschitz} Let $K \in [0 ,\infty)$. 
We say that a function of (symmetric Lawvere) metric spaces $f: M \to M'$ is \define{$K$-Lipschitz (continuous) at $x \in M$}, if there is a $\delta >0$, such that 
\[
\dmet{f(x)}{f(y)} \leq K \dmet{x}{y},
\]
for all $y \in M$ with $\dmet{x}{y} \leq \delta$. We say \define{$f$ is $K$-Lipschitz (continuous)}, if it is $K$-Lipschitz for $\delta = \infty$, at every $x \in M$.
\end{definition}

Let us begin by gathering some of the more obvious elementary results. For the remainder of this subsection let $\pos= \{ p <q \}$ be a poset with two elements and let $\diagParam= \diagParamTup \in \paramSpace$ and $\spaceParam \in [0,1]$.

\begin{proposition}\label{prop:strong_str_1}
The map
\begin{align*}
    \estrongStr\colon \SamP \to \SamNP
\end{align*}
is 2-Lipschitz. 
\end{proposition}
\begin{proof}
This is immediate from the triangle inequality. 
\end{proof}
As an immediate consequence of the definition of the metric for strongly stratified samples, one obtains.
\begin{lemma}\label{lem:inclusion_lemma}
Let $\stratSamp, \stratSamp' \in \SamNP$ and $v' \leq v \in [0,1]$.
Then
 \[ 
  \supbLevel{(\stratSamp')}{v'}{v} \subset \thicken[\delta]{(\supbLevel{\stratSamp}{v'-\delta}{v + \delta })},
 \]
for any $\delta > \dmet[\SamNP]{\stratSamp}{\stratSamp'}$. 
\end{lemma}
As a corollary of \cref{lem:inclusion_lemma,prop:strong_str_1} together with the definition of the flow distance on $ \opStrat$ (\cref{def:global_interleaving_distance}), we obtain:
\begin{corollary}\label{cor:flexible_pers_strat_Lipsch}
The map
\[
\epers\colon \SamP \to \opStrat
\]
is $2$-Lipschitz.
\end{corollary}
The case of the persistent stratified homotopy type depending on $\diagParam$ is more subtle, since one lacks the possibility of diagonal interleavings. The following technical lemma is the decisive argument in showing stability nevertheless.
\begin{lemma}\label{lem:diag_inequality}
Let $\stratSamp, \stratSamp' \in \SamNP$. 
Let $\delta> \dmet[\SamNP]{\stratSamp}{\stratSamp'}$ and suppose $\diagParam + \delta := (\diagParamLow - \delta, \diagParamUp + \delta) \in \paramSpace$ and $\diagParam - \delta:=  (\diagParamLow + \delta, \diagParamUp - \delta) \in \paramSpace$.
Then
\[
\dmet[\DSam]{\diagification{\stratSamp}}{ \diagification{\stratSamp'}} \leq \delta + \max\{ \dmet[\DSam]{\diagification{\stratSamp}}{\diagification[\diagParam \pm \delta]{\stratSamp}} \}.
\]
Similarly, if $ u \pm \delta \in (0,1)$, then 
    \[ \dmet[\SamP]{\forgetStr{\stratSamp}}{\forgetStr{\stratSamp'}} \leq \delta + \max \{\dmet[\SamP]{\forgetStr{\stratSamp}}{\forgetStr[u \pm \delta]{\stratSamp}} \}. 
    \]
\end{lemma}
\begin{proof} We prove the diagram case, the other one can be shown completely analogously.
Let $\alpha >  \dmet[\DSam]{\diagification{\stratSamp}}{\diagification[\diagParam \pm \delta]{\stratSamp}}$. We then have inclusions
\[
\begin{tikzcd}[row sep = small]
\diagification{\stratSamp} \arrow[r, hook]&\thicken[\alpha]{\diagification[\diagParam-\delta]{\stratSamp}}\arrow[r, hook]& \thicken[\delta + \alpha]{\diagification{\stratSamp'}} \spacecomma \\ 
\diagification{\stratSamp'} \arrow[r, hook]& \thicken[\delta]{\diagification[\diagParam + \delta]{\stratSamp}} \arrow[r, hook] & \thicken[\delta+\alpha]{\diagification{\stratSamp} } \spaceperiod 
\end{tikzcd}
\]
The upper left and lower right inclusion follow by the assumption on $\alpha$. The lower left and upper right inclusions follow by \cref{lem:inclusion_lemma}. Hence, the result follows by considering the diagram distance as coming from a thickening flow as in \cref{ex:hausdorff_distance}.
\end{proof}
Morally speaking, the way we should think of \cref{lem:diag_inequality}, is that the continuity of $\ediagification$ in a strongly stratified sample $\stratSamp$ depends on the continuity of $\diagification{\stratSamp}$ in the parameter $\diagParam$.
As an immediate corollary of the second part of \cref{lem:diag_inequality} we obtain the following result, which will come in handy in \cref{subsec:restrat}.
\begin{corollary}
\label{prop:cont_if_const_u}
Let $\delta >0$ such that $u \pm \delta \in (0,1)$.
Let $\stratSamp = ( \sampX, \phiStratsampX) \in \SamNP$ be such that  $\stratSamp_{\leq u}= \stratSamp_{\leq u\pm \delta}$. Then
\[
\eforgetStr\colon \SamNP \to \SamP
\]
is $1$-Lipschitz at $\stratSamp$ (on an open ball with radius $\delta$).
\end{corollary}
The continuity of $\diagification{\stratSamp}$ in $\diagParam$ can furthermore be reduced to the continuity of the $\{ p < q\}$ parts of diagrams, by the following lemma.
\begin{lemma}\label{lem:diag_trick}
Let $\stratSamp \in \SamNP$ and $\diagParam, \diagParam' \in \paramSpace$  and set $a= \min\{\diagParamLow,\diagParamLow'\}, b=\max\{\diagParamUp,\diagParamUp'\}$,
then
\[
 \dmet[\DSam]{\diagification[\diagParam]{\stratSamp}}{\diagification[\diagParam']{\stratSamp}} \leq \max \{ \dmet[\HD]{\stratSamp^{\diagParamLow}_{\diagParamUp}}{\stratSamp^{\diagParamLow'}_{\diagParamUp'}}, \dmet[\HD]{\stratSamp^{a}_{\diagParamUp'}}{\stratSamp^{a}_{\diagParamUp}}, \dmet[\HD]{\stratSamp^{\diagParamLow}_{b}}{\stratSamp^{\diagParamLow'}_{b}} \}. 
\]
\end{lemma}
\begin{proof}
This is an immediate consequence of the fact that 
\[
\dmet[\HD]{\sampX}{ \sampX[Y]} \leq \dmet[\HD]{\sampX \setminus \sampX[A]}{\sampX[Y] \setminus \sampX[A]},
\]
for $\sampX[A] \subset \sampX , \sampX[Y] \in \Sam$.
\end{proof}
In case of compact cylindrically stratified spaces, $\diagification{\stratSamp}$ does indeed vary continuously in $\diagParam$.
\begin{proposition}\label{prop:diag_vary_cont_in_v}
Let $\str \in \SamNP$ be compact and cylindrically stratified. Then
\begin{align*}
    (0,1) &\to \SamP \\
    u &\mapsto \forgetStr{\stratSp}
\end{align*}
and 
\begin{align*}
    \paramSpace &\to \DSam \\
    \diagParam &\mapsto \diagification{\stratSp}
\end{align*}
are continuous.
\end{proposition}
\begin{proof}
Note that it suffices to show the case of $\ediagification$, since the nontrivial part of the continuity for $\eforgetStr$ is given by the $(\eforgetStr)_p$ component, and the latter is defined identically to the $p$-component of $\diagification{\stratSp}$.
By \cref{lem:diag_trick}
it suffices to show that for $\diagParam \to \diagParam^0$, we also have  
\[
\dmet[\HD]{ \supbLevel{\stratSp}{\diagParamLow}{\diagParamUp}}{\supbLevel{\stratSp}{\diagParamLow^0}{\diagParamUp^0}} \to 0 .
\]
Next, note that the topology of the Hausdorff distance on the space of compact subspaces of a space only depends on the topology of the latter. Set $L:=\stratSp^{\frac{1}{2}}_{\frac{1}{2}}$. Then by the cylinder assumption we may without loss of generality compute $\mathrm{d}$ in $L \times (0,1)$ equipped with the product metric. In particular, $\supbLevel{\stratSp}{\diagParamLow}{\diagParamUp} = L \times [\diagParamLow, \diagParamUp]$. We then have 
\begin{align*}
\dmet[\HD]{\supbLevel{\stratSp}{\diagParamLow}{\diagParamUp}}{ \supbLevel{\stratSp}{\diagParamLow^0}{\diagParamUp^0}} &= \dmet[\HD]{L \times [\diagParamLow, \diagParamUp]}{L \times [\diagParamLow^0, \diagParamUp^0]} \\
& \leq \max\{\vert \diagParamLow- \diagParamLow^0\vert, \vert\diagParamUp - \diagParamUp^0\vert \} \xrightarrow{\diagParam \to \diagParam^0} 0.
\end{align*}
\end{proof}
From \cref{prop:diag_vary_cont_in_v,lem:diag_inequality} we obtain the following result. Here $\paramSpace$ is equipped with the metric induced by the maximum norm.
\begin{corollary}\label{cor:diag_forget_cont}
Let $\stratSp \in \SamNP$ be compact and cylindrically stratified. Then
\begin{align*}
    \eforgetStr\colon \SamNP \to \SamP, \\
    \ediagification\colon \SamNP \to \DSam
\end{align*}
are continuous at $\stratSp$.\\ Even more, if $\stratSp_{\leq -}\colon (0,1) \to \Sam$ is $K$-Lipschitz in a neighborhood of $u$ (respectively $\stratSp_{-}^{-}$ in a neighborhood of $v$), then $\eforgetStr$ ($\ediagification$) is $(K+1)$-Lipschitz at $\stratSp$. 
\end{corollary}
In total, we finally obtain the following stability result for persistent stratified homotopy types, which can be seen as a (slightly weaker) version of the classical, non-stratified Property \ref{enum:properties_of_PH1}. In the next subsection (specifically in \cref{prop:pers_strat_htpy_type_C-lipschitz}), we strengthen this general stability result significantly for the case of Whitney stratified spaces.
\begin{theorem}
\label{cor:continuity_of_pers_tame}
Let $\stratSp \in \SamP$ be compact and cylindrically stratified. Then
\[
\epersParam\colon \SamP \to \pStrat
\]
is continuous at $\stratSp$. Even more, if $\stratSp_{-}^{-} \colon \paramSpace \to \Sam$ is $K$-Lipschitz in a neighborhood of $\diagParam$, then $\epersParam$ is $2(K+1)$-Lipschitz at $\stratSp$. 
\end{theorem}
\begin{proof}
Recall that $\epersParam = \epersd \circ \ediagification \circ \estrongStr$. $\estrongStr$ is $2$-Lipschitz by \cref{prop:strong_str_1}. Furthermore, by \cref{cor:diag_forget_cont}, $\ediagification$ is continuous in $\strongStr{S}$. Finally, $\epersd$ is $1$-Lipschitz by \cref{lem:pers_diag_lip}. The second statement follows similarly. 
\end{proof}
\subsection{Stability at Whitney stratified spaces}\label{subsec:stab_pers_type_whit}
One way to think of Whitney's condition (b) is that it gives additional control over the derivatives of the rays of the mapping cylinder neighborhood of a stratified space. This additional control can be used to improve the stability result in \cref{cor:continuity_of_pers_tame} to Lipschitz continuity (\cref{prop:links_thicken_L_continuously}).
To show this, we need to first consider an asymmetric version of the Hausdorff distance for subspaces of $\supSp$. For the remainder of this subsection, $\pos$ is not restricted to the case of two elements.
\begin{definition}
Let $V,U \subset \supSp$ be linear subspaces. The (asymmetric) \define{distance of $V$ to $U$} is given by 
    \[
    \distVec{V}{U}=\sup_{v \in V, \norm{v} = 1}\inf \{ \norm{v-u} \mid u\in U \} =  \sup_{v \in V, \norm{v}=1}\{ \norm{ \pi_{U^{\perp}}(v)} \},
    \]
where $\pi_{U^{\perp}}$ denotes the orthonormal projection to the orthogonal complement of $U$.
\end{definition}
Whitney's condition (b) can be expressed in terms of a function, which measures the failure of secants being contained in the tangent space, as follows (compare \cite{hironaka1969normal}).
\begin{construction}\label{con:pqbeta}
Let $\str = \stratPair$ be a stratified space with smooth strata, contained in $\supSp$.
Consider the function
\begin{align*}
\funpqbeta\colon \topSp \times \topSp \to \mathbb R; &&
\begin{cases}
 (x,y) &\mapsto \distVec{\secant{x}{y}}  {\tangentSp[x]{X_{\phiStrat(x)}}} \text{ if }x \neq y, \\
 (x,x) &\mapsto 0 \text{ else}
\end{cases}
\end{align*}
where we consider all tangent spaces involved as linear subspaces of $\supSp$. 
\end{construction}
Condition (b) can be formulated as $\beta$ restricting to a continuous function on certain subspaces of $X \times X$.
\begin{proposition}\label{prop:equ_char_Whitney}
Let $\str = \stratPair$, be as in the assumption of \cref{con:pqbeta} and further so that the frontier and local finiteness condition are fulfilled. 
Then, $\str$ is a Whitney stratified space if and only if \[ \funpqbeta\mid_{ (\topSp_q \times \topSp_p) \cup \Delta_{\topSp_p}}:  (\topSp_q \times \topSp_p) \cup \Delta_{\topSp_p} \to \mathbb R,\] is continuous, for all pairs $q \geq p \in \pos$.  Here $\Delta_{\topSp_p}$ denotes the diagonal of $\topSp_p \times \topSp_p$,
\end{proposition}
\begin{proof}
This statement is somewhat folklore. For the sake of completeness, we provide a proof in \cref{pf:equ_char_Whitney}.
\end{proof}
Next, we need the notion of integral curves, as defined for example in \cite{hironaka1969normal}. 
\begin{proposition}\cite[Lemma 4.1.1]{hironaka1969normal}\label{prop:existent_of_integral_curves}
Let $\whs = \stratPair$ be a Whitney stratified space and $y \in \whs_p$, for some $p\in \pos$.
Let $B=\ballEps[d]{y} \subset \supSp$ be a ball of radius $d$ around $y$, such that $\funpqbeta(-,y)$ is bounded uniformly by some $\delta < 1$, on $W_{\geq p} \cap B$.
Then, for any $x \in \whs_{\geq p} \cap B$, $x\neq y$, there exists a unique curve $\phi\colon [0,d] \to \whs \cap B$, fulfilling
\begin{enumerate}
\item $\phi(0)=y$ and $\phi(\norm{y-x}) = x$,
    \item $\phi$ is almost everywhere differentiable. At differentiable points, $t \neq 0$, the differential is given by
    \[\phi'(t) = 
    \frac{\norm{ \phi(t) - y}}{\norm{\pi_{\phi(t)}(\phi(t)-y)}^2}\pi_{\phi(t)}(\phi(t)-y),
    \]
    where $\pi_{\phi(t)}$ denotes the projection to $\tangentSp[\phi(t)]{\whs_{\phiStrat(\phi(t))}}$. 
     \end{enumerate}
\end{proposition}
\begin{definition}
A curve as in \cref{prop:existent_of_integral_curves} is called the \define{integral curve associated to the pair $x,y$}.
\end{definition}
The existence of integral curves allows for additional control over the mapping cylinder neighborhoods defined in \cref{ex:cyl_nbhds}. This is essentially due to the following result.
\begin{proposition}\cite[Proof of 4.1.1]{hironaka1969normal}\label{prop:properties_of_integral_curves}
Let $\whs$ be a Whitney stratified space over $\pos$ and let $\phi\colon [0,d] \to \whs$ be the integral curve associated to $x \in \whs_q$, $y \in \whs_p$, $q \geq p \in \pos$, with notation as in \cref{prop:existent_of_integral_curves}. Then $\phi$ has the following properties. 
\begin{enumerate}
    \item $\norm{\phi(t) - y} = t$, for $t \in [0,d]$.
    \item $\norm{\phi(t) - \phi(t')} \leq \frac{1}{\sqrt{1-\delta^2}} \vert t-t'\vert$, for $t,t'
\in [0,d]$.
\end{enumerate}
\end{proposition}
As a consequence of this result, the continuity result of \cref{cor:continuity_of_pers_tame}
can be improved to Lipschitz continuity.
\begin{proposition}\label{prop:links_thicken_L_continuously}
Let $\pos = \{ p <q \}$ and let $\whs \in \SamP$ be a Whitney stratified space with compact singular stratum $\whs_{p}$. Then, for any $C >1$, there exists an $R>0$, such that the function  
\begin{align*}
   \paramSpace \cap (0,R)^2 &\to \DSam \\
    \diagParam &\mapsto \diagification{\strongStr{\whs}}
\end{align*}
is $C$-Lipschitz continuous.
\end{proposition}
\begin{proof}
We omit the $\estrongStr$, to keep notation concise.
By \cref{lem:diag_trick}, it again suffices to consider the link part of the diagrams given by $\supbLevel{\whs}{\diagParamLow}{\diagParamUp}$.
Choose $\delta <1$ such that $\frac{1}{\sqrt{1-\delta^2}} < C$. Next, take $R$ small enough such that $ \neiEps[R]{\whs_p}$, with retraction $r\colon  \neiEps[R]{\whs_p} \to \whs_p$ is a standard tubular neighborhood of $\whs_p$. 
By \cite[Lemma 2.1]{nocera2021conically}, for $R$ small enough the spaces $\whs^y=r^{-1}(y) \cap \whs$ of $y$ are given by Whitney stratified spaces with singular stratum given by a point. 
Then, using \cref{con:hatbeta}, we may also choose $R$ so small, that
\[\beta(x,y) \leq \delta,\] for the respective $\beta$ on the fiber $\whs^y$. 
Now, let $\diagParam, \diagParam' \in \paramSpace \cap [0,R]$. Let $x \in \supbLevel{\whs}{\diagParamLow}{\diagParamUp}$ and assume that $\diagParamUp > \diagParamUp'$ (the other cases work similarly). Now, consider the integral curve $\phi$ from $y:=r(x) \in \whs_p$ to $x$ in $r^{-1}(y) \cap \whs$. 
By \cref{prop:properties_of_integral_curves} we have,
\[
\vert x - \phi( \diagParamUp[v'])\vert  = \vert \phi(\vert x\vert )- \phi( \diagParamUp[v'])\vert  \leq C \vert  \vert x\vert  - \diagParamUp[v']\vert  \leq C \vert \diagParamUp - \diagParamUp[v']\vert  \leq  C \vert \diagParam -\diagParam'\vert .
\]
Since $\phi(\diagParamUp[v']) \in \whs^{\diagParamLow[v']}_{\diagParamUp[v']}$, going through all the cases, we obtain
\[
\supbLevel{\whs}{\diagParamLow}{\diagParamUp} \subset  \thicken[C \vert \diagParam -\diagParam'\vert ]{(\supbLevel{\whs}{\diagParamLow[v']}{\diagParamUp[v']})}.
\]
Thus, the result follows by symmetry.
\end{proof}
We thus obtain, as a corollary of \cref{cor:continuity_of_pers_tame}, that for $\diagParam$ sufficiently small the persistent stratified homotopy type $\epersParam$ is even Lipschitz continuous at a Whitney stratified space.
\begin{theorem}\label{prop:pers_strat_htpy_type_C-lipschitz}
Let $\pos = \{ p < q\}$ and $\whs \in \SamP$ be Whitney stratified with $\whs_p$ compact. Then, for any $C >1$, there exists some $R>0$, such that the map
\[
\epersParam: \SamP \to \pStrat
\]
is $2(C+1)$-Lipschitz continuous at $\whs$, for all $\diagParam \in \paramSpace \cap (0, R)^2$. 
\end{theorem}

\section{Learning stratifications}\label{sec:recover_strat}
In practice, we can generally not expect that sample data is already equipped with a stratification. This requires for notions of stratification which are intrinsic to the geometry of a space. One such example are homology stratifications, as used by Goresky and MacPherson in \cite{goresky1983intersection}. 
\begin{example}\label{ex:loc_homology}
For the sake of simplicity, we describe the case of two strata. Suppose $\stratSp=  (\phiStrat \pp \topSp \to \{p <q \})$ is stratified conically as follows: \\ $\str_q$ is locally Euclidean of dimension $q$, and $\str_p$ of dimension $p$, and
$x \in \str_p$ admits a neighborhood
\[
 U \cong_P \mathbb R^p \times C(L)
\]
for some $q - (p+1)$ dimensional compact manifold $L$, called the link of $x$.
Here $C(L)$ is the stratified cone on $L$, stratified over $\{ p<q\}$, by sending only the cone point to $p$.
This holds, for example, if $\stratSp$ is a Whitney stratified space.
Suppose further that $L$ is not a homology sphere, and that the strata are connected. Then, the stratification of $\stratSp$ can be recovered from the underlying space as follows.
For each $x \in \topSp$, we can compute the local homology of $X$ at $x$ 
\[
\locHlg{\topSp} := \hlg{\topSp, \topSp \setminus \{x\}} = \varinjlim \hlg{\topSp, \topSp \setminus U},
\]
where the colimit ranges over the open subsets of $\topSp$ containing $x$.
By the assumption on the local geometry of $\topSp$, for any $x \in \topSp$ there exists a small open neighborhood $U_x$ such that the natural map
\[ 
 \hlg{\topSp, \topSp \setminus U_x} \to \locHlg{\topSp} 
\]
is an isomorphism. In particular, for each $x \in \topSp$ one obtains natural maps
\[
\locHlg{\topSp} \cong \hlg{\topSp, \topSp \setminus U_x} \to \locHlgy{\topSp}
\]
for $y \in U_x$. If $x,y$ are contained in the same stratum, then all of these maps are given by isomorphisms. By the path connectedness assumption any two points in the same strata are connected by such a sequence of isomorphisms. Conversely, since we assumed that $L$ is not a homology sphere, we have 
\[
\locHlg{\topSp} \cong \locHlg{U_x} \cong \locHlg{\mathbb R^p \times C(L)} \cong \hlgred[\bullet-(p+1)]{L} \neq \locHlgy{\topSp},
\]
whenever $x \in \str_p$ and $y \in \str_q$. Thus, we can reobtain the stratification of $\str$, by assigning to points the same stratum, if and only if they are connected through such a sequence of isomorphisms. Stratifications with the property that all the induced maps of local homologies on a stratum are isomorphisms are called \define{homology stratifications}.
\end{example}
Local homology as a means to obtain stratifications of point clouds (or combinatorial objects) have recently been investigated in several works (\cite{bendich2012local,skraba2014approximating, fasy2016exploring, nanda2020local, stolz2020geometric,mileyko2021another}). Both \cite{bendich2012local} and \cite{nanda2020local} make use of the structure maps $\locHlg{\topSp} \to \locHlgy{\topSp}$ to determine the strata. 
Note, however, that in the case of two strata it suffices to study the isomorphism type at each point, and there is no need to study the maps themselves, as stated by the following lemma.
\begin{lemma}\label{lem:strat_by_localcohom}
Let $ \str = (\topSp, \phiStrat \pp \topSp \to \{p < q\})$ be a Whitney stratified space (more generally conically stratified space) with manifold strata of dimension $q$ and $p$ respectively. Then $s$ is a homology stratification. \\
\\Furthermore, if the local homology of $\topSp$ is different from $\textnormal{H}_{\bullet}(\mathbb R^q;0)$, at each $y \in \str_p$, then $s$ is the only homology stratification of $\topSp$ with two strata.\\
\\
Conversely,
one always obtains a homology stratification $\tilde{\phiStrat}:\topSp \to \{p< q\}$ defined by:
\[\tilde{\phiStrat}(x) = q \iff \locHlg{\topSp} \cong  \textnormal{H}_{\bullet}(\mathbb R^q;0),\]
for $x \in \topSp$. 
\end{lemma}
\begin{proof}
See \cref{pf:lem:strat_by_locacohom}.
\end{proof}
Now, let us consider the scenario of working with a (potentially noisy) sample $\sampX$ instead of considering the whole space $\topSp$. Even when working persistently, to obtain non-trivial information, one can not pass all the way to the limit, when computing local homology. Indeed, for any thickening $\thicken{\sampX}$, $\alpha >0$, $\locHlg{\thicken{\sampX}} = \textnormal{H}_{\bullet}(\supSp; 0)$. Instead, one considers persistent local homology of the sample, with respect to a parameter $\frac{1}{\zoomParam}$, specifying the radius of the ball representing $U_x$ (see \cite{bendich2012local}, \cite{skraba2014approximating}). In other words, one computes persistent local homology using the spaces
\[
( \thicken[\alpha]{\sampX}, \thicken[\alpha]{\sampX} \setminus \ballEpso[\frac{1}{2\zoomParam}]{x}).
\]
For computational reasons, it is beneficial to use the intrinsically local notion of this structure. By the excision theorem, one may equivalently work with:
\[
( \thicken[\alpha]{(\sampX \cap \ballEps[\frac{1}{\zoomParam}]{x})},   \thicken[\alpha]{(\sampX \cap \ballEps[\frac{1}{\zoomParam}]{x})} \setminus \ballEpso[\frac{1}{2\zoomParam}]{x}).
\]
If one does not want the resulting barcodes to become shorter as $\zoomParam \to \infty$ and instead wants a measure of singularity that is comparable for different scales, then this needs to be normalized, and one may instead compute persistent homology via the stretched pair
\[
( \thicken[\alpha]{( \zoomParam \sampX \cap \ballEps[1]{\zoomParam x})},   \thicken[\alpha]{(\zoomParam \sampX \cap \ballEps[1]{ \zoomParam x})} \setminus \ballEpso[\frac{1}{2}]{ \zoomParam x}).
\]
Let us take a bit more of a conceptual look on this procedure in the following remark.
\begin{remark}\label{rem:analysis_of_loc_hom}
The procedure we just described may abstractly be rephrased as follows. We want to obtain a stratification of $\sampX$ using local data.  Hence, we only consider sets of the form
\[ \zoomParam \sampX \cap \ballEps[1]{\zeta x}. \]
By shifting into the origin, we may equivalently investigate the space
\[\magnification{\sampX}:= \zoomParam (\sampX-x) \cap \ballEps[1]{0} \subset \supSp, \]
with $\sampX -x = \{y-x \mid y \in \sampX \}$.
We can think of $\magnification{\sampX}$ as zooming into $\sampX$ at $x$ by a magnification parameter $\zoomParam$. We then want to determine how far from a $q$-dimensional euclidean unit disk $D^q \subset \mathbb R^q \hookrightarrow\supSp$ the space $\magnification{\sampX}$ is. In the particular case of persistent local homology, we apply the map
\[
\elocPersHlg: M \mapsto \{ \alpha \mapsto \hlg{ \thicken[\alpha]{M}, \thicken[\alpha]{M} \setminus \ballEpso[\frac{1}{2}]{0}  }\}
\]
to obtain a persistence module indexed over $[0,\frac{1}{2})$ and thus a quantitative invariant. The interleaving distance to $\elocPersHlg(D^q) \cong \elocPersHlg(\mathbb R^q)$ then gives a quantitative measure of singularity.
\end{remark}
\subsection{Magnifications and $\Phi$-stratifications}
Let us now put our observations on persistent local homology made in the beginning of this section and especially in \cref{rem:analysis_of_loc_hom} into a more abstract framework.
\begin{definition}\label{def:sam_centr}
Denote by $\SamB$ the (symmetric Lawvere) metric space
\[
\SamB:=\{\sampX \mid \sampX \subset \supSp \},
\]
equipped with the following truncated version of the Hausdorff distance: We pull back the metric on $\Sam$ along
\begin{align*}
    \SamB &\to \Sam \\
    \sampB &\mapsto \ballEps[1]{0}\cap \sampB .
\end{align*}
We call $\SamB$ the \define{space of local samples} (of $\supSp$), and denote its metric by $\dmet[\SamB]{-}{-}$.
\end{definition}
\begin{remark}
Note that the way the metric on $\SamB$ is defined, it automatically identifies a local sample with its intersection with a unit ball around the origin. 
Indeed, $\SamB$ is by definition isometric to the space of subspaces of $\ballEps[1]{0} \subset \supSp$. One may as well have used the latter, however, that involves a series of inconvenient truncations, so the above perspective is notationally preferable. In particular, in this context it makes sense to write $V \in \SamB$, for $V \subset \supSp$ a linear subspace.
\end{remark}
Next, we define the magnified spaces which showed up in our analysis of local homology in \cref{rem:analysis_of_loc_hom}.
\begin{definition}\label{def:magnification}
Let $\sampX \in \Sam$, $x \in \supSp$ and $\zoomParam>0$. We denote by 
\[
 \magnification{\sampX}:= \zoomParam{(\sampX -x )} \cap \ballEps[1]{0} \in \SamB,
\]
with $\sampX = \{ y -x \mid y \in \sampX \}$, 
the \define{$\zoomParam$-magnification of $\sampX$ at $x$}.
\end{definition}
Let us assume for a second that $\sampX = \topSp$ and the latter admits a locally conelike stratification (as in \cref{ex:loc_homology}), such that we need not worry about zooming in too far. Then, theoretically speaking, to make sure we identify every locally Euclidean region as such, we want the information obtained to be as local as possible, i.e. we want to consider the case $\zoomParam \to \infty$. Local homology, as described in \cref{rem:analysis_of_loc_hom}, defines a continuous map on $\SamB$. Hence, to understand the behavior of local persistent homology for $\zoomParam \to  \infty$ it suffices to understand the behavior of $\magnification{\topSp}$, for $\zoomParam \to \infty$. The following example illustrates when this limit can be used to determine local singularity.
\begin{example}\label{ex:cross_and_cusp}
Consider the two real algebraic varieties \[
X = \{(x_1,x_2) \in \mathbb{R}^2 \mid x_1^4 - x_1^2 + x_2^2 = 0 \}\]
and \[ Y = \{(x_1,x_2) \in \mathbb{R}^2 \mid ((x_1+0.5)^2 + x_2^2 - 0.25)((x_1-0.5)^2 + (x_2)^2 - 0.25) = 0 \}.\] These varieties are Whitney stratified spaces with the singular set containing only the origin. In \cref{fig:hippo_cross}, we show magnifications of $X$ at the origin $x = (0,0)$, i.e. $\magnification{X}$ for three different $\zeta \in \{1,3,45\}$. We can observe that the homeomorphism type of the magnifications stabilizes as we increase $\zeta$. In the limit the spaces $\magnification{X}$ converge (in Hausdorff distance for $\zoomParam \to \infty$), to a space of the same homeomorphism type. In contrast, $Y$ shows a different convergence behavior. Although the spaces $\magni[x]{Y}$ share the same homeomorphism type with the magnifications of $X$ at the origin, for $\zeta$ large enough, \cref{fig:hippo_tangent} illustrates that the homeomorphism type changes when passing to the limit (see also \cref{fig:tangentcones}).
\begin{figure}
    \centering
    \begin{minipage}{0.32\textwidth}
        \includegraphics[width=1\textwidth]{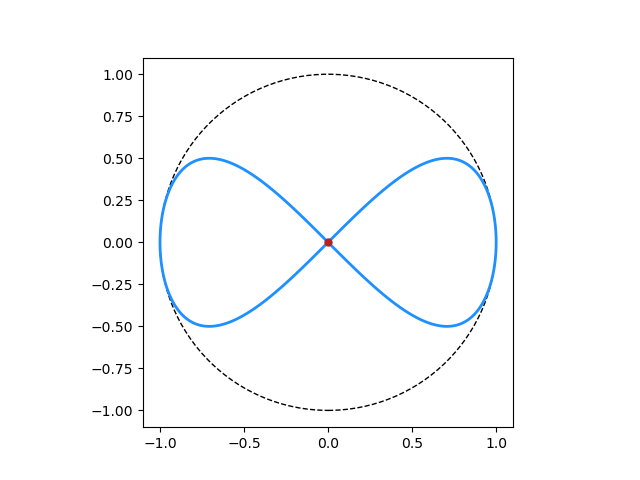}
    \end{minipage}\hfill
    \begin{minipage}{0.32\textwidth}
        \centering
        \includegraphics[width=1\textwidth]{Figures/lem_magni1_circ.png}
    \end{minipage}
    \begin{minipage}{0.32\textwidth}
        \centering
        \includegraphics[width=1\textwidth]{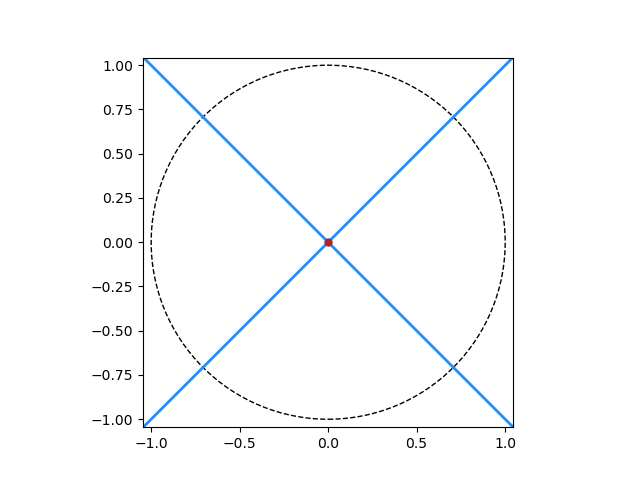}
    \end{minipage}
\caption{Three magnifications of $X$ at the origin}
\label{fig:hippo_cross}
\end{figure}

\begin{figure}
    \centering
    \begin{minipage}{0.32\textwidth}
        \centering
        \includegraphics[width=1\textwidth]{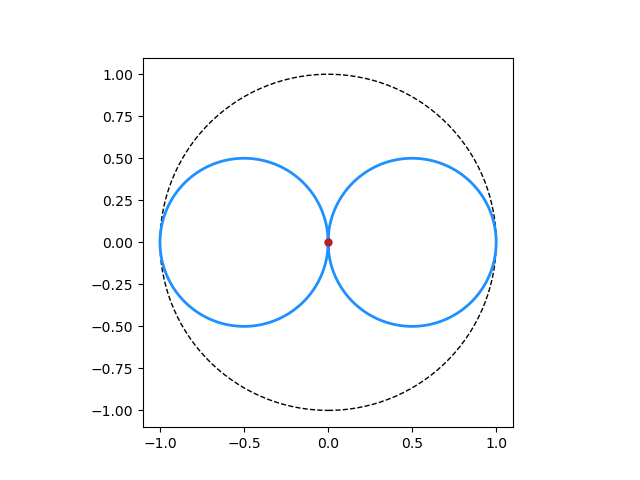}
    \end{minipage}\hfill
    \begin{minipage}{0.32\textwidth}
        \centering
        \includegraphics[width=1\textwidth]{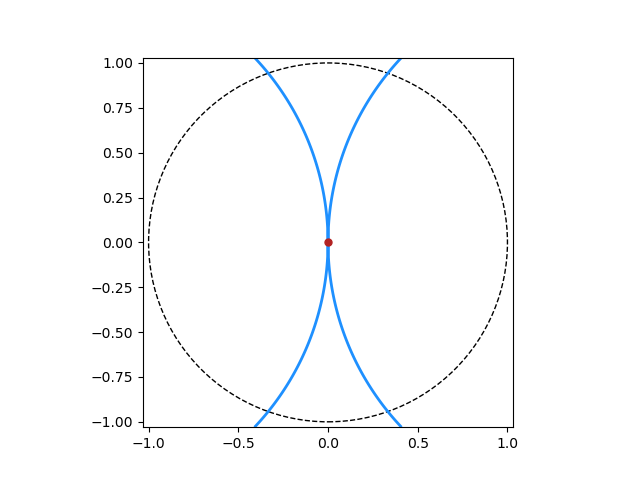}
    \end{minipage}
    \begin{minipage}{0.32\textwidth}
        \centering
        \includegraphics[width=1\textwidth]{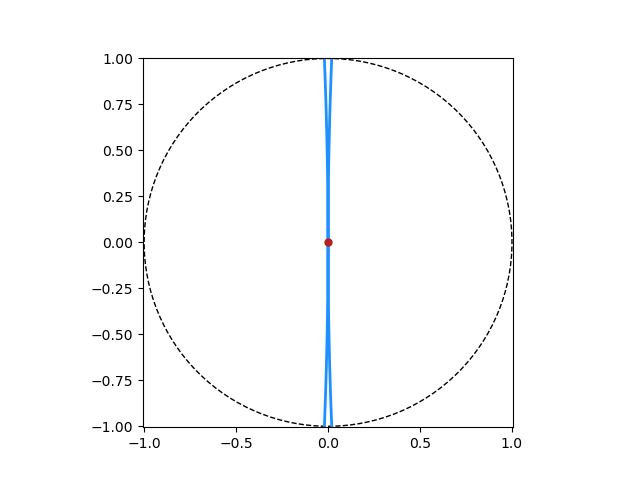}
    \end{minipage}      
\caption{Three magnifications of $Y$ at the origin.} \label{fig:hippo_tangent}
\end{figure}
\end{example}
If $\topSp$ admits a (subanalytic) Whitney stratification, then limit spaces of magnifications (in $\SamB$) exist and are known as the (extrinsic) tangent cones of $\topSp$ at $x$.
For a more detailed investigation of metric tangent cones see \cite{LytchakDiff,bernig2007tangent}. For our purpose, the following definition will suffice.
 \begin{definition}\label{def:tangent_cone}
    Let $\topSp \subset \supSp$. The \define{(extrinsic) tangent cone of $\topSp$ at $x \in \topSp$} is defined as
\[
\tangentCo[x]{\topSp} := \{v \in \supSp \mid \forall \varepsilon > 0 \exists y \in \ballEps{x}\cap \topSp:  v \in \thicken[\varepsilon]{(\mathbb R_{\geq 0} (y-x))}\}.
\]
The extrinsic tangent cones define a map
\begin{align*}
    \freetangentCo{\topSp}: \topSp &\to \SamB \\
    x &\mapsto \tangentCo[x]{\topSp}.
\end{align*}
\end{definition}
\begin{example}\label{ex:taylor}
By Taylor's expansion theorem one has 
\[
\tangentCo{\topSp} =  \tangentCo{U} = \tangentSp[x]{U}
\]
where $U \subset \topSp \subset \supSp$ is a neighborhood of $x$ in $X$ and furthermore $U$ is a smooth submanifold of $\supSp$.
\end{example}
\begin{example}\label{ex:cplx_var_tangent}
For an (affine) complex algebraic variety $\topSp$ the tangent cone at the origin coincides with the algebraic tangent cone, i.e. the
set of common zeroes of all polynomials in the ideal generated by the homogeneous elements of lowest degree of all polynomials that vanish identically on $\topSp$.
\end{example}
It is a classical result (see e.g. \cite{hironaka1969normal}, \cite{bernig2007tangent}) that when $\topSp$ admits a subanalytic Whitney stratification, then
\[
\magnification{\topSp} \xrightarrow{\zeta \to \infty} \tangentCo[x]{\topSp}
\]
in $\SamB$. Since we are mostly interested in the case of what happens when one replaces $\topSp$ by samples and needs uniform versions of this result, we will recover this result as a special case of \cref{cor:magni_conv}. However, it already points at what kind of information one may expect to obtain when one uses local features such as local persistent homology obtained from magnifications to stratify a data set. In the limit $\zoomParam \to \infty$ one can only expect to extract information that is contained in the extrinsic tangent cones. This leads to the following definition. 
\begin{definition}\label{def:tangentially_stratified} Let $\pos= \{p <q \}$.
Let $\whs = (\topSp, \phiStrat) \in \SamP$ be a $q$-dimensional Whitney stratified space. We say that $\whs$ is \define{tangentially stratified} if 
\[ \dmet[\SamB]{\tangentCo[x]{\whs}}{V} > 0, \]
for all $x \in \whs_p$ and for all $V \subset \supSp$ $q$-dimensional linear subspaces of $\supSp$.
\end{definition}
Tangentially stratified spaces are precisely the type of Whitney stratified spaces for which we may expect to recover stratifications by using magnifications with large $\zoomParam > 0$. That this holds true rigorously is essentially the content of \cref{subsec:restrat}.
\begin{example}\label{ex:whitney_not_tang_strat}
Not every Whitney stratified space is tangentially stratified. Consider again $ \topSp[Y] = \{(x_1,x_2) \in \mathbb{R}^2 \mid ((x_1+0.5)^2 + x_2^2 - 0.25)((x_1-0.5)^2 + (x_2)^2 - 0.25) = 0 \}$ from \cref{ex:cross_and_cusp}. In this case, the above condition specifies to $ \dmet[\SamB]{\tangentCo[(0,0)]{\topSp[Y]}}{V} > 0, $
for all $1$-dimensional linear subspaces $V \subset \supSp$. The tangent cone of $Y$ at the origin is a $1$-dimensional linear space given by \[\tangentCo[(0,0)]{\topSp[Y]} = \{(x_1,x_2) \in \mathbb{R}^2 \mid x_1  = 0 \},\] see \cref{fig:tangentcones} on the right, which already serves as a linear subspace $V\subset \mathbb{R}^2$ such that $ \dmet[\SamB]{\tangentCo[(0,0)]{\topSp[Y]}}{V} = 0$. For the space \[\topSp = \{(x_1,x_2) \in \mathbb{R}^2 \mid x_1^4 - x_1^2 + x_2^2 = 0 \} \] on the other hand we find that the tangent cone at the origin is given by \[\tangentCo[(0,0)]{\topSp}\{(x_1,x_2) \in \mathbb{R}^2 \mid (x_1 + x_2)(x_2-x_1) = 0 \},\] see \cref{fig:tangentcones} on the left. Clearly, there is no $1$-dimensional linear subspace $V \subset \mathbb{R}^2$ such that $\dmet[\SamB]{\tangentCo[(0,0)]{\topSp}}{V} = 0$.
\begin{figure}
    \centering
    \begin{minipage}{0.5\textwidth}
        \centering
        \includegraphics[width=1\textwidth]{Figures/lem_tangent.png}
    \end{minipage}\hfill
    \begin{minipage}{0.5\textwidth}
        \centering
        \includegraphics[width=1\textwidth]{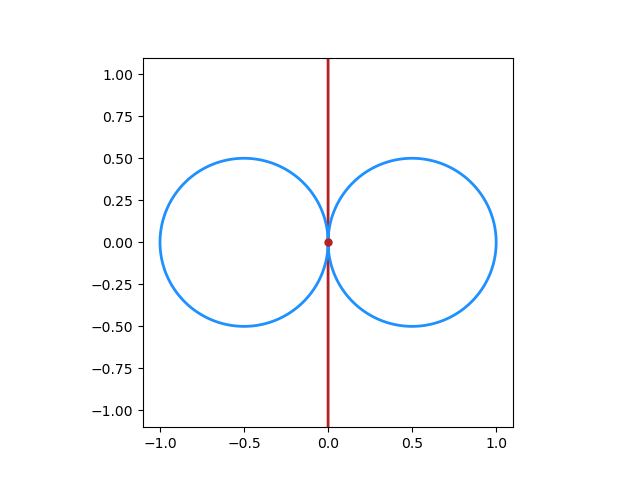}
    \end{minipage}     
\caption{Two curves with their respective tangent cones (red) at their singular stratum} \label{fig:tangentcones}
\end{figure}

\end{example}
In practice, we may want to use other local invariants, such as local persistent homology in \cref{rem:analysis_of_loc_hom}, to identify singular points as in \cref{lem:strat_by_localcohom}. This leads to the following definition. Again, for the remainder of this subsection, let $\pos= \{p <q \}$.
\begin{definition}\label{def:Phi_strat}
Let $\Phi: \SamB \to [0,1]$ be a continuous function, such that $\Phi(V,0)=1$, whenever $V$ is a $q$-dimensional linear subspace of $\supSp$. 
Let $\whs \in \SamP$ be $q$-dimensional Whitney stratified space. We say that $\whs$ is \define{(tangentially) $\Phi$-stratified} if \[ \Phi(\tangentCo[x]{\whs})   < 1, \]
for all $x \in \whs_p$.
\end{definition}
Let us begin with some examples of functions $\Phi: \SamB \to [0,1]$ which may be used to detect singularities.
\begin{example}\label{ex:phi-strat_hausdorff}
Consider the continuous map
\begin{align*}
    \Phi_q \pp \SamB &\to [0,1] \\
    \sampB &\mapsto 1-\inf\{ \dmet[\SamB]{\sampB}{V} \},
\end{align*}
where $V$ ranges over the $q$-dimensional linear subspaces of $\supSp$.
A $q$-dimensional Whitney stratified space $\whs \in \SamP$ is tangentially stratified if and only if it is $\Phi_q$- stratified. 
$\Phi_q$ is thus universal in the sense that if $\whs$ is $\Phi$-stratified for some $\Phi$ as in \cref{def:Phi_strat}, then it is $\Phi_q$-stratified. 
\end{example}
\begin{example}\label{ex:phi-strat_pers_cohom}
Persistent local homology can be used as a function $\Phi$, as was done similarly in \cite{bendich2012local, stolz2020geometric, nanda2020local, mileyko2021another}. Precisely, we use $\elocPersHlg[i]: \SamB \to \Vect^{[0, \frac{1}{2})} $ as defined in \cref{rem:analysis_of_loc_hom}.
Consider a linear embedding $\mathbb R^q \subset \supSp$, allowing us to write $\mathbb R^q \in \SamB$ and set
\begin{align*}
    \Phi\pp \SamB &\to [0,1] \\
    \sampB &\mapsto 1- 2\max_{i \leq q}\dmet[\In]{\elocPersHlg[i](\sampB)}{\elocPersHlg[i](\mathbb R^q)}.
\end{align*}
Indeed, as no bar in $\elocPersHlg[i]$ can be longer than $\frac{1}{2}$, this function is well defined.
Let $\whs \in \SamP$ be a (definable) $q$-dimensional Whitney stratified such that for all $x \in \whs_p$ we have $\elocPersHlg(\tangentCo[x]{\whs}) \neq \elocPersHlg(\mathbb R^q)$, i.e. the two persistence modules are not isomorphic. Then, $\whs$ is a $\Phi$-stratified space. 
\end{example}
One of the advantages of allowing for different $\Phi$ than just the universal one is that in practice one may use a series of rougher invariants which may be easier to compute.
\begin{example}\label{cref:phi_strat_one_sided}
Instead of using 
\[
\sampB \mapsto 1-\inf\{ \dmet{\sampB}{V} \},
\]
as in \cref{ex:phi-strat_hausdorff},
one can only use half of the numbers used to compute Hausdorff distance, i.e. only consider 
\[
\sampB \mapsto \inf\{\varepsilon \mid   \ballEps[1]{0} \cap \sampB \subset  \ballEps[1]{0} \cap \thicken[\varepsilon]{V}\}.
\]
Note that this definition of $\Phi$ will identify points in the boundary of a smooth manifold as regular. While this decreases the class of $\Phi$-stratified spaces, this $\Phi$ can generally be easier to compute when using optimization techniques to find an optimal $V$. 
Similarly, instead of computing $\elocPersHlg(\sampB)$ as in \cref{ex:phi-strat_pers_cohom} one may use a Vietoris-Rips version of the latter as described in \cite{skraba2014approximating} or only consider certain dimensions. 
\end{example}
\subsection{\LW stratified spaces}
The previous section illustrates that in order to reconstruct stratifications from sample data we have to obtain a better understanding of the convergence properties of the magnification spaces to tangent cones. Such results are the content of \cref{subsec:conv_tang}. Before we investigate these, we need a series of results on Whitney stratified spaces which are definable with respect to particularly well behaved o-minimal structures.
Our methods heavily rely on the work of \cite{hironaka1969normal} and \cite{bernig2007tangent}. However, note that the results there are local, while ours are more global in nature, and that we consider the case of magnifications of samples as well. We use the following result due to Hironaka, which gives us additional control over integral curves.
\begin{lemma}\label{lem:approximation_lemma}
Let $\whs \to \pos$ be a Whitney stratified space, $p \in \pos$ and $y \in \whs_p$.
Suppose there exists $d_0 > 0$ such 
that there exists $\alpha>0$, with
\[\funpqbeta(x,y) \leq \norm{y - x}^{\alpha}
\]
for all $x \in \whs_{\geq p} \cap \ballEps[d_0]{y}$.
Then, for any $C > 0$, there exists $d >0$ only depending on $d_0,\alpha$ (and the dimension of $\whs$), such that 
for any integral curve $\phi: [0,d] \to \whs$ starting in $y$ and ending in $\ballEps[d]{y}$ the inequality
\[\norm{\frac{1}{t}(\phi(t) - \phi(0)) - \frac{1}{s}(\phi(s) - \phi(0))} \leq C \vert t-s\vert ^{\alpha}
\]
holds
for all $t,s \in [0,d]$. In particular, all integral curves starting at $y$ are differentiable at $0$. 
\end{lemma}
\begin{proof}
A complete proof of this statement can be found in \cite{hironaka1969normal}.
\end{proof}
Spaces fulfilling a local version of the above condition were investigated in \cite{hironaka1969normal}. It was called the strict Whitney condition there.
\begin{definition}\label{def:Loj_Whitney}
A Whitney stratified space fulfilling the requirements of \cref{lem:approximation_lemma} on any compactum $K$ contained in some pure stratum $\whs_p$ of $\whs$, is called a \define{Lojasiewicz-Whitney stratified space}. That is, $\whs$ is called \LW stratified, if the following condition holds. 
Let $K \subset \whs_p$ be a compact, definable subset of some stratum $\whs_p$ of $\whs$.
Then there exist $\alpha >0, d_0>0$ such that
\[
\funpqbeta(x,y) \leq \norm{y - x}^{\alpha},
\]
for all $y \in K$ and $x \in \whs\cap \ballEps[d_0]{y}$. 
\end{definition}
In other words, \LW stratified spaces are Whitney stratified spaces for which the speed at which secant lines diverge from the tangent spaces is bounded by some root.
It turns out that most of the definably stratified spaces one is interested in i.e. compact subanalytic or semialgebraic are \LW stratified (\cref{prop:beta_globally_bounded}). 
\begin{recollection}
    Recall that an o-minimal structure is called polynomially bounded, if for all $f\pp \mathbb R \to \mathbb R$ definable with respect to the structure, there exists an $n \in \mathbb N$
    such that
    \[
    \vert f(t)\vert  \leq \vert t\vert ^n,
    \]
    for $t$ sufficiently large. Polynomially bounded structures include the structure of semialgebraic sets and finitely subanalytic sets (see \cite{van1986generalization} and \cite{miller1994expansions}). In particular, any compact subanalytically definable stratified space is definable with respect to a polynomially bounded o-minimal structure.
\end{recollection}
A proof of the following statement can be found in \cref{proof:beta_globally_bounded}. 

\begin{proposition}\label{prop:beta_globally_bounded}
Let $\whs$ be a Whitney stratified space which is definable with respect to some polynomially bounded o-minimal structure. Then, $\whs$ is \LW stratified.
\end{proposition}
\begin{remark}\label{rem:results_hold_for_gen_P}
    In this section and in the \cref{subsec:conv_point,subsec:conv_tang}, there is no need to restrict to the two strata case. The results hold for general $\pos$.
\end{remark}
As an almost immediate consequence of \cref{lem:approximation_lemma,prop:beta_globally_bounded}, we obtain:
\begin{proposition}\label{prop:char_tangent_cone}
Let $\whs$ be a \LW stratified space. Then, for any $x \in \whs$, every integral curve starting at $x$ is differentiable in $0$. Furthermore, we have
\[
 \tangentCo[x]{\whs} \cap \partial \ballEps[1]{x}=  \overline{\{ \phi'(0) \mid \phi \textnormal{ is an integral curve starting at } x \}}.
\]
Hence, 
\[
\tangentCo[x]{\whs} = \overline{\{\alpha \phi'(0) \mid \phi \textnormal{ is an integral curve starting at } x, \alpha \geq 0\}}.
\]
\end{proposition}
\begin{proof}
First, note that $\tangentCo[x]{\whs}$ is closed by definition.
The containment of the right hand side in the left hand side is immediate by definition of the derivative (compare to \cref{prop:existent_of_integral_curves}). For the converse inclusion, let $v \in \tangentCo[x]{\whs} \cap \partial\ballEps[1]{x}$. For $\varepsilon >0$ small enough, we have $y \in \whs_{\geq p} \cap \ballEps[\varepsilon]{x}$, with $p = s(x)$, such that
\[
    \norm{v - \lambda(y-x)} < \varepsilon,
\]
for some $\lambda \geq 0$.
In particular, we also obtain
\[
 \big \vert 1-\lambda \norm{y-x} \big\vert  < \varepsilon.
\]
Now, $t= \norm{y-x}$ and let $\phi: [0,d] \to \whs$ be the integral curve starting at $x$ and passing through $y$.
We then have
\begin{align*}    \norm{v - \phi'(0)} 
    & \leq \norm{v - \lambda(y-x)} + \norm{\lambda(y-x) - \frac{y-x}{\norm{y-x}} } + \norm{\frac{y-x}{\norm{y-x}} - \phi'(0)}\\
    &= \norm{v - \lambda(y-x)} + \big \vert 1-\lambda\norm{y-x} \big\vert  + \norm{\frac{\phi(t)-x}{t} - \phi'(0)} \\
    &\leq \varepsilon + \varepsilon + C\varepsilon^{\alpha},
\end{align*}
for some $C, \alpha>0$ independent of the choices above.
In particular, we can choose $\phi$ such that $\phi'(0)$ is arbitrarily close to $v$.
\end{proof}
We can now obtain the following key technical result, to investigate the convergence behavior of magnifications of samples.
\begin{proposition}\label{prop:loc_sampl_convergence_of_magnifications}
Let $\whs$ be a \LW stratified space over $\pos$, with underlying space $X\subset \supSp$. Let $p \in \pos$ and $K \subset \whs_p$ be a compact subset.
Then, there exist $d, C, \alpha >0$ such that the following holds. \\
For all $\zoomParam$ such that $\frac{1}{\zoomParam} \in [0,d]$ there exists $\varepsilon_0 >0$, such that 
\[
\dmet[\SamB]{\tangentCo[x]{\whs}}{\magnification[w]{\sampX}} \leq C ( \zoomParam^{-\alpha} + \zoomParam \varepsilon), 
\]
for $\sampX \in \Sam$ with $\dmet[\HD]{\sampX}{X} = \varepsilon \leq \varepsilon_0$, $w \in \supSp$ and $x \in K$ with $\vert x-w\vert  \leq \varepsilon$.
\end{proposition}
\begin{proof}
Denote $r := \frac{1}{\zoomParam}$.
We work with the non-normalized spaces instead, that is instead of working in the unit ball of $\magnification{\sampX}$, we work in the ball of radius $r$ in $\sampX$. Furthermore, without loss of generality let $x = 0$.
Again, choose $d, C', \alpha$ as in \cref{lem:approximation_lemma}, possibly slightly decreasing $d$, such that the requirements on $r$ still hold for $r + 2\varepsilon$.
Let $c \in \tangentCo[x]{\whs}$ with $\vert c\vert  \leq r$. Let $\tilde c := \frac{r -2\varepsilon}{r} c$. 
We have \[
\vert c - \tilde{c}\vert  \leq 2 \varepsilon .
\]
Next, using \cref{prop:char_tangent_cone}, consider the integral curve starting in $0$ with initial direction $\frac{c}{\vert c\vert }$, $\phi: [0,d] \to \whs$ (or, by passing to the limit if necessary a curve with initial direction arbitrarily close to $\frac{c}{\norm{c}}$). We then have
\[
\norm{ \tilde{c} - \phi( \norm{ \tilde c })}  \leq C' r^{\alpha +1} 
\]
and 
\begin{equation}\label{equ:lots_of_equ_2}
    \vert \phi( \vert  \tilde{c}\vert  ) - w\vert  \leq \vert \tilde{c}\vert  + \varepsilon \leq r - \varepsilon.
\end{equation}
Choose $w' \in \sampX$ with $\vert w' - \phi( \norm{\tilde{c}}) \vert \leq \varepsilon$. Then, by \cref{equ:lots_of_equ_2} $w' \in \ballEps[r]{w} \cap \sampX$. Summarizing, we have 
\[
\vert c + w - w'\vert  \leq 2 \varepsilon + \varepsilon + C'r^{\alpha +1} + \varepsilon \leq C( r^{\alpha +1} + \varepsilon), 
\]
for appropriate $C>0$. \\
\\
Conversely, let $ w' \in \sampX$ with $\vert w - w'\vert  \leq r$. By assumption, we find $y \in \whs$ with $\vert y - w'\vert  \leq \varepsilon$ and have $\vert y\vert  \leq r + 2\varepsilon$. Thus, for $\phi_y$ the integral curve starting in $0$ through $y$ we have 
\[
\vert \vert y\vert  \phi_{y}'(0) - y\vert  \leq C'({r + 2 \varepsilon})^{\alpha + 1}.  
\]
Take $c = ({r-\varepsilon}) \frac{\vert y\vert }{r+2\varepsilon} \phi_{y}'(0) \in \tangentCo[x]{\whs} \cap \ballEps[r - \varepsilon]{x}$. 
Note, that $\vert c+w\vert  \leq \vert w\vert  + \vert c\vert  \leq r$ i.e. $c+w \in \ballEps[r]{w} \cap (\tangentCo[x]{\whs} +w )$.
We further have
\[
\vert c - \vert y\vert  \phi'_{y}(0)\vert  \leq \vert y\vert  (1 - \frac{r-\varepsilon}{r + 2 \varepsilon}) \leq  3\varepsilon.
\]
Summarizing, we have
\begin{align*}
\vert  c + w - w'\vert  \leq \varepsilon + \vert c  - w'\vert  &\leq \varepsilon + \vert c-\vert y\vert  \phi'_{y}(0)\vert  + \vert \vert y\vert  \phi'_{y}(0)-y\vert  + \vert y-w'\vert  \\
&\leq \varepsilon + 4 \varepsilon + C'( r + 2 \varepsilon)^{\alpha +1} \\
&\leq C( r^{\alpha +1} + \varepsilon)
\end{align*}
for appropriate $C >0$ and $\varepsilon < r/2$. We obtain the result by multiplying with $\zoomParam$ to pass to the magnification.
\end{proof}
As a first corollary of \cref{prop:loc_sampl_convergence_of_magnifications}, we obtain that the tangent cones of a \LW stratified space vary continuously on each stratum.
\begin{proposition}\label{cor:tangent_var_con}
Let $\whs $ be a \LW stratified space over $\pos$ and $p \in \pos$. Then, the map
    \begin{align*}
        \tangentCo[-]{\whs}:\whs_p &\to \SamB \\
        x &\mapsto \tangentCo[x]{\whs}
    \end{align*}
is continuous. 
\end{proposition}\label{cor:magni_conv}
\begin{proof}
To see this, note that by \cref{prop:loc_sampl_convergence_of_magnifications}, restricted to any compactum, $\tangentCo[-]{X}$ is the uniform limit of the family of maps given by $f_{\zoomParam}\pp x \mapsto \magnification{\whs}$. By exhausting $\whs_p$ by compacta it suffices to see that the $f_{\zoomParam}$ are continuous for $\zoomParam$ large enough.  Again, set $r= \frac{1}{\zoomParam}$.
So, let $K \subset \whs_p$ be a compactum and let $r$ be small enough, such that $\neiEps[r]{K} \cap \whs \subset \whs_{\geq p}$. In other words, we may assume without loss of generality that $\whs_p$ is the minimal stratum of $\whs$.
 Next, note that 
 \begin{equation}\label{equ:bound_on_microhood}
      \dmet[\SamB]{\magnification{\whs}}{\magnification[x']{\whs}}  \leq  r \dmet[\HD]{\ballEps[r]{x} \cap \whs}{\ballEps[r]{x'} \cap \whs} + \norm{x-x'},
 \end{equation}
for $x,x' \in \whs_p$.
By an application of Thom's isotopy lemma, the map
 \begin{align*}
     \hat g: \whs \times \whs_p &\to [0,\infty) \times \whs_p\\
     (x,y) &\mapsto (\norm{x-y},y)
 \end{align*}
restricts to a fiber bundle over $(0, r] \times \whs_p$ for $r$ small enough. In particular, it follows that if we set $X =\{ (x,y) \in \whs \times \whs_p  \mid \norm{x-y}\leq r \}$, we obtain an induced fiber bundle
\begin{align*}
    g: X &\to \whs_p \\
    (x,y) &\mapsto y
\end{align*}
with fiber $\ballEps[r]{y} \cap \whs$ over $y$. Again, locally using the independence of the Hausdorff-distance topology of the choice of metric, we obtain that $\ballEps[r]{y} \cap \whs$ varies continuously in $y$. Hence, by \cref{equ:bound_on_microhood} so does $\magnification[y]{\whs}$.
\end{proof}
\subsection{Pointwise convergence of magnifications of a sample}\label{subsec:conv_point}
As an immediate consequence of \cref{prop:loc_sampl_convergence_of_magnifications}, we obtain that for a \LW stratified space $\whs$ we have 
\[
\magnification{\whs} \xrightarrow{\zoomParam \to \infty} \tangentCo{\whs},
\]
for all $x \in \whs$. This result can already be found in similar form in \cite{hironaka1969normal}.
What we want to do, however, is to describe the case occurring in application. That is, we aim to analyse the convergence behavior of magnifications for samples of $\topSp$, as $\zoomParam \to \infty$. At first glance, this is a nonsensical question. For a fixed sample $\sampX$, $\magnification{\sampX}$ has distance $0$ to a one point (or empty) space, when $\zoomParam$ is large enough. Instead, the correct notion of convergence is already suggested by the inequality in \cref{prop:loc_sampl_convergence_of_magnifications}. What needs to be described is a convergence behavior where the quality of the sample is allowed to improve at the same time as $\zoomParam \to \infty$.
\begin{notation}\label{not:convergence_sampling_mult}
    Given a function $f \colon M \times (0, \infty) \to T$, where $M$ is a metric space and $T$ a topological space, we write 
    \[
    \sampconvmo{f(\sampX, \zoomParam)}{Y}{\topSp}
    \]
    for $\topSp \in M$ and $Y \in T$ to state that for any pair of sequences $\zoomParam_n \in (0, \infty)$ converging to $\infty$, and $\sampX_n \in M$, such that $\zoomParam_n\dmet{\sampX_n}{\topSp}$ converges to $0$, the sequence $f(\sampX_n,\zoomParam_n)$ converges to $Y$.
\end{notation}
\begin{remark}
    We may think of the type of convergence in \cref{not:convergence_sampling_mult}, as convergence of $f(\sampX,\zoomParam)$ to $Y$, for $\sampX \to X$ and $\zoomParam \to \infty$, under the additional condition that the convergence in the $\sampX$ variable is faster than the convergence in the $\zoomParam$ variable. This corresponds to the idea that if we want to zoom in further by a magnitude of $k$, and investigate some point locally, the quality of the sample also needs to improve by more than this magnitude $k$, so that we do not zoom in too far and end up only considering a single point. We can think of this as a notion of convergence in $\zoomParam$, while improving the quality of the sample. Hence, we will also speak of \define{convergence while sampling}.
\end{remark}
Now, we can interpret \cref{prop:loc_sampl_convergence_of_magnifications} with $\sampX = X$, $x=w$ and $K= \{x\}$ as the following convergence while sampling result.
\begin{corollary}\label{cor:samp_conv_magnifications}
Let $\topSp \in \Sam$ be a \LW stratifiable space. Let $x \in \topSp$. Then,
\[
 \sampconvm{\magnification{\sampX}}{\tangentCo[x]{\topSp}}{\topSp}.
\]
Furthermore, this convergence is uniform on any compactum $K$ contained in a stratum.
\end{corollary}
\subsection{Convergence of tangent bundles}\label{subsec:conv_tang}
To prove a global recovery of stratifications result, we need to obtain a more global version of \cref{cor:samp_conv_magnifications}. For this we need to treat tangent cones not as separate spaces but as a (stratified) bundle of cones. To describe the resulting convergence result, we need a space of samples of bundles.
\begin{definition}\label{def:Bsam}
Denote by $\SamBun$ the set \[ \{(\sampX,  F:\sampX \to \SamB) \mid\sampX \in \Sam \},\] equipped with the (extended pseudo) metric given by regarding $F$ as a subset of $\supSp \times \SamB$, equipping the latter with the product metric, and then using the resulting Hausdorff distance.\\
That is, for $(\sampX,F), (\sampX',F') \in \SamBun$, we define
\begin{align*}
    \dmet[\SamBun]{(\sampX, F)}{(\sampX', F')} :=
\max_{(\sampX_0, \sampX_1) \in \{\sampX, \sampX'\}^2} 
\inf \{ \varepsilon>0 \mid & \forall x \in \sampX_0 \exists y \in \sampX_1: \norm{x-y}, \\ &\dmet[\sampB]{F_0(x)}{F_1(y)} \leq \varepsilon \}.
\end{align*}
We also refer to $\SamBun$ as the \define{space of bundle samples} (of $\supSp$).
\end{definition}
\begin{definition}
The \define{$\zoomParam$-magnification bundle of $\sampX \in \Sam$} is defined as the image of $\sampX$ under the map
\begin{align*}
    \emagnification: \Sam &\to \SamBun \\
    \sampX \mapsto & (\sampX, \{ x \mapsto \magnification{\sampX} \} ).
\end{align*}
The \define{tangent cone bundle of $\topSp \in \Sam$}
is defined as the image of $\topSp$ under the map
\begin{align*}
    \efreetangentCo: \Sam &\to \SamBun \\
    \topSp \mapsto & (\topSp, \{ x \mapsto \tangentCo{\topSp} \} ).
\end{align*}
\end{definition}
\begin{remark}
Note that the nomenclature warrants some care, as for an arbitrary space $\sampX$, neither $\emagnification$ nor $\freetangentCo{\sampX}$ are anything close to a fiber bundle and even for a \LW stratified space they are stratified fiber bundles at best. 
\end{remark}
Note that \cref{prop:loc_sampl_convergence_of_magnifications} does not imply convergence of magnification bundles in the metric on $\SamBun$, as the convergence is only uniform on compacta contained in pure strata. However, we may equip the spaces $\SamBun$ with alternative topologies, allowing us to formulate notions of convergence on a compactum.  Again, for the remainder of this subsection let $\pos = \{ p < q\}$.
\begin{construction}\label{con:convergence_on_compacta}
Let $K \in \Sam$ and let $T$ be any of the spaces $\SamNP$, $\SamBun$. Let $\varepsilon: \Sam \to \mathbb R_+$ be some continuous map.
Define a map 
\begin{align*}
    g^{K}_{\varepsilon}: T &\to T \\
     (\sampX,f) &\mapsto (\sampX \cap K_{\varepsilon({\sampX})}, f\mid_{K_{\varepsilon({\sampX})}}).
\end{align*}
If $\mathcal{K}=(E,\varepsilon)$ is a pair consisting of a set $E \subset \Sam$, together with a continuous map $\varepsilon: \Sam \to \mathbb{R}_+$, we denote by $T^{\mathcal{K}}$, the space with the same underlying set as $T$, but equipped with the initial topology with respect to the maps $g^{K}_{\varepsilon}$ and $T \to \Sam \xrightarrow{\varepsilon} \mathbb{R}_{+}$. In particular, with respect to this topology, a sequence $\sampB_n = (\sampX_n,F_n)$ in $T$ converges to $\sampB = (\sampX, F) \in T$, if and only if \[g^{K}_{\varepsilon}{(\sampB_n)} \xrightarrow{n \to \infty} g_{\varepsilon}^K{(\sampB)},\] for all $K \in E$ and \[\varepsilon(\sampX_n) \xrightarrow{n \to \infty} \varepsilon(\sampX)
 .\]
\end{construction}
\begin{remark}\label{rem:countability}
In the case where $E$ is a countable set, the topology on $T^{\mathcal K}$ is still first countable. All cases we consider here can be reduce to this scenario. Alternatively, all of the proofs using sequences below work identically when using nets instead of sequences.
\end{remark}
We can now rephrase \cref{prop:loc_sampl_convergence_of_magnifications} as a global convergence result, which is essential for the stratification learning theorems of \cref{subsec:restrat}.
\begin{proposition}\label{prop:tangent_co_conv_compacta}
Let $\topSp \in \Sam$ be equipped with a \LW stratification $\whs= (\topSp, \topSp \to \pos)$. Denote $\varepsilon:= \dmet[\HD]{\topSp}{-}$ .
Let $E \subset \Sam$ be such that for all $K \in E$ there exist a decomposition into compacta $K = K_{p} \sqcup K_{q}$ such that $K_p \subset \whs_p$, $K_q \subset \whs_q$. Denote $\mathcal{K}= (E, \varepsilon)$.
Then, 
\[
\sampconvm{\freemagnification{\sampX}}{\freetangentCo{\topSp}}{\topSp} \textnormal{ in $\SamBun^{\mathcal K}$.}
\]
\end{proposition}
\begin{proof}
Let $K \in E$. We need to show
\[
\sampconvm{g^{K}_{\varepsilon}(\freemagnification{\sampX})}{g^{K}_{\varepsilon}({\freetangentCo{\topSp}})}{\topSp}.
\]
Note that since $K  \subset \whs$, $g^{K}_{\varepsilon}({\freetangentCo{\topSp}}) = \freetangentCo{\topSp}\vert_{K}$. The result is now an immediate consequence of \cref{prop:loc_sampl_convergence_of_magnifications}.
\end{proof}
\subsection{The stratification learning theorem}\label{subsec:restrat}
We now have all the tools in place to recover stratifications from samples. We have seen in \cref{prop:pers_strat_htpy_type_C-lipschitz} that the persistent stratified homotopy type is (Lipschitz) continuous in compact Whitney stratified spaces $\whs$ over $\pos= \{ p < q \}$. In particular, we can approximate the persistent stratified homotopy type of $\whs$ from a stratified sample $\sampW$ close to $\whs$ in the metric on $\SamP$. In practice, we can generally only expect to be given non-stratified samples. Even naively, if one had a means to decide when a point has ended up precisely in the singular stratum, one should expect the latter to be a $0$-set with respect to the used density, and hence usually end up with non-stratified sets.
Nevertheless, our investigations of magnifications and $\Phi$-stratifications already suggest that local tangent cones may be used to recover stratifications which approximate the original one. Let us first illustrate how the procedure works in case one is given a perfect sample, i.e. one can work with the whole of $\whs$. Again, for the remainder of this section let $\pos = \{ p <q \}$.
\begin{construction}\label{con:recover_strat_A}
Let $\whs \in \SamP$ be a compact \LW $\Phi$-stratified space, with respect to a function $\Phi$ as in \cref{def:Phi_strat}. Suppose we forget the stratification of $\whs = ( \topSp, \phiStrat)$, and only have the data given by $\topSp$. We can then associate to $\topSp$ its tangent cone bundle $ \efreetangentCo{\topSp} \in \SamBun$. Next, we use the function $\Phi$ to decide which regions should be considered singular. We can do so by applying $\Phi$ to $\efreetangentCo{\topSp}$ fiberwise. As a result we obtain a strong stratification $\tilde s$ of $\topSp$, given by
\[
x \mapsto \tangentCo[x]{\topSp} \mapsto \Phi( \tangentCo[x]{\topSp} ).
\]
By \cref{cor:tangent_var_con}, this map is continuous on all strata. In particular, by assumption, it takes a maximum value $m<1$ on $\whs_p$. Since $\whs_q$ is a manifold, we have 
\[
\tangentCo[x]{\topSp} =  \tangentCo[x]{\whs_q} =  \tangentSp{\whs_q} = \mathbb R^q
\]
for $x \in \whs_q$, and thus the strong stratification has constant value $1$ on $\whs_q$.
Therefore, we may recover the stratification of $s$ by choosing $u > m$ and applying $\eforgetStr$:
\[
\eforgetStr( \topSp, \tilde{s}) = \whs.
\]
\end{construction}
We now replicate the procedure described in \cref{con:recover_strat_A} in case of working with samples and investigate its convergence behavior.
\begin{lemma}\label{lem:Phi_cont_A}
Let $\bigPhi: \SamB \to [0,1]$ be a continuous map. Then, the induced map
\begin{align*}
\bigPhi_*: \SamBun \to \SamNP \\
(\sampX, F) \mapsto (\sampX, \Phi \circ F)
\end{align*}
is continuous. Even more, if $\bigPhi$ is $C$-Lipschitz, then so is $\bigPhi_*$.
\end{lemma}
\begin{proof}
 Since $\SamB$ is isometric to the space of compact subspaces of $\ballEps[1]{0} \subset \supSp$ and thus compact, $\bigPhi$ is a uniformly continuous map. Hence, the result follows immediately by definition of the metrics on $\SamBun$ and $\SamNP$.
\end{proof}
It turns out $\bigPhi_*$ also descends to a continuous map on the alternative topologies of \cref{con:convergence_on_compacta}.
\begin{lemma}\label{lem:Phi_cont_B}
Let $\bigPhi: \SamB \to [0,1]$ be a continuous map. Let $E \subset \Sam$, $\varepsilon: \Sam \to \mathbb{R_+}$ be some continuous function and $\mathcal{K} = (E, \varepsilon)$. Then, the map
\begin{align*}
\bigPhi_*: \SamBun^{\mathcal{K}} \to \SamNP^{\mathcal K} \\
(\sampX, F) \mapsto (\sampX, \Phi \circ F)
\end{align*}
is continuous. \end{lemma}
\begin{proof}
By definition of the topologies on $\SamBun^{\mathcal{K}} \to \SamNP^{\mathcal K}$, it suffices to show the result for the case where $E=\{K\}$ is a singleton. 
Continuity of $\SamBun^{\mathcal{K}} \to \SamNP^{\mathcal K} \to \Sam \xrightarrow{\varepsilon} \mathbb R_{+}$ holds trivially. Next, note that the diagram
\[
\begin{tikzcd}
\SamBun^{\mathcal{K}} \arrow[d, "g^{K}_{\varepsilon}"] \arrow[r, "\Phi_*"]& \SamNP^{\mathcal K} \arrow[d, "g^{K}_{\varepsilon}"]\\
\SamBun \arrow[r, "\Phi_*"]& \SamNP
\end{tikzcd}
\]
trivially commutes, since the $g$ are given by restricting, i.e. precomposition and $\Phi_*$ by postcomposition.

Then, for a sequence $\sampB_n \in \SamBun^{\mathcal K}$ and $\sampB \in \SamNP^{\mathcal{K}}$ we have:
\begin{align*}
    \sampB_n \xrightarrow{ n \to \infty} \sampB \textnormal{ in $\SamBun^{\mathcal K}$} &\iff g^{K}_{\varepsilon}( \sampB_n) \xrightarrow{n \to \infty} g^{K}_{\varepsilon} (\sampB) \textnormal{ in $\SamBun$} \\&\implies \Phi_*(g^{K}_{\varepsilon}  (\sampB_n)) \xrightarrow{n \to \infty} \Phi_* (g^{K}_{\varepsilon} (\sampB)) \textnormal{ in $\SamNP$}
    \\&\iff g^{K}_{\varepsilon} (\Phi_* (\sampB_n)) \xrightarrow{n \to \infty} g^{K}_{\varepsilon} (\Phi_* (\sampB)) \textnormal{ in $\SamNP$} \\
&\iff \Phi_* (\sampB_n) \xrightarrow{n \to \infty} \Phi_* (\sampB) \textnormal{ in $\SamNP^{\mathcal{K}}$}, 
\end{align*}
where the implication in the second line follows by \cref{lem:Phi_cont_A}.
\end{proof}
We have already seen, that with respect to the alternative topologies the magnification bundles do indeed converge uniformly to the tangent cone bundle. 
This is however not the case with the usual topologies. 
Hence, to approximate stratifications using a magnification version of \cref{con:recover_strat_A}, we need to show that $\eforgetStr$ is continuous in the respective tangent cone bundles with respect to the alternative topology.
\begin{proposition}\label{prop:forgetstr_cont_wrt_new_topology}
Let $\stratSp = ( \topSp, \phiStrat) \in \SamNP$, $\topSp$ compact. Let $\spaceParam \in [0,1]$ be such that $\subLevel{\stratSp}{\spaceParam}$ is closed and such that $\subLevel{\stratSp}{\spaceParam \pm \delta} = \subLevel{\stratSp}{\spaceParam}$ for $\delta$ sufficiently small. Let $\varepsilon = \dmet[\HD]{\topSp}{-}$. Finally, let 
\[
\mathcal{K}= (\{ K \in \Sam \mid K =  K_p \sqcup K_q, K_p, K_q \textnormal{ compact}, K_p \subset \subLevel{\stratSp}{\spaceParam}, K_q \subset s^{-1}{(u,1]}\} , \varepsilon).
\]
Then 
\[
\eforgetStr\colon \SamNP^{\mathcal K} \to \SamP
\]
is continuous at $\stratSp$.
\end{proposition}
\begin{proof}
Let $\stratSamp = (\sampX, s') \in  \SamNP^{\mathcal K}$.
Note that convergence in $\SamP$ may be verified componentwise. Since convergence in $\SamNP^{\mathcal K}$ also implies $\varepsilon(\sampX) = \dmet[\HD]{\topSp}{\sampX} \to 0$, we only need to verify convergence in the component $\subLevel{\stratSamp}{\spaceParam}$.
We have 
\[
\dmet[\HD]{\stratSp_{\leq\spaceParam}}{ \stratSamp_{\leq\spaceParam}}\leq \dmet[\HD]{\stratSp_{\leq\spaceParam}}{ \stratSamp_{\leq\spaceParam} \cap K_{\varepsilon{(\sampX)}}} + \dmet[\HD]{\stratSp_{\leq\spaceParam}} { \stratSamp_{\leq\spaceParam}  \cap ({\stratSp_{\leq\spaceParam}})_{\gamma}},
\]
whenever $K = \overline{(\topSp - ({\stratSp_{\leq\spaceParam}}) _{\frac{\gamma}{2}} )} \sqcup \stratSp_{\leq\spaceParam}$
and $\gamma >0$ such that, $\sampX \subset K_{\varepsilon{(\sampX)}} \cup ({\stratSp_{\leq\spaceParam}})_{\gamma} $. Note, that for this to hold, it suffices that $\varepsilon(\sampX) \leq \frac{\gamma}{2}$. 
For the left summand we obtain,
\begin{align*}
    \dmet[\HD]{\stratSp_{\leq\spaceParam}}{ \stratSamp_{\leq\spaceParam} \cap K_{\varepsilon{(\sampX)}}} &= \dmet[\HD]{\forgetStr{g^K_{\varepsilon}(\stratSp)}}{\forgetStr{g^K_{\varepsilon}(\stratSamp)}} 
\leq \dmet[\SamNP]{{g^K_{\varepsilon}(\stratSp)}}{{g^K_{\varepsilon}(\stratSamp)}}, 
\end{align*}
by \cref{prop:cont_if_const_u}, for $\varepsilon({\sampX})$ sufficiently small and $g^K_{\varepsilon}(\stratSamp)$ close to $g^K_{\varepsilon}(\stratSp)$. For the other summand we first split the Hausdorff distance into the directed distances
\[
\dmet[\HD]{\stratSp_{\leq\spaceParam}} { \stratSamp_{\leq\spaceParam}  \cap ({\stratSp_{\leq\spaceParam}})_{\gamma}} \leq d_L(\stratSp_{\leq\spaceParam},\stratSamp_{\leq\spaceParam}  \cap ({\stratSp_{\leq\spaceParam}})_{\gamma}) + d_L(\stratSamp_{\leq\spaceParam}  \cap ({\stratSp_{\leq\spaceParam}})_{\gamma},\stratSp_{\leq\spaceParam})
\]
where $d_L(A,B) = \inf \{\delta \geq 0 \mid A \subset B_{\delta} \}$. Then, the second summand is bounded by $\gamma$ and for the first summand we observe that 
\[
d_L(\stratSp_{\leq\spaceParam},\stratSamp_{\leq\spaceParam}  \cap ({\stratSp_{\leq\spaceParam}})_{\gamma}) \leq d_L(\stratSp_{\leq\spaceParam},\stratSamp_{\leq\spaceParam}  \cap ({\stratSp_{\leq\spaceParam}})_{\varepsilon(\sampX)}).
\]
This is due to the fact that $\varepsilon(\sampX) < \gamma$ and $\stratSamp_{\leq\spaceParam}  \cap ({\stratSp_{\leq\spaceParam}})_{\varepsilon(\sampX)} \subset \stratSamp_{\leq\spaceParam}  \cap ({\stratSp_{\leq\spaceParam}})_{\gamma}$. If we set $K' = \subLevel{\stratSp}{\spaceParam}$ and invoke \cref{prop:cont_if_const_u} again we obtain
\begin{align*}
    d_L(\stratSp_{\leq\spaceParam},\stratSamp_{\leq\spaceParam}  \cap (\stratSp_{\leq\spaceParam})_{\varepsilon(\sampX)}) &= d_L(\forgetStr{g^{K'}_{\varepsilon}(\stratSp)},\forgetStr{g^{K'}_{\varepsilon}(\stratSamp)}) \\
&\leq \dmet[\SamNP]{g^{K'}_{\varepsilon}(\stratSp)}{g^{K'}_{\varepsilon}(\stratSamp)} ,
\end{align*}
for $g^{K'}_{\varepsilon}(\stratSamp)$ close to $g^{K'}_{\varepsilon}(\stratSp)$.
Summarizing, we have:
\[
\dmet[\HD]{\stratSp_{\leq\spaceParam}}{\stratSamp_{\leq\spaceParam}} \leq \dmet[\SamNP]{g^K_{\varepsilon}(\stratSp)}{g^K_{\varepsilon}(\stratSamp)} + \dmet[\SamNP]{g^{K'}_{\varepsilon}(\stratSp)}{g^{K'}_{\varepsilon}(\stratSamp)} + \gamma.
\]
In particular, we may first fix some $\gamma$ while the other terms converge to $0$ for $\stratSamp\to \stratSp$ in $\SamNP^{\mathcal K}$ by assumption. Since $\gamma$ can be taken arbitrarily small, the result follows.
\end{proof}
We are now finally in shape to define a map which equips samples with stratifications, depending on their approximate tangential structure.
\begin{definition}\label{def:zeta_phi_strat}
Let $\Phi:\SamB \to [0,1]$ be a continuous map and $\spaceParam \in [0,1)$, $\zoomParam \in \mathbb{R}_{+}$.
Let $\sampX \in \SamB$.
We call the image of $\sampX$ under the composition
\[
  \ePhiStr:  \Sam \xrightarrow{\efreemagnification} \SamBun \xrightarrow{\Phi_*} \SamNP \xrightarrow{\eforgetStr} \SamP
\]
the \define{ $\zoomParam$-th $\Phi$-stratification of $\sampX$ (with respect to $\spaceParam$).} In the case where $\zoomParam= \infty$, replace $\efreemagnification$ by $\efreetangentCo$. 
\end{definition}

\begin{example}\label{ex:zeta_phi_strat}
To illustrate the concepts in \cref{def:zeta_phi_strat} let us walk through every component of the composition defining $\ePhiStr$ for a specific sample. Let $\topSp$ denote the algebraic variety given by
\begin{equation}\label{eqn:cyclide}
\{ (x,y,z) \in \mathbb{R}^3 \mid (x^2 + y^2 + z^2 + 1.44)^2 - 7.84x^2 + 1.44y^2 = 0 \}.
\end{equation}
In the bottom left of \cref{fig:zeta_phi_strat}, a visual representation of $\topSp$ can be found. A finite sample from this variety, denoted $\sampX$, was obtained by randomly picking points from an enclosing rectangular cuboid and only keeping points that satisfy \eqref{eqn:cyclide} up to a small error. Choosing a magnification parameter $\zoomParam = 5$ we obtain the magnification bundle $\efreemagnification(\sampX)$ for $\sampX$, depicted in the top middle of \cref{fig:zeta_phi_strat}. $\Phi$ was chosen as in \cref{ex:phi-strat_pers_cohom}. Evaluating the fibers of $\efreemagnification(\sampX)$ we obtain a strongly stratified sample $\Phi_* (\efreemagnification(\sampX))$, shown on the left of \cref{fig:zeta_phi_strat}. Next, picking the threshold value $\spaceParam \in [0,1)$ to be $0.83$ induces a stratified sample via $\eforgetStr$. A visual comparison indicates that the resulting stratified sample is close to the Whitney stratified space given by $\topSp$ with two isolated singularities. This already points at the convergence behavior predicted by \cref{thrm:recovery_thrm}.
\begin{figure}[H]
\centering

\includegraphics[width=1\textwidth,trim={3cm 15pt 2cm 0pt},clip]{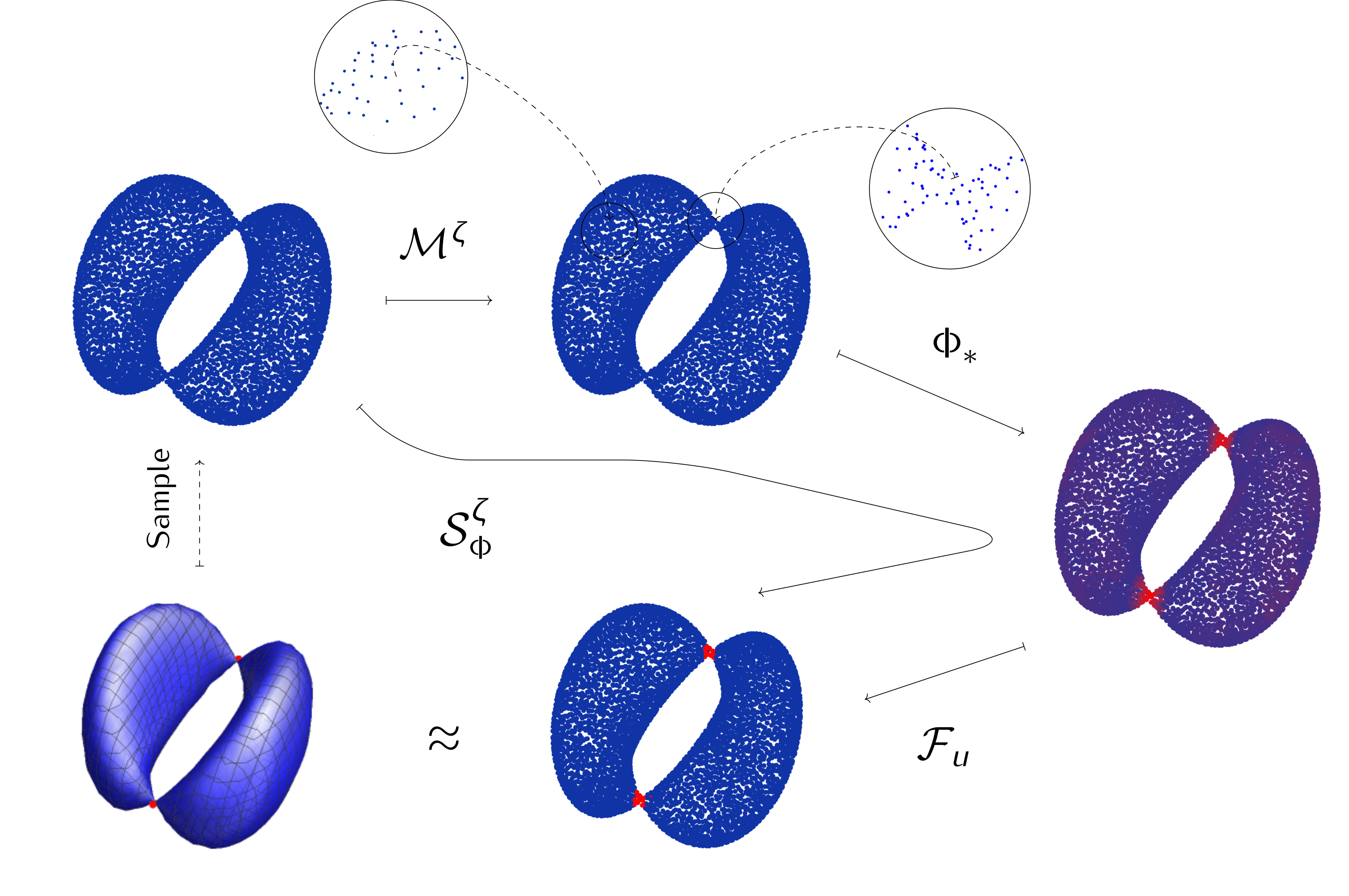}
\caption{Illustration of $\ePhiStr$ for a sample from a 2-dimensional real algebraic variety}
\label{fig:zeta_phi_strat}

\end{figure}

\end{example}
Using \cref{def:zeta_phi_strat}, we can restate the content of \cref{con:recover_strat_A} as follows.
\begin{proposition}\label{prop:recovering_strat_wo_convergence}
Let $\whs \in \SamP$ be a \LW stratified space, $\Phi$-stratified with respect to $\Phi: \SamB \to [0,1]$ as in \cref{def:Phi_strat}.
Then,  \[\sup\{\Phi(\tangentCo[x]{\topSp}) \mid x \in \whs_p \} < 1.\] In particular, 
\[
\PhiStr[\infty]{\topSp} = \whs,
\]
for $\sup\{\Phi(\tangentCo[x]{\topSp}) \mid x \in \whs_p \} < \spaceParam <  1$.
\end{proposition}
\begin{proof}
This was already covered in \cref{con:recover_strat_A}.
\end{proof}
We can now finally state the main theorem about approximating the stratification of a \LW $\Phi$-stratified space $\whs$. In practice, it guarantees that for $\zoomParam$ large enough and given a sufficiently good sample one can use the $\zoomParam$-th $\Phi$-stratification to approximate the stratified space $\whs$. In particular, this result can be applied to all compact, subanalytically Whitney stratified spaces.
\begin{theorem}\label{thrm:recovery_thrm} Let $\pos = \{ p < q\}$ and let $\whs=(\topSp, \topSp \to \pos) \in \SamP$ be a compact \LW stratified space,  $\Phi$-stratified with respect to $\Phi: \SamB \to [0,1]$. Then there exists $u_0 \in (0,1)$ such that 
\[
\sampconvm{\PhiStr{\sampX}}{\whs}{\topSp},
\]
for $\spaceParam \in [u_0, 1)$.
\end{theorem}
\begin{proof}
Let $\mathcal{K}$ be as in \cref{prop:tangent_co_conv_compacta}. It is the content of the latter proposition that 
\[
\sampconvm{\freemagnification{\sampX}}{\freetangentCo{\topSp}}{\topSp} \textnormal{ in $\SamBun^{\mathcal K}$}.
\]
Applying $\Phi_*$ to this and using \cref{lem:Phi_cont_B}, we obtain
\[
\sampconvm{ \Phi_* \circ \freemagnification{\sampX}}{\Phi_* \circ\freetangentCo{\whs}}{\topSp} \textnormal{ in $\SamNP^{\mathcal K}$}.
\]
Now, note that $\Phi_* \circ \freetangentCo{\topSp}$ fulfills the requirements of \cref{prop:forgetstr_cont_wrt_new_topology}, if we take $1 > u > \max\{ \Phi_* \circ \freetangentCo{\topSp}(x) \mid x \in \topSp \}$. Hence,
\[
\sampconvm{\PhiStr{\sampX} = \eforgetStr \circ \Phi_* \circ \freemagnification{\sampX}}{\eforgetStr \circ \Phi_* \circ \freetangentCo{\topSp} = \whs,}{\topSp}
\]
where the equality follows by \cref{prop:recovering_strat_wo_convergence}.
\end{proof}
Finally, we can now combine this result with \cref{cor:flexible_pers_strat_Lipsch} and \cref{cor:continuity_of_pers_tame} which guarantees that $\ePhiStr$ may be used to infer stratified homotopy types from non-stratified samples. Note that in the following we again assume $\whs$ to be linearly rescaled in such a way that it is cylindrically stratified. Equivalently, this does not need to be assumed if $\paramSpace$ is reparametrized by the scaling factor.
\begin{corollary}\label{cor:composed_convergence}
Let $\pos= \{p <q\}$ and let $\whs=(\topSp, \topSp \to \pos) \in \SamP$ be a compact \LW stratified space,  $\Phi$-stratified with respect to $\Phi: \SamB \to [0,1]$. Then there exists $u_0 \in (0,1)$ such that
\[
\sampconvm{\epers \circ \PhiStr{\sampX}}{\pers{\whs}}{\topSp},
\]
for $\spaceParam \in [u_0, 1)$. Furthermore,
\[
\sampconvm{\epersParam \circ \PhiStr{\sampX}}{\persParam{\whs}}{\topSp},
\]
for $\spaceParam \in [u_0, 1)$ and $\diagParam \in \paramSpace$.
\end{corollary}
\section{Conclusion}\label{sec:conclusion}
The central advantage of the approach to stratified TDA we have described in this work is that it is highly modular. In summary, it can be decomposed into three steps. 
\begin{enumerate}
    \item From non-stratified data obtain stratified data (\cref{subsec:restrat}).
    \item From stratified data obtain a persistent stratified homotopy type (\cref{subsec:def_pers_strat}).
    \item From a stratified homotopy type compute algebraic invariants.
\end{enumerate}
The goal of this work was to show the feasibility of the first two steps in the restricted case of two strata. Our results in \cref{sec:pers_strat,sec:recover_strat} show that the resulting notion of persistent stratified homotopy type fulfills many of the properties required in application (\ref{enum:properties_of_PH1}, \ref{enum:properties_of_PH2} and \ref{enum:properties_of_PH3}), which are fulfilled by the classical persistent homotopy type, such as stability (\cref{prop:pers_strat_htpy_type_C-lipschitz}), computability (\cref{rem:computability_of_strat_pers}) and the availability of inference results (\cref{prop:thicken_stab_glob,thrm:recovery_thrm,cor:composed_convergence}).
There are a series of promising avenues arising from this first step in persistent stratified homotopy theory.
\begin{enumerate}
    \item So far, our constructions are mostly developed for the case of two strata. In the introduction, we have already described in some detail why we decided to restrict to this scenario. Nevertheless, for possible applications, the case of multiple strata seems of great interest. We are aware that there is currently ongoing research concerning how to recover stratifications in the case of arbitrary posets, which could greatly increase the possible realm of application. At the same time, such an approach would also require a generalization of the inference and stability results of  \cref{subsec:def_pers_strat,subsec:stab_pers_type,subsec:stab_pers_type_whit} to persistent stratified homotopy types with more than two strata. Proofs of such are expected to be inductive in nature, which suggests an inductive approach to stratified homotopy theory on the theoretical side. This has yet to be established in detail. 
    \item While our results in this work are mostly theoretical, we are currently working on implementing the stratification learning method and persistent stratified homotopy types on a computer. One possible next step is then to apply these methods to inherently singular data sets such as retinal artery photos, and investigate the stability and expressiveness of our approach in practice. This also requires a more detailed study and evaluation of choices of functions $\Phi$, for the construction of $\Phi$-stratifications (see \cref{def:Phi_strat}) in an applied scenario.
    \item The application of persistent stratified homotopy types to real-world data also requires a further investigation of the last step - i.e. passing to algebraic invariants such as persistent homology. While there are some expressive and well-understood algebraic invariants at hand - for example the persistent homology of the links and strata - there is a series of more intricate invariants to consider. These include a persistent version of intersection homology, as well as an interpretation of the persistent stratified homotopy type as a multi-parameter persistence module. Studying the properties of such invariants, ranging from computability to expressiveness, leaves much room for future research projects both in theory as well as in application.
\end{enumerate}

\printbibliography

\appendix
\section{Some details on abstract homotopy theory}
\begin{remark}\label{rem:order_ho_fun}
There are some subtleties to be considered, which come down to the order in which one passes to the persistent and homotopical perspective. 
We emphasize that by $\ho \textbf{T}^{I}$, for some indexing category $I$, we mean the localization of the functor category at pointwise weak equivalences, and not the functor category $(\ho \textbf{T})^{I}$, obtained by localizing at weak equivalences first. The universal property of the localization induces a canonical functor 
\begin{align*}
    \ho \textbf{T}^{I} \to (\ho \textbf{T})^{I}.
\end{align*}
This functor is essentially never an equivalence of categories. For example, for $\textbf{T} = \Top$ with the usual class of weak equivalences, the notion of isomorphism on the left-hand side is fine enough to compute homotopy limits and colimits. This is not the case on the right-hand side (see for example \cite{hirschhorn2003modelcat}, for an introduction to the theory). Generally, the functor will be neither essentially surjective nor fully faithful. Essential surjectivity, for example, comes down to whether or not a homotopy commutative diagram is equivalent to an actual commutative diagram (see \cite{dwyer1984kandiag} for a detailed discussion.) \\
To see that faithfulness is generally not the case, consider replacing $\mathbb R_+$ by $I = \{ 0 < 1\}$, and taking $\textbf{T} = \Top$, $D = \{* \to S^1 \}$ and $D' = \{* \to X\}$, for some pointed space $X$.
Both objects may be considered as pointed spaces.
Then, the hom-objects from $D$ to $D'$ in  $\ho \textbf{T}^{I}$ are the homotopy groups of $X$. In $(\ho \textbf{T})^{I}$, however, the hom-object is given by free homotopy classes from $S^1$ to $X$, i.e. by the abelianization of the homotopy group of $X$.\\
In the special case where $I = \mathbb R_+$ and $\textbf{T} = \Top$ this leaves, a priori, an ambivalence by what one means by a persistent homotopy type. Given a persistent space, i.e. an object in $\Top ^{\mathbb R_+}$, one can either consider its isomorphism class in $\ho( \Top^{\mathbb R_+})$ or in $(\ho \Top)^{\mathbb R_+}$. We argue that the former is the conceptually better notion since properties \ref{enum:properties_of_PH1} to \cref{enum:properties_of_PH3} may already be stated on this level. At the same time, due to the comparison functor between the two categories, results obtained in $\ho( \Top^{\mathbb R_+})$ are generally stronger than results in $(\ho \Top)^{\mathbb R_+}$. \\
However, one should note that when passing to the algebraic world by applying homology index-wise, both perspectives agree. Finally, we may add that for most applications the difference is negligible. This is a consequence of \cref{prop:almost_comm_of_ho_and_diag} which, among other things, implies, as long as one restricts to persistent objects which are tame in the sense that their homotopy type only changes at finitely many points, then the functor 
\begin{diagram}
    \ho \textbf{T}^{\mathbb{R}_+} \to (\ho \textbf{T})^{\mathbb R_+}
\end{diagram}
induces a bijection on isomorphism classes. In particular, there is no difference in the resulting notion of persistent homotopy type.
\end{remark}
\begin{lemma}
\label{prop:almost_comm_of_ho_and_diag}
Let $\mathcal M$ be a (simplicial) model category and $I$ be any small indexing category. Then,
\begin{align*}
  F:  \ho \mathcal M^{I} \to (\ho \mathcal M)^{I}
\end{align*}
reflects isomorphisms. Furthermore, let $I$ be a finite, totally ordered poset. Then $F$
is essentially surjective and full. In particular, two objects in $\ho \mathcal M^{I}$ are isomorphic, if and only if their images under $F$ are isomorphic.
\end{lemma}
\begin{proof}
Ultimately, this comes down to the fact that homotopy coherent diagrams of the particularly simple shapes involved are easy to understand.
The proof requires a series of standard arguments in the theory of model categories. See \cite{hirschhorn2003modelcat} for a comprehensive overview.
To see that the functor characterizes isomorphisms, note that a morphism in a functor category is an isomorphism, if and only if is so pointwise. Since a morphism descends to an isomorphism in the homotopy category, if and only if it is a weak equivalence, and weak equivalences in $\mathcal M^{I}$ (equipped with any of the usual model structures) are defined pointwise, this shows that a morphism in $\mathcal M^{I}$ is an isomorphism in $\ho(\mathcal M^{I})$ if and only if its image under $F$ is an isomorphism.\\
The next statement holds in general. However, we show only the case of a simplicial model category, since all model categories relevant in this paper fulfill this property and the availability of a canonical cylinder object makes the proof somewhat more digestible.
Essential surjectivity is immediate, as functors $D$ defined on a totally ordered set $I \cong [n]=\{ 0, ..., n\}$ are entirely determined by their values valued on $D(i \leq i+1)$ and conversely any sequence of morphisms $X_i \to X_{i+1}$, uniquely determines a functor. Thus, (up to an isomorphism in the right-hand side category) being a functor with values in the homotopy category $\ho \mathcal M ^{I}$ is equivalent to being a functor with values in $\mathcal M^{I}$.\\ Now, to see fullness, consider objects $D$ and $D'$ on the left-hand side. Without loss of generality, we may assume that $D$ is a cofibrant and $D'$ a fibrant object with respect to the injective model structure (which exists since $I$ is a Reedy category in the obvious fashion. See \cite{hirschhorn2003modelcat} for an introduction to Reedy model structures.) 
In particular, morphisms in the homotopy categories between $D$ and $D'$ and between $D_i, D_i'$ are given by (simplicial) homotopy classes.
We now proceed to show fullness by induction over $n$. The case $n = 0$ is trivial. Now let $f: D \to D'$ be a morphism in  $(\ho \mathcal M)^{I}$, i.e. we are given a homotopy commutative diagram
\begin{diagram}
    D_0 \arrow[d, "f_0"] \arrow[r] & \dots \arrow[r] & D_n \arrow[d, "f_n"] \arrow[r, "i_n"] & D_{n+1} \arrow[d, "f_{n+1}"] \\
    D'_0 \arrow[r] & \dots \arrow[r] & D'_n \arrow[r, "i'_n"] & D'_{n+1} 
\end{diagram}
by inductive assumption, we can assume that up to $D_n$ the diagram is actually commutative. It remains to show, that there exists a morphism $g_{n+1}$, simplicially homotopic to $f_{n+1}$, such that the right-hand square commutes on the nose. Let $H: D_n \to D'^{\Delta^1}_{n+1}$ be the adjoint to the simplicial homotopy from the right down to the down right composition.
We may instead solve the induced lifting problem
\begin{diagram}
D_n \arrow[r, "H"] \arrow[d, "i_n"] & {D_{n+1}'}^{\Delta^{1}} \arrow[d, "ev_1"]\\
D_{n+1} \arrow[r, "f_{n+1}"'] \arrow[ru, dashed]& D'_{n+1}
\end{diagram}
such a lift exists, since, by assumption, the left-hand side map is a cofibration and the right-hand side map is a fibration. Thus, we have shown fullness.
\end{proof}
\begin{lemma}\label{lem:appendix_constant_diagram}
Let $\mathcal{M}$ be a relative category (i.e. a category equipped with a notion of weak equivalence). Let $U \subset \paramSpace \times \mathbb R_+$ be a subset containing $\paramSpace \times \{0 \}$. Let $D \in \mathcal{M}^{U}$, be such that \[D(\diagParam, 0) \to D(\diagParam, \alpha)\] is a weak equivalence,
for all $(\diagParam, \alpha) \in U$. Let $D\mid_{\paramSpace \times \{0\}}$ be homotopically constant of value $M \in \mathcal M$. Then $D$ is also homotopically constant of value $M$.
\end{lemma}
\begin{proof}
This follows from the fact that $\paramSpace$ is initial in $U$. 
Let $i\pp \paramSpace \times \{ 0\} \hookrightarrow U$ be the inclusion. Note that, in this specific scenario
\[
i_*(D')(\diagParam, \alpha) = D'_{(\diagParam, 0)}
\]
since any slice involved in the right Kan-extension have a terminal object of the form $(\diagParam,0)$. In particular, this means that $i_*$ preserves weak equivalences between all objects. Furthermore, by assumption, this equality implies that the natural transformation
\[
i_* D\mid_{(\paramSpace, 0)} \to D
\]
is a weak equivalence. Now, by assumption, $D\mid_{(\paramSpace, 0)}$ is weakly equivalent to some constant functor $C$ in $\mathcal {M}^{\paramSpace \times \{0 \}}$. In particular, this implies that there is a zigzag of weak equivalences 
\begin{diagram}
C \arrow[r] & D' \arrow[r] \arrow[l] & \dots \arrow[l] \arrow[r]&  D\mid_{\paramSpace \times 0} \arrow[l]
\end{diagram}
applying $i_*$ and using the fact that it preserves weak equivalences between all objects, we thus obtain a zigzag of weak equivalences
\begin{diagram}
i_*C \arrow[r] & i_*D' \arrow[r] \arrow[l] & \dots \arrow[l] \arrow[r] &  i_*D\mid_{(\paramSpace, 0)} \arrow[l] \arrow[r] & D 
\end{diagram}
which induces an isomorphism in $\ho \mathcal M^{U}$. Finally, note that if $C$ is constant of value $M \in \mathcal M$, then so is $i_*C$.
\end{proof}
\section{Results on definable and Whitney Stratified spaces}
\subsection{Definable sets can be thickened}
The following lemma seems folklore knowledge to some degree. We provide it here for the sake of completeness. It seems to us that, with some extra technical effort, methods used in \cite{chazal2005weak} may even be used to obtain strongly stratified mapping cylinder neighborhoods. However, the following result suffices for our purposes.
\begin{lemma}\label{lem:appendix_definably_thickenable}
Let $\topSp \subset Y \subset \supSp$ be definable with respect to some o-minimal structure and $X$ compact. Then, there exists a $\varepsilon>0$ such for $0 < \alpha < \varepsilon$ the following holds: 
\begin{enumerate}
\item $\topSp \hookrightarrow \thicken[\alpha]{\topSp} \cap Y $ is a strong deformation retract.
\item \label{item:second_def_thick} There is a homeomorphism $(\thicken[\alpha]{\topSp} \cap Y) \setminus X \cong \mathrm{d}_{\topSp}^{-1}{(\frac{\alpha}{2})} \times (0, \alpha]$, such that the diagram
\[\begin{tikzcd}
(\thicken[\alpha]{\topSp} \cap Y) \setminus X \arrow[rd, "\mathrm{d}_{X}"'] \arrow[rr, "\sim"] && \mathrm{d}_{\topSp}^{-1}{(\frac{\alpha}{2})} \times (0, \alpha] \arrow[ld, "\pi_{(0,\alpha]}"]\\
& (0, \alpha]&
\end{tikzcd}
\]
commutes.
\end{enumerate}
 Furthermore, if $Y= \supSp$, then $\varepsilon$ may be taken to be the weak feature size of $\topSp$ as in \cite[Definition 3.1]{chazal2005weak}.
\end{lemma}
\begin{proof}
The statement on the homemomorphism type of the complements is an immediate application of Hardt's theorem for definable sets together with the fact that $\mathrm{d}_{X}$ is definable (see e.g. \cite{van1998tame}). One may then use the isotopies induced by flows used for example in \cite{chazal2005weak} to extend this homeomorphism to the case where $Y = \supSp$ and $\varepsilon$ is the weak feature size. To see that the latter is positive, note that the argument for positivity of weak feature sizes of semialgebraic sets in \cite[Remark 5.3]{fu1985tubular} also applies to the definable case. Finally, we need to see that the inclusion is a strong deformation retraction. Note that by the triangulability of definable sets (see for example \cite[Theorem 2.9]{van1998tame}), $\supSp$ may be equipped with a triangulation compatible with $X$ and $Y$. In particular, by subdividing if necessary, $X$ has arbitrarily small mapping cylinder neighborhoods in $Y$, given by piecewise linear regular neighborhoods. Furthermore, this means that $X \hookrightarrow X_{\alpha} \cap Y$ is a cofibration. Thus, it suffices to show that $X \hookrightarrow X_{\alpha} \cap Y$ is a homotopy equivalence.
Now, for $\alpha < \alpha' < \varepsilon$, with $\varepsilon$ such that \ref{item:second_def_thick} holds. Then, we have inclusions
\[
X \hookrightarrow X_{\alpha} \cap Y \hookrightarrow N \hookrightarrow X_{\alpha'} \cap Y ,
\]
where $N$ and $N'$ are regular neighborhoods with respect to the piecewise linear structure induced by the triangulation. By the open cylinder structure (assumption \labelcref{item:second_def_thick}) of the set $(X_{\alpha'} \cap Y) \setminus X$, the inclusion $X_{\alpha} \cap Y \hookrightarrow X_{\alpha'} \cap Y$ is a homotopy equivalence. The same holds for the inclusion $X \hookrightarrow N$. It follows by the two-out-of-six property of homotopy equivalences, that all maps are homotopy equivalences.
\end{proof}
\subsection{Proof of \cref{prop:equ_char_Whitney}}
\begin{proof}[Proof of \cref{prop:equ_char_Whitney}]\label{pf:equ_char_Whitney} The map
$\beta$ is clearly continuous on $\str_q \times \str_p$. The condition on $\funpqbeta$ is thus equivalent to the extension by $0$ to $\Delta_{\str_p}$ being continuous.
Indeed, by continuity of $\distVec{-}{-}$, this extension condition immediately implies condition (b). For the converse, as $\funpqbeta \geq 0$, it suffices to show upper semi-continuity. This is the content of \cref{prop:semicon_beta}.
\end{proof}
\begin{proposition}\label{prop:semicon_beta}
Let $\whs = (X, s: X \to \pos)$ be a Whitney stratified space. Then, the restriction of $\funpqbeta$ to $\whs_{\geq p} \times \whs_{p} \to \mathbb R$ is upper semi-continuous. 
\end{proposition}
\begin{proof}
$\funpqbeta$ is clearly continuous on the strata of $\whs \times \whs$. Now, suppose $(x_n, y_n) \in \whs_{\geq p} \times \whs_{p}$ is a sequence converging to a point $(x,y) \in \whs_{p'} \times \whs_{p}$, for some $p' \geq p$. Then, for sufficiently large $n \in \mathbb N$, we have $s(x_n) \geq p'$. To show upper semi-continuity, we may thus without loss of generality assume that $x_n$ lies in the same stratum $\whs_{q}$. We show that any subsequence of $(x_n,y_n)$ has a further subsequence (all named the same by abuse of notation), for which $\funpqbeta(x_n,y_n)$ converges to a value lesser or equal then $\funpqbeta(x,y)$. By compactness of Grassmannians, we may first restrict to a subsequence such that $\tangentSp[x_n]{\whs_q}$ and $\secant{x_n}{y_n}$ converge to $\tau$ and $l$ respectively.
By Whitney's condition (a) (\cite{whitney1965local}, \cite{whitney1965tangent}) -  which by \cite{mather1970notes} follows from condition (b) - we have  $\tangentSp[x]{\whs_{p'}} \subset \tau$.
Summarizing, this gives:
\begin{align*}
    \lim \funpqbeta(x_n,y_n) = \distVec{l}{\tau} &\leq \distVec{l}{\tangentSp[x]{\whs_{p'}}}.
\end{align*}
Now, in case when $x \neq y$, the last expression equals $\funpqbeta(x,y)$ by definition. In the case when $x=y$ then, by condition $(b)$, $l \subset \tau$. Thus, again, we have
\begin{align*}
\lim \funpqbeta(x_n,y_n) = \distVec{l}{\tau} = 0 = \funpqbeta(y,y)
\end{align*}
finishing the proof.
\end{proof}
\subsection{A normal bundle version of $\funpqbeta$}
Furthermore, we are going to make use of the following fiberwise version of $\funpqbeta$.
\begin{construction}\label{con:hatbeta}
Again, in the framework of \cref{con:pqbeta}, assume that $\whs = \stratPair$ is a Whitney stratified space, with $\whs_p$ compact. 
Take $N$ to be a standard tubular neighborhood of $\whs_p$ in $\supSp$ with retraction $r: N \to \whs_p$. 
Note that by Whitney's condition (a), for $N$ sufficiently small, $r_{\mid \whs_q}$ is a submersion for $q \geq p$. In particular, by \cite[Lemma 2.1]{nocera2021conically} the fiber of 
\[
 \whs^{y}:=(r)_{\mid N \cap \supLevel{\whs}{p}}^{-1}(y)
\]
is a Whitney stratified space over $\{q \in \pos \mid q \geq p \}$ with the $p$-stratum given by $\{y\}$. Furthermore, we have
\[ \tangentSp{\whs_q} \cap \normalSp{\whs_p} = \tangentSp{\whs^{r(x)}_q},\]
where $\normalSp{\whs_p}$ denotes the normal space of $\whs_p$ at $r(x)$.
In particular, the dimension of these spaces is constant, and they vary continuously in $x$.
Then, consider the following function:
\begin{align*}
\norpqbeta(-):  N \cap \supLevel{\whs}{p} \to  \mathbb R &&
    \begin{cases}
    x &\mapsto  \distVec{\secant{x}{r(x)}}  {\tangentSp{\whs^{r(x)}_{\phiStrat(x)}}} \textnormal{, for } s(x) >p\\
    x & \mapsto 0 \textnormal{, for } s(x) = p.
    \end{cases}
\end{align*}
Noting that $\secant{x}{r(x)} \in \normalSp{W_{p}}$, by an analogous argument to the proof of \cref{prop:equ_char_Whitney}, one obtains that $\norpqbeta(-)$ is continuous on $\whs_{q} \cup \whs_{p}$.
Note that if we restrict $\norpqbeta(-)$ to $\whs^{y}$, then we obtain the function $\funpqbeta(-,y)$ associated to $\whs^{y}$. Let us denote this $\funpqbeta_{y}$. In particular, by compactness of $\whs_{p}$, we obtain that the functions $\funpqbeta_{y}$ can be globally bounded by any $\delta >0$, for $N$ sufficiently small.
\end{construction}

\subsection{Definability of $\funpqbeta$}
\begin{proposition}
Let $\str = \stratPair$ be as in \cref{con:pqbeta}. Then, if $\topSp \subset \supSp$ is definable, then so is $\funpqbeta$.
\end{proposition}
\begin{proof}
As all the strata of $\topSp \times \topSp$ are again definable, it suffices to show that $\funpqbeta$ is definable on the strata of $\topSp \times \topSp$. Furthermore, as $\funpqbeta$ is $0$ along $\Delta_{\topSp}$, it suffices to show definability away from the diagonal. Here $\funpqbeta$ is equivalently given by
\[
\funpqbeta(x,y) = \inf_{v \in \tangentSp[x]{\topSp_{s(x)}}}{ \norm{\frac{x-y}{\norm{x-y}} -v }}.
\]
It follows from the fact that for $q \in P$, $\tangentSp[]{\topSp_q} \subset \supSp \times \supSp$ is definable (see \cite{coste2000introduction} and \cref{lem:definability}) that this defines a definable function $\topSp_q \times \topSp_p \to \mathbb R$.
\end{proof}

\subsection{Proof of \cref{prop:beta_globally_bounded}}
We begin by proving a series of technical lemmas.
\begin{lemma}\label{lem:definability}
Consider two definable maps $f: X \to \mathbb R$, $\pi: X \to Y$ such that $f$ is bounded from above on every fiber of $\pi$. Then the map
\begin{align*}
    g: Y &\to \mathbb R \\
     y &\mapsto \sup_{x \in \pi^{-1}(y)}f(x)
\end{align*}
is again definable. 
\end{lemma}
\begin{proof}
This is immediate, if one interprets the graph of $g$ in terms of a formula being expressible with respect to the o-minimal structure.
\end{proof}
\begin{lemma}\label{lem:continuity_of_sup}
Let $X \to \{ p < q\}$ be a stratified metric space and $Y$ a first countable, locally compact Hausdorff space. Let $\pi : X \to Y$ be a proper map, such that both the fibers of $\pi$, as well as the fibers of $\pi_{\mid X_{p}}$ vary continuously in the Hausdorff distance. Let $f:X \to \mathbb R$ be upper semi-continuous and continuous on the strata. Then,
\begin{align*}
    g: Y &\to \mathbb R \\
    y &\mapsto \sup_{x \in \pi^{-1}(y)}f(x)
\end{align*}
is continuous.
\end{lemma}
\begin{proof}
Note first that as the fibers of $\pi$ are compact and $f$ is upper semi continuous, it takes its maximum on every fiber. Now, let $y_n \to y$ be a convergent sequence in $Y$. We show that any of its subsequences $y'_n$, has a further subsequence $ \tilde  y_n \to y$, with \[ \sup_{x \in \pi^{-1}(\tilde y_n)}f(x) \to \sup_{x \in \pi^{-1}(y)}f(x).\]
Let $x'_n \in \pi^{-1}(y_n)$ for all $n$ such that $f(x'_n) = \sup_{x \in \pi^{-1}(y'_n)}f(x)$. As $Y$ is locally compact and $\pi$ is proper, $x'_n$ has a convergent subsequence $\tilde x_n \to \tilde x$. Define $\tilde y_n:=\pi(\tilde x_n)$. Since the fibers of $\pi$ vary continuously and $\tilde y_n \to y$, we also have $\tilde x \in \pi^{-1}(y)$. Thus, we have
\[
 \limsup \sup_{x \in \pi^{-1}(\tilde y_n)}f(x) =  \limsup f( \tilde x_n) \leq f(\tilde x) \leq g(y).
\]
It remains to see the converse inequality for a subsequence of $\tilde y_n$. Let $\hat x \in \pi^{-1}(y)$ be such that $f(\hat x) = \sup_{x \in \pi^{-1}(y)}f(x)$. By assumption we can find a sequence $x''_n$ with $x''_n \in \pi^{-1}(\tilde y_n)$ converging to $\hat x$.
If $\hat x \in X_p$, then $x''_n$ can be taken to be in $X_p$, as $\pi^{-1}(\tilde y_n) \cap X_p$ converges to $\pi^{-1}(y) \cap X_p$. If $\hat x \in X_q$, then, as the latter is open, $x''_n$ ultimately lies in $X_q$. Hence, by continuity of $f$ on the strata, we have
\[
g(y) = f(\hat x) = \lim f(x''_n) = \liminf{f(x''_n)} \leq \liminf \sup_{x \in \pi^{-1}(\tilde y_n)}f(x).
\]
\end{proof}
As a consequence of the prior two lemmas we obtain:
\begin{lemma}\label{lem:hatbeta_def}
If $\whs$ is a definably Whitney stratified over $\pos =\{ p <q\}$. Then the map 
 \begin{align*}
    \hat \beta: \whs_p \times \mathbb R_{\geq 0} &\to \mathbb R \\
     (y,d) &\mapsto \sup_{\norm{x-y}=d, x\in \whs}\funpqbeta(x,y) 
 \end{align*}
 is continuous in a neighborhood of $\whs_p \times \{0\}$, definable and vanishes on $\whs_p \times \{0\}.$
\end{lemma}
\begin{proof}
Definability follows immediately from \cref{lem:definability}.
 consider the map 
\begin{align*}
  B: \whs \times \whs_p \to \whs_p \times \mathbb R_{\geq 0} \\
    (x,y) \mapsto (y, \norm{x-y}).
\end{align*}
Over $\whs_p \times \mathbb R_{>0}$ it is given by submersion on each stratum of $\whs \times \whs_p$. In particular, by Thom's first isotopy lemma \cite[Proposition 11.1]{mather1970notes} it is a fiber bundle with fibers $\partial \ballEps[d]{y}$ at $(y,d)$ over $\mathbb R_{>0}$. 
In particular, the fibers of $B$ vary continuously over $\whs_p \times \mathbb R_{>0}$. Additionally, for $(y_n, d_n) \to (y,0)$ the fiber converges to the point $y$. Hence, $B$ fulfills the requirements of \cref{lem:continuity_of_sup}. Furthermore, $\funpqbeta: \whs \times \whs_p \to \mathbb R$ also fulfills the requirements of \cref{lem:continuity_of_sup}, showing the continuity of $\hat \beta$. Lastly, $\hat \beta$ vanishes on $\whs_p \times \mathbb R_{\leq 0}$ by definition of $\beta$.
\end{proof}
We now have all the tools available to obtain a proof of \cref{prop:beta_globally_bounded},
\begin{proof}[Proof of \cref{prop:beta_globally_bounded}]\label{proof:beta_globally_bounded}
We conduct this proof for the case of $\pos = \{ p < q \}$ and $K = \whs_p$ (with notation as in \cref{def:Loj_Whitney}). The general case follows analogously by working strata-wise and then passing to maxima. 
By \cref{lem:hatbeta_def} for $d$ small enough, the function $\hat \beta: \whs_p \times \mathbb R_{\geq 0} \to \mathbb R$ fulfills the requirements of Lojasiewicz' theorem for (polynomially bounded) o-minimal structures \cite{LojInequ}. 
Hence, we find $\hat \phi: \mathbb R_{\geq 0} \to \mathbb R_{\geq 0}$ to be a definable and monotonous bijection such that on $\whs_p \times [0,d]$ we have  
\[\hat \phi( \hat \beta(y,t) ) \leq t. 
\]
If the relevant o-minimal structure is polynomially bounded, then there exist $ n >0$, such that \[ t^{n} \leq \hat \phi(t) \] for $t \in [0,d']$. 
Hence, we obtain
\begin{align*}
   \hat \beta(y,t)^n  &\leq \hat \phi( \hat \beta(y,t) ) \leq t. \\
   \implies  \hat \beta(y,t) &\leq t^{\alpha}
\end{align*}
for $t \in [0,d]$, $\alpha= \frac{1}{n}$  and $d:=\phi^{-1}(d')$.
\end{proof}
\subsection{Proof of \cref{lem:strat_by_localcohom}}
\begin{proof}[Proof of \cref{lem:strat_by_localcohom}]\label{pf:lem:strat_by_locacohom}
The first result is immediate from the local conical structure of $\topSp$. The second is immediate from the definition of a homology stratification, as clearly $\topSp - \str_p$ is a homology manifold. 
For the final result, note first that by the local conical structure, having local homology isomorphic to $\textnormal{H}_{\bullet}(\mathbb R^q;0)$ is an open condition on $\str_p$. 
In particular, since this condition holds on all of $\topSp - \str_p$ it is an open condition on all of $\topSp$. 
Thus, $s: \topSp \to \{ p < q \}$ as defined in the statement is actually a stratification of $\topSp$.
To see that this is indeed a homology stratification we need to see that the local isomorphism condition is fulfilled.
By construction, we have $\topSp- \str_p \subset \tilde{\phiStrat}^{-1}\{q\}$. Within $\str_p \cap \tilde{\phiStrat}^{-1}\{p\}$ the local isomorphism condition again holds by the local conical structure of $\topSp$. Thus, it remains to consider the case where $x \in \str_p$, and $\locHlg{\topSp} \cong \textnormal{H}_{\bullet}(\mathbb R^q;0)$. 
We need to show that, for $U_x \cong \mathbb R^{q-p-1} \times \mathring C(L_{x})$, an open neighborhood of $x$,
the natural map \[ \locHlg{\topSp} \cong \hlg{\whs, \whs -U_x} \to \locHlgy{\topSp} \] is an isomorphism, for all $y \in U_x$. 
The only nontrivial degree in this case is $q=\dim{\whs}$. By an application of the K\"unneth formula $L_x$ is again an orientable manifold. 
Hence, up to suspension, from this perspective, the claim reduces to the fact that if $L_x$ is an orientable, closed manifold. Then, under the natural isomorphism \[\hlg{CL_x, L_x} \cong \hlgred[\bullet-1]{L_x}\]
the fundamental class of $L_x$ induces a fundamental class of $CL_x$.
\end{proof}

\end{document}